\documentclass{article}
\usepackage{amssymb}
\usepackage{amsmath, amsthm, amscd, amsfonts}
\usepackage{verbatim}
\usepackage{tensor}
\usepackage{tikz}
\usepackage{relsize}
\usepackage{float}
\usepackage{etoolbox}
\usepackage{bm}
\restylefloat{table}
\restylefloat{figure}
\usetikzlibrary{matrix,arrows,decorations.pathmorphing}
\usetikzlibrary{decorations.pathreplacing}

\newtheorem{theorem}{Theorem}
\newtheorem{lemma}[theorem]{Lemma}
\newtheorem{corollary}[theorem]{Corollary}
\newtheorem{definition}[theorem]{Definition}
\newtheorem{construction}[theorem]{Construction}
\newtheorem{proposition}[theorem]{Proposition}

\theoremstyle{definition}
\newtheorem{remark}[theorem]{Remark}

\usepackage{tikz}
\usepackage{subfig}
\usepackage{ifthen}
\usetikzlibrary{arrows}

\newcommand{\showpmcdots}[3]{
\foreach \x in {1,...,#3} {
	\filldraw (#1,\x) circle (0.1);
	\filldraw (#2,\x) circle (0.1);
}}

\newenvironment{sdn}[2][7pt]
{\begin{tikzpicture} [x=#1,y=#1,baseline=(current bounding box.center)]
\showpmcdots{0}{3}{#2}}
{\end{tikzpicture}}

\newenvironment{sdbig2}[1][7pt]
{\begin{tikzpicture} [x=#1,y=#1,baseline=(current bounding box.center)]
\showpmcdots{0}{3}{3}
\filldraw[color=white] (0,2) circle (0.2);
\filldraw[color=white] (3,2) circle (0.2);}
{\end{tikzpicture}}

\newenvironment{sd2sep}[1][7pt]
{\begin{tikzpicture} [x=#1,y=#1,baseline=(current bounding box.center)]
\showpmcdots{0}{3}{5}
\filldraw[color=white] (0,3) circle (0.2);
\filldraw[color=white] (3,3) circle (0.2);}
{\end{tikzpicture}}

{\begin{tikzpicture} [x=#1,y=#1,baseline=(current bounding box.center)]
\showpmcdots{0}{3}{6}
\filldraw[color=white] (0,3) circle (0.2);
\filldraw[color=white] (3,3) circle (0.2);
\filldraw[color=white] (0,4) circle (0.2);
\filldraw[color=white] (3,4) circle (0.2);}
{\end{tikzpicture}}

{\begin{tikzpicture} [x=#1,y=#1,baseline=(current bounding box.center)]
\showpmcdots{0}{3}{8}
\filldraw[color=white] (0,4) circle (0.2);
\filldraw[color=white] (3,4) circle (0.2);
\filldraw[color=white] (0,5) circle (0.2);
\filldraw[color=white] (3,5) circle (0.2);
\filldraw[color=white] (0,6) circle (0.2);
\filldraw[color=white] (3,6) circle (0.2);}
{\end{tikzpicture}}

\newenvironment{sd6m24}[1][7pt]
{\begin{tikzpicture} [x=#1,y=#1,baseline=(current bounding box.center)]
\showpmcdots{0}{3}{6}
\filldraw[color=white] (0,2) circle (0.2);\filldraw[color=white] (3,2) circle (0.2);
\filldraw[color=white] (0,4) circle (0.2);\filldraw[color=white] (3,4) circle (0.2);}
{\end{tikzpicture}}

\newenvironment{sd6m46}[1][7pt]
{\begin{tikzpicture} [x=#1,y=#1,baseline=(current bounding box.center)]
\showpmcdots{0}{3}{6}
\filldraw[color=white] (0,4) circle (0.2);\filldraw[color=white] (3,4) circle (0.2);
\filldraw[color=white] (0,6) circle (0.2);\filldraw[color=white] (3,6) circle (0.2);}
{\end{tikzpicture}}

\newcommand{\strandup}[2]{\draw [->] (0,#1) to [out=0,in=180] (3,#2);}
\newcommand{\stranddown}[2]{\draw [->] (3,#2) to [out=180,in=0] (0,#1);}
\newcommand{\doublehor}[2]{\draw [dashed] (0,#1) to (3,#1); \draw [dashed] (0,#2) to (3,#2);}
\newcommand{\singlehor}[1]{\draw [dashed] (0,#1) to (3,#1);}

\newenvironment{dsdn}[2][7pt]
{\begin{tikzpicture} [x=#1,y=#1,baseline=(current bounding box.center)]
\showpmcdots{0}{6}{#2}
\draw (3,0) to (3,#2) to +(0,1);}
{\end{tikzpicture}}

\newenvironment{dsdmr2}[1][7pt]
{\begin{tikzpicture} [x=#1,y=#1,baseline=(current bounding box.center)]
\showpmcdots{0}{6}{3}
\filldraw[color=white] (6,2) circle (0.2);
\draw (3,0) to (3,4);}
{\end{tikzpicture}}

\newenvironment{dsdml2}[1][7pt]
{\begin{tikzpicture} [x=#1,y=#1,baseline=(current bounding box.center)]
\showpmcdots{0}{6}{3}
\filldraw[color=white] (0,2) circle (0.2);
\draw (3,0) to (3,4);}
{\end{tikzpicture}}

\newenvironment{dsd2sep}[1][7pt]
{\begin{tikzpicture} [x=#1,y=#1,baseline=(current bounding box.center)]
\showpmcdots{0}{6}{5}
\filldraw[color=white] (0,3) circle (0.2);\filldraw[color=white] (6,3) circle (0.2);
\draw (3,0) to (3,6);}
{\end{tikzpicture}}

\newenvironment{dsdlrise}[1][7pt]
{\begin{tikzpicture} [x=#1,y=#1,baseline=(current bounding box.center)]
\showpmcdots{0}{6}{6}
\filldraw[color=white] (0,4) circle (0.2);\filldraw[color=white] (6,4) circle (0.2);
\filldraw[color=white] (0,2) circle (0.2);\filldraw[color=white] (6,6) circle (0.2);
\draw (3,0) to (3,7);}
{\end{tikzpicture}}

\newenvironment{dsdrrise}[1][7pt]
{\begin{tikzpicture} [x=#1,y=#1,baseline=(current bounding box.center)]
\showpmcdots{0}{6}{6}
\filldraw[color=white] (0,4) circle (0.2);\filldraw[color=white] (6,4) circle (0.2);
\filldraw[color=white] (0,6) circle (0.2);\filldraw[color=white] (6,2) circle (0.2);
\draw (3,0) to (3,7);}
{\end{tikzpicture}}

\newenvironment{subsd}[2]
{\begin{scope} [shift={(#1,#2)}]
\showpmcdots{3}{3}{8}
\draw (3,0) to (3,9);\draw[white](0,0) to (0,9);}
{\end{scope}}

\newenvironment{subsdn}[3]
{\begin{scope} [shift={(#1,#2)}]
\showpmcdots{0}{6}{#3}
\draw (3,0) to (3,#3) to +(0,1);}
{\end{scope}}

\newenvironment{subssdn}[3]
{\begin{scope} [shift={(#1,#2)}]
\showpmcdots{0}{3}{#3}}
{\end{scope}}

\newenvironment{subsdbig2}[2]
{\begin{scope} [shift={(#1,#2)}]
\showpmcdots{0}{3}{3}
\filldraw[color=white] (0,2) circle (0.2);
\filldraw[color=white] (3,2) circle (0.2);}
{\end{scope}}

\newenvironment{subsd6m24}[2]
{\begin{scope} [shift={(#1,#2)}]
\showpmcdots{0}{3}{6}
\filldraw[color=white] (0,2) circle (0.2);\filldraw[color=white] (3,2) circle (0.2);
\filldraw[color=white] (0,4) circle (0.2);\filldraw[color=white] (3,4) circle (0.2);}
{\end{scope}}

\newenvironment{subsd6m46}[2]
{\begin{scope} [shift={(#1,#2)}]
\showpmcdots{0}{3}{6}
\filldraw[color=white] (0,4) circle (0.2);\filldraw[color=white] (3,4) circle (0.2);
\filldraw[color=white] (0,6) circle (0.2);\filldraw[color=white] (3,6) circle (0.2);}
{\end{scope}}

\newenvironment{subdsdml2}[2]
{\begin{scope} [shift={(#1,#2)}]
\showpmcdots{0}{6}{3}
\filldraw[color=white] (0,2) circle (0.2);
\draw (3,0) to (3,4);}
{\end{scope}}

\newenvironment{subdsdmr2}[2]
{\begin{scope} [shift={(#1,#2)}]
\showpmcdots{0}{6}{3}
\filldraw[color=white] (6,2) circle (0.2);
\draw (3,0) to (3,4);}
{\end{scope}}

\newenvironment{subdsdrrise}[2]
{\begin{scope} [shift={(#1,#2)}]
\showpmcdots{0}{6}{6}
\filldraw[color=white] (0,4) circle (0.2);\filldraw[color=white] (6,4) circle (0.2);
\filldraw[color=white] (0,6) circle (0.2);\filldraw[color=white] (6,2) circle (0.2);
\draw (3,0) to (3,7);}
{\end{scope}}

\newenvironment{subdsdlrise}[2]
{\begin{scope} [shift={(#1,#2)}]
\showpmcdots{0}{6}{6}
\filldraw[color=white] (0,4) circle (0.2);\filldraw[color=white] (6,4) circle (0.2);
\filldraw[color=white] (0,2) circle (0.2);\filldraw[color=white] (6,6) circle (0.2);
\draw (3,0) to (3,7);}
{\end{scope}}

\newcommand{\rstrandup}[2]{\draw [->] (3,#1) to [out=0,in=180] (6,#2);}
\newcommand{\lstrandup}[2]{\draw [->] (0,#1) to [out=0,in=180] (3,#2);}
\newcommand{\rstranddown}[2]{\draw [->] (6,#2) to [out=180,in=0] (3,#1);}
\newcommand{\ldoublehor}[2]{\draw [dashed] (0,#1) to (3,#1);\draw [dashed] (0,#2) to (3,#2);}
\newcommand{\rdoublehor}[2]{\draw [dashed] (3,#1) to (6,#1);\draw [dashed] (3,#2) to (6,#2);}
\newcommand{\lsinglehor}[1]{\draw [dashed] (0,#1) to (3,#1);}
\newcommand{\rsinglehor}[1]{\draw [dashed] (3,#1) to (6,#1);}

\newcommand{\keypair}[3]{
\ifthenelse{\equal{#3}{left}}{\draw [dashed] (-0.5,#1) .. controls (-1.5,#1) and (-1.5,#2) .. (-0.5,#2);}{}
\ifthenelse{\equal{#3}{right}}{\draw [dashed] (6.5,#1) .. controls (7.5,#1) and (7.5,#2) .. (6.5,#2);}{}
}

\newcommand{\setarrows}[4]{
	\def\upperarrow{#1}
	\def\leftarrow{#2}
	\def\rightarrow{#3}
	\def\lowerarrow{#4}
}
\newcommand{\resetarrows}{\setarrows{$H$}{$d$}{$d$}{$H$}}

\resetarrows

\newcommand{\canceldiag}[6]{
\begin{tikzpicture} [x=8pt,y=8pt]
\begin{subsd}{0}{14}\ifthenelse{\not\equal{#1}{#2}}{\keypair{#1}{#2}{left}}{}#3\end{subsd}
\begin{subsd}{12}{14}\ifthenelse{\not\equal{#1}{#2}}{\keypair{#1}{#2}{right}}{}#4\end{subsd}
\begin{subsd}{0}{0}\ifthenelse{\not\equal{#1}{#2}}{\keypair{#1}{#2}{left}}{}#5\end{subsd}
\begin{subsd}{12}{0}\ifthenelse{\not\equal{#1}{#2}}{\keypair{#1}{#2}{right}}{}#6\end{subsd}
\draw [->] (7,18.5) to node [above] {\upperarrow} (11,18.5);
\draw [->] (7,4.5) to node [above] {\lowerarrow} (11,4.5);
\draw [->] (3,13) to node [right] {\leftarrow} (3,10);
\draw [->] (15,13) to node [right] {\rightarrow} (15,10);
\end{tikzpicture}
}

\newcommand{\mapdiag}[5]{
\begin{tikzpicture} [x=10pt,y=10pt,baseline=(current bounding box.center)]
\begin{subsd}{0}{0}\ifthenelse{\not\equal{#1}{#2}}{\keypair{#1}{#2}{left}}{}#4\end{subsd}
\begin{subsd}{12}{0}\ifthenelse{\not\equal{#1}{#2}}{\keypair{#1}{#2}{right}}{}#5\end{subsd}
\draw [->] (7,4.5) to node [above] {#3} (11,4.5);
\end{tikzpicture}
}

\newcommand{\doublemapdiag}[7]{
\begin{tikzpicture} [x=8pt,y=8pt]
\begin{subsd}{0}{0}\ifthenelse{\not\equal{#1}{#2}}{\keypair{#1}{#2}{left}}{}#5\end{subsd}
\begin{subsd}{12}{0}#6\end{subsd}
\begin{subsd}{24}{0}\ifthenelse{\not\equal{#1}{#2}}{\keypair{#1}{#2}{right}}{}#7\end{subsd}
\draw [->] (7,4.5) to node [above] {#3} (11,4.5);
\draw [->] (19,4.5) to node [above] {#4} (23,4.5);
\end{tikzpicture}
}

\newcommand{\iz}{\mathbb{I}_\mathcal{Z}}
\newcommand{\slz}{\mathcal{Z}}
\newcommand{\sla}{\mathcal{A}}
\newcommand{\hg}{\mathbf{H}^g}
\newcommand{\cob}{\widehat{\mathrm{Cob}}}
\newcommand{\overto}[1]{\overset{#1}{\to}}

\newcommand{\hataa}{\widehat{\mathcal{AA}}}

\newcommand{\hatda}{\widehat{\mathcal{DA}}}
\newcommand{\hatdd}{\widehat{\mathcal{DD}}}
\newcommand{\widehatit}[1]{\widehat{\mathit{#1}}}
\renewcommand{\AA}{\mathit{AA}}
\newcommand{\DA}{\mathit{DA}}
\newcommand{\AD}{\mathit{DD}}
\newcommand{\DD}{\mathit{DD}}
\newcommand{\opp}{\mathrm{opp}}

\newcommand{\midarrow}{\tikz \draw[-triangle 60] (0,0) -- +(.1,0);}

\newenvironment{threebase}{
\begin{tikzpicture} [x=15,y=15,baseline=(current bounding box.center)]
\draw [line width=2pt] (-1.2,-0.2) -- (-0.8,0.2);
\draw [line width=2pt] (0.8,0.2) -- (1.2,-0.2);
\draw [line width=2pt] (-0.3,-1.6) -- (0.3,-1.6);}
{\end{tikzpicture}}
\newenvironment{fourbase}{
\begin{tikzpicture} [x=15,y=15,baseline=(current bounding box.center)]
\draw [line width=2pt] (-0.3,1) -- (0.3,1);
\draw [line width=2pt] (-0.3,-1) -- (0.3,-1);
\draw [line width=2pt] (-1,-0.3) -- (-1,0.3);
\draw [line width=2pt] (1,-0.3) -- (1,0.3);}
{\end{tikzpicture}}
\newenvironment{dsd3}{\begin{dsdn}{3}}{\end{dsdn}}

\allowdisplaybreaks[1]

\newcounter{daeq}
\renewcommand*{\thedaeq}{DA\arabic{daeq}}

\makeatletter
\def\@equationname{equation}
\newenvironment{daequation}[1]{
  \def\mymathenvironmenttouse{#1}
  \ifx\mymathenvironmenttouse\@equationname
  \refstepcounter{daeq}
  \else
  \patchcmd{\@arrayparboxrestore}{equation}{daeq}{}{}
  \patchcmd{\print@eqnum}{equation}{daeq}{}{}
  \patchcmd{\incr@eqnum}{equation}{daeq}{}{}
  \fi
  \csname\mymathenvironmenttouse\endcsname%
}{
  \ifx\mymathenvironmenttouse\@equationname%
  \tag{\thedaeq}
  \fi
  \csname end\mymathenvironmenttouse\endcsname%
}
\makeatother

\newcounter{ddeq}
\renewcommand*{\theddeq}{DD\arabic{ddeq}}

\makeatletter
\def\@equationname{equation}
\newenvironment{ddequation}[1]{
  \def\mymathenvironmenttouse{#1}
  \ifx\mymathenvironmenttouse\@equationname
  \refstepcounter{ddeq}
  \else
  \patchcmd{\@arrayparboxrestore}{equation}{ddeq}{}{}
  \patchcmd{\print@eqnum}{equation}{ddeq}{}{}
  \patchcmd{\incr@eqnum}{equation}{ddeq}{}{}
  \fi
  \csname\mymathenvironmenttouse\endcsname%
}{
  \ifx\mymathenvironmenttouse\@equationname%
  \tag{\theddeq}
  \fi
  \csname end\mymathenvironmenttouse\endcsname%
}
\makeatother

\newcounter{aaeq}
\renewcommand*{\theaaeq}{AA\arabic{aaeq}}

\makeatletter
\def\@equationname{equation}
\newenvironment{aaequation}[1]{
  \def\mymathenvironmenttouse{#1}
  \ifx\mymathenvironmenttouse\@equationname
  \refstepcounter{aaeq}
  \else
  \patchcmd{\@arrayparboxrestore}{equation}{aaeq}{}{}
  \patchcmd{\print@eqnum}{equation}{aaeq}{}{}
  \patchcmd{\incr@eqnum}{equation}{aaeq}{}{}
  \fi
  \csname\mymathenvironmenttouse\endcsname%
}{
  \ifx\mymathenvironmenttouse\@equationname%
  \tag{\theaaeq}
  \fi
  \csname end\mymathenvironmenttouse\endcsname%
}
\makeatother

\begin{document}

\title{Combinatorial Proofs in Bordered Heegaard Floer homology}
\author{Bohua Zhan}
\date{January 24, 2016}
\maketitle
\begin{abstract}
  Using bordered Floer theory, we give a combinatorial construction and proof of
  invariance for the hat version of Heegaard Floer homology. As part of the
  proof, we also establish combinatorially the invariance of the
  linear-categorical representation of the strongly-based mapping class groupoid
  given by the same theory.
\end{abstract}

\section{Introduction}\label{sec:intro}

Heegaard Floer homology, introduced by Ozsv\'ath and Szab\'o in \cite{OS04a} and
\cite{OS04b}, gives several kinds of invariants for closed 3-manifolds. The
invariants are defined using holomorphic curves, so in general they are not
directly computable from the definitions. However, for the hat version of the
invariant (denoted $\widehatit{HF}$), there are ways to give combinatorial
definitions. There are two steps in this process. First, we want to give, to a
particular kind of description of a 3-manifold (such as a Heegaard splitting), a
description of $\widehatit{HF}$ associated to that 3-manifold. This means that,
at least in principle, the invariant can be computed algorithmically for any
3-manifold. Second, we want to give combinatorial proofs for the main properties
of $\widehatit{HF}$, beginning with the statement that it depends only on the
diffeomorphism class of the 3-manifold, rather than on a particular description
of it.

Bordered Floer theory gives a way to extend the hat version of Heegaard Floer
homology to 3-manifolds with one or two boundary components. The theory is also
defined using holomorphic curves. However, some of the invariants associated to
certain simple types of 3-manifolds with boundary have been computed. By
breaking an arbitrary closed 3-manifold into simpler pieces, the theory gives a
combinatorial description of $\widehatit{HF}$ \cite{LOT10c}, achieving the first
step in the process described above.

In this paper, we give the second step of the process, namely prove
combinatorially that the construction of $\widehatit{HF}$ given by bordered
Floer theory in fact produces an invariant of the 3-manifold. One main result we
use is an alternative description of $\widehatit{CFAA}(\iz)$ given in
\cite{BZ1}. This allows us to use a combinatorial construction that is easier to
reason about.

An intermediate statement in the proof, which may be of independent interest, is
that bordered Floer theory gives a linear-categorical representation of the
strongly-based mapping class groupoid (which contains the strongly-based mapping
class group). By a linear-categorical representation of a group or groupoid, we
mean assigning homotopy equivalence classes of bimodules to each element of the
group (resp. groupoid), in such a way that composition in the group
(resp. groupoid) corresponds to taking an appropriate tensor product of
bimodules.

Sarkar and Wang gave in \cite{SW} the first combinatorial description of
$\widehatit{HF}$, by giving a systematic way to convert any Heegaard diagram
into a nice diagram, in which case counting holomorphic curves is
combinatorial. Ozsv\'ath, Stipsicz, and Szab\'o gave in \cite{OSS} the first
combinatorial proof of invariance for $\widehatit{HF}$, using another way to
convert general Heegaard diagrams into \emph{convenient} diagrams -- a more
restricted kind of nice diagrams, and by studying Heegaard moves on convenient
diagrams.

Linear-categorical representations of important groups in topology have also
been investigated before. Bordered Floer theory actually gives a family of
representations of the strongly-based mapping class groupoid. For a given genus
$g$, the representations are indexed by an integer $w$, called the weight,
between $-g$ and $g$. The representation that is relevant for 3-manifold
invariants, and which we will focus on in this paper, corresponds to $w=0$. The
cases $w=\pm g$ are trivial. The cases $w=\pm (g-1)$ are described
combinatorially by Lipshitz, Ozsv\'ath, and Thurston in \cite{LOT10e}, and a
combinatorial proof of invariance is given by Siegel in
\cite{Siegel11}. Linear-categorical representations of other groups occuring in
topology have also been studied. See the introduction in \cite{KT07} for a
review and a list of references. One major example is linear-categorical
representations of the braid group, studied in, for example, \cite{KS},
\cite{ST01}, \cite{CK08}, \cite{SS06}, and \cite{K02}.

We now give an overview of this paper. In Section \ref{sec:combconstr}, we
review the structure of bordered Floer theory, and describe the combinatorial
construction of it considered here. In Section \ref{sec:computationda}, we prove
some preliminary results on type $\DA$ bimodules and our construction of the
type $\DA$ invariants. Using these results, we prove in Section \ref{sec:relmcg}
the intermediate statement on the linear-categorical representation of the
strongly-based mapping class groupoid. Finally, we complete the proof of
invariance for closed 3-manifolds in Section \ref{sec:3manifold}.

\subsection{Acknowledgements}
I would like to thank Peter Ozsv\'ath for offering the ideas which led to this
paper, and to Zolt\'an Szab\'o and Robert Lipshitz for many suggestions. I also
want to thank the Simons Center for Geometry and Physics for their hospitality
when part of this work is done. Finally, I thank the referee for many helpful
comments.

\section{Overview of the construction}\label{sec:combconstr}

In the first part of this section, we briefly review the structure of bordered
Floer theory, as is defined analytically in \cite{LOT08} and \cite{LOT10a}. In
the second part, we describe some of the existing combinatorial constructions
given in \cite{LOT10c}, and then the construction that will be studied in this
paper.

\subsection{Pointed matched circles and strand algebras}
\label{sec:intropmc}

In bordered Floer theory, the connected, compact, orientable surfaces that serve
as boundary components of 3-manifolds are specified using pointed matched
circles. A pointed matched circle is a quadruple $\slz=(Z,z,\mathbf{a},M)$,
consisting of a circle $Z$, a point $z\in Z$, a set of $4k$ points
$\mathbf{a}\subset Z\setminus \{z\}$, and a two-to-one map $M$ from $\mathbf{a}$
to $\{1,2,\dots,2k\}$, pairing the points in $\mathbf{a}$, that satisfies the
following condition: if we thicken the circle $Z$ to an annulus $Z\times [0,1]$
and attach a 1-handle to the outside boundary $Z\times \{1\}$ of the annulus
joining each pair of points in $\mathbf{a}$, then the new outside boundary must
be a single circle. Given this requirement, we may glue a disk onto that
boundary, obtaining a genus $k$ surface $F^\circ(\slz)$ with one boundary
component $Z\times \{0\}$ and a basepoint $z$ on the boundary. We say that the
pointed matched circle $\slz$ \emph{parametrizes} $F^\circ(\slz)$. Let $-\slz$
be the pointed matched circle obtained by reversing orientation on $\slz$. Then
$F^\circ(-\slz)$ is the orientation reversal of $F^\circ(\slz)$.

Let $F(\slz)$ be the result of filling the boundary of $F^\circ(\slz)$ with a
disk. Then $F(\slz)$ is a closed surface of genus $k$, marked with a
homotopically trivial circle $Z$ and a basepoint $z\in Z$. We will also say
$F(\slz)$ is parametrized by $\slz$.

An example of a pointed matched circle for $k=2$ is shown in Figure
\ref{fig:linearpmc}.

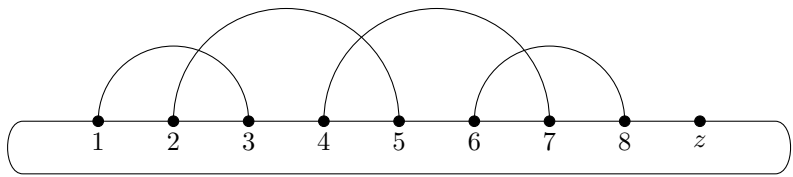
\begin{figure}[h!tb] \centering
  \begin{tikzpicture}
    \draw [-] (0,0) to (10,0) to [out=0,in=0] (10,-0.7) to
    (0,-0.7) to [out=180,in=180] (0,0);
    \foreach \x in {1,2,...,9} {
      \filldraw (\x,0) circle (2pt);
      \ifnum \x<9
      \draw (\x,0) node[below=1pt] {$\x$};
      \else
      \draw (\x,0) node[below=2pt] {$z$};
      \fi
    }
    \draw (3,0) arc (0:180:1);
    \draw (5,0) arc (0:180:1.5);
    \draw (7,0) arc (0:180:1.5);
    \draw (8,0) arc (0:180:1);
  \end{tikzpicture}
  \caption{Linear pointed matched circle with $k=2$.}
  \label{fig:linearpmc}
\end{figure}

To each pointed matched circle $\slz$, bordered Floer theory associates a
combinatorially defined dg-algebra $\sla(\slz)$. We refer to the original papers
for the description of $\sla(\slz)$. Here we just fix some notations and
terminologies used in this paper. For any generator $a\in\sla(\slz)$, the
\emph{multiplicity} of $a$, denoted $\mathrm{mult}(a)$, is an element in
$H_1(Z\setminus\{z\},\mathbf{a})$, recording how many times the strands in $a$
cover each non-basepoint interval on $Z$. The \emph{length} of $a$ is the sum of
coefficients in $\mathrm{mult}(a)$. Equivalently, it is the sum of lengths of
strands in $a$. It is clear from the definitions that the algebra $\sla(-\slz)$
is the opposite algebra of $\sla(\slz)$. In particular, there is a canonical
identification of their generators. For any $a\in\sla(\slz)$, let $\overline{a}$
denote the corresponding element in $\sla(-\slz)$. If $i\in\sla(\slz)$ is an
idempotent, let $o(i)\in\sla(\slz)$ denote the idempotent complementary to
$i$. A \emph{chord} is a single strand on $Z$. For any given chord $\xi$, we
define $a(\xi)\in\sla(\slz)$ to be the sum of all generators that result from
adding horizontal strands to $\xi$.

Given a 3-manifold $Y$ with one boundary component $\partial Y$, a
parametrization of $\partial Y$ by a pointed matched circle $\slz = (Z, z,
\mathbf{a}, M)$ is a diffeomorphism $\phi:F(\slz)\to\partial Y$. This marks
$\partial Y$ with a circle and a basepoint on the circle, which by abuse of
notation we will also call $Z$ and $z$. Bordered Floer theory associates two
invariants to a 3-manifold $Y$ with boundary $\partial Y$ parametrized by
$\slz$:
\begin{itemize}
\item Type $A$ invariant $\widehatit{CFA}(Y)_{\sla(\slz)}$, a right
  $A_\infty$-module over $\sla(\slz)$.
\item Type $D$ invariant $^{\sla(-\slz)}\widehatit{CFD}(Y)$, a left type $D$
  module over $\sla(-\slz)$.
\end{itemize}

They are invariants of $Y$ up to homotopy equivalence of $A_\infty$-modules or
type $D$ modules. We use the following standard convention in expressing types
of actions on the module: each algebra is written on the side it acts on,
subscripts indicate $A_\infty$-actions, and superscripts indicate type $D$
actions. These may be omitted when there is no danger of confusion.

These invariants satisfy the following pairing theorem: let $Y_1$ and $Y_2$ be
two 3-manifolds with boundaries parametrized by $\slz$ and $-\slz$,
respectively. Let $Y=Y_1\cup_\partial Y_2$ be the closed 3-manifold obtained by
gluing them along their boundaries (with the gluing map induced by the
parametrizations). Then the chain complex $\widehatit{CF}(Y)$ (whose homology is
$\widehatit{HF}(Y)$), is given by

\begin{equation} \label{eq:pairingtheorem} \widehatit{CF}(Y) \simeq
  \widehatit{CFA}(Y_1)_{\sla(\slz)} \boxtimes~^{\sla(\slz)}\widehatit{CFD}(Y_2),
\end{equation}
(\cite[Theorem 1.3]{LOT08}).

The theory extends to 3-manifolds with two boundary components as follows. Given
$Y$ with two boundary components $\partial_L Y$ and $\partial_R Y$. Fix
parametrizations $\phi_1:F(\slz_1)\to \partial_L Y$ and
$\phi_2:F(\slz_2)\to \partial_R Y$. This induces circles $Z_1$ and $Z_2$ on
$\partial_L Y$ and $\partial_R Y$, and basepoints $z_1\in Z_1,z_2\in Z_2$. We
further fix a map $\gamma$ from the framed cylinder $(S^1,z)\times [0,1]$ into
$Y$, so that $(S^1,z)\times\{0\}$ and $(S^1,z)\times\{1\}$ map to $(Z_1,z_1)$
and $(Z_2,z_2)$, respectively. We call the totality of the data $(Y,\partial_L
Y,\partial_R Y,\phi_1,\phi_2,\gamma)$ a \emph{strongly bordered 3-manifold with
  two boundary components}. From now on whenever we mention a 3-manifold $Y$
with two boundary components, we mean a strongly bordered 3-manifold, omitting
the other data when they are clear from context. To a 3-manifold $Y$ with two
boundary components, bordered Floer theory associates the following invariants:
\begin{itemize}
\item Type $\AA$ invariant $\widehatit{CFAA}(Y)_{\sla(\slz_1),\sla(\slz_2)}$, a
  right $A_\infty$-bimodule over $\sla(\slz_1)$ and $\sla(\slz_2)$.
\item Type $\DD$ invariant $^{\sla(-\slz_1),\sla(-\slz_2)}\widehatit{CFDD}(Y)$, a
  left type $D$ bimodule over $\sla(-\slz_1)$ and $\sla(-\slz_2)$.
\item Type $\DA$ invariant $^{\sla(-\slz_1)}\widehatit{CFDA}(Y)_{\sla(\slz_2)}$, a
  left type $D$, right $A_\infty$-bimodule over $\sla(-\slz_1)$ and
  $\sla(\slz_2)$.
\item Type $\AD$ invariant $^{\sla(-\slz_2)}\widehatit{CFAD}(Y)_{\sla(\slz_1)}$, a
  right $A_\infty$, left type $D$ bimodule over $\sla(\slz_1)$ and
  $\sla(-\slz_2)$.
\end{itemize}

These bimodules satisfy similar pairing theorems, as described in \cite[Section
7.1]{LOT10a}. The general rule is that box tensor product can be taken between a
right $A_\infty$-action and a left type $D$ action over the same algebra
$\sla(\slz)$. Taking this box tensor product corresponds to gluing two
boundaries parametrized by $\slz$ and $-\slz$.

Following the convention in \cite{BZ1}, we will write actions on the various
kinds of modules and bimodules as sums of \emph{arrows}. For example, if the
coefficient of $\mathbf{y}$ is 1 in
$m_{1,i,j}(\mathbf{x};a_1,\dots,a_i;b_1,\dots,b_j)$, where each $a_k, 1\le k\le
i$ and $b_l, 1\le l\le j$ is a generator of the appropriate algebra, we say
there is an arrow $m_{1,i,j}: (\mathbf{x};a_1,\dots,a_i;b_1,\dots,b_j)\to
\mathbf{y}$. Likewise, an arrow in the type $\DA$ action is of the form
$\delta^1_{1+i}: (\mathbf{x},a_1,\dots,a_i)\to b\otimes\mathbf{y}$, and an arrow
in the type $\DD$ action is of the form $\delta^1: \mathbf{x}\to a\otimes
b\otimes\mathbf{y}$.

We will also need the concept of \emph{dual} on bimodules. This is called
opposite structures in \cite[Definition 2.2.31, 2.2.53]{LOT10a}. For a left type
$\DD$ bimodule $^{A,B}M$, its dual $\overline{M}^{A,B}$ is the type $\DD$
bimodule over the same generators, where each arrow $\delta^1_M: \mathbf{x}\to
a_1\otimes a_2\otimes\mathbf{y}$ in the type $\DD$ action of $^{A,B}M$
corresponds to an arrow $\delta^1_{\overline{M}}: \mathbf{y}\to a_1\otimes
a_2\otimes\mathbf{x}$ in the type $\DD$ action of $\overline{M}^{A,B}$. So the
left actions by $A$ and $B$ become right actions, or equivalently left actions
by $A^{\opp}$ and $B^{\opp}$. So we will also write the dual as
$^{A^{\opp},B^{\opp}}\overline{M}$. Similarly, we can define duals on type $\DA$
and type $\AA$ bimodules. The dual commutes with box tensor product. That is:
\[ \overline{M_A \boxtimes \tensor[^A]{N}{}} = \overline{N}^A \boxtimes
\tensor[_A]{\overline{M}}{} = \overline{M}_{A^{\opp}} \boxtimes
\tensor[^{A^{\opp}}]{\overline{N}}{}, \] where $M$ and $N$ may have additional
actions.

\subsection{Gradings on bordered invariants}

In this section, we give a brief overview of gradings on the bordered
invariants. For details, see \cite[Chapter 10]{LOT08} and \cite[Section
6.5]{LOT10a}.

We begin with gradings on the dg-algebra $\sla(\slz)$. There are two kinds of
gradings, by a larger group $G'(\slz)$ and a ``refined'' grading by a smaller
group $G(\slz)$. Both $G(\slz)$ and $G'(\slz)$ are non-commutative, equipped
with a distinguished central element $\lambda$.

An element of $G'(\slz)$ is specified by a pair $(k,\alpha)$, where
$k\in\frac{1}{2}\mathbb{Z}$ and $\alpha\in H_1(Z', \mathbf{a})$. With points of
$\mathbf{a}$ labeled $1, \dots, 4k$, we can write $\alpha$ as a sequence of
integers $\alpha_i, 1\le i\le 4k-1$, where $\alpha_i$ is the multiplicity of
$\alpha$ at the interval $[i, i+1]$. Then multiplication on $G'(\slz)$ is
defined by
\[ (k, \alpha)\cdot (l, \beta) = (k + l + L(\alpha, \beta), \alpha + \beta), \]
where
\[ L(\alpha, \beta) = \sum_{i=1}^{4k-2} \frac{1}{2}(\alpha_i \beta_{i+1} -
\alpha_{i+1}\beta_i). \]

Actually, the grading lies in an index 2 subgroup of $G'(\slz)$, but we will not
make use of this here.

For later use, we define an anti-homomorphism
\[ R : G'(\slz) \to G'(-\slz) \]
given by
\[ R(k, \alpha_1, \dots, \alpha_{4k-1}) = (k, -\alpha_{4k-1}, \dots,
-\alpha_1). \]

To define the grading of a generator of $\sla(\slz)$, we first define a map
\[ m : H_1(Z', \mathbf{a}) \times H_0(\mathbf{a}) \to \frac{1}{2}\mathbb{Z}. \]
For an interval $\alpha$ (with orientation from $Z$) and a point $p$, let
$m(\alpha, p)=1$ if $p$ is in the interior of $\alpha$, $\frac{1}{2}$ if $p$ is
on the boundary, and 0 otherwise. This is then extended bilinearly to all of
$H_1(Z',\mathbf{a})\times H_0(\mathbf{a})$ to define $m$.

Given a generator $a\in\sla(\slz)$, let $\bm{\rho}$ be the non-horizontal
strands of $a$. Let $\mathrm{inv}(\bm{\rho})$ be the number of inversions in
$\bm{\rho}$, $S\in H_0(\mathbf{a})$ be the starting points of $\bm{\rho}$, and
$[a]\in H_1(Z',\mathbf{a})$ be the multiplicity of $a$. Then
\[ \mathrm{gr}'(a) = (\mathrm{inv}(\bm{\rho}) - m([a], S), [a]). \]

Next, we consider relative gradings on the type $D$ invariant. Fix a bordered
Heegaard diagram $\mathcal{H}$. Let $\mathbf{x}, \mathbf{y}$ be generators and
$B\in\pi_2(\mathbf{x}, \mathbf{y})$, define $g'(B)\in G'(\slz)$ as
\[ g'(B) = (-e(B) - n_{\mathbf{x}}(B) -
n_{\mathbf{y}}(B), \partial^\partial(B)). \] Here $e(B)$ is the Euler measure of
$B$, and $n_{\mathbf{x}}(B), n_{\mathbf{y}}(B)$ are multiplicities of $B$ at
$\mathbf{x}, \mathbf{y}$ (each corner around $\mathbf{x}$ or $\mathbf{y}$ counts
as multiplicity $\frac{1}{4}$), and $\partial^\partial(B)$ is the boundary of
$B$ on $H_1(Z',\mathbf{a})$.

There is a grading set for each spin$^c$ class on $\mathcal{H}$ (in most
bordered cases we consider here, there is just one spin$^c$ class). The grading
set $S'_D(\mathcal{H},\mathfrak{s})$ for the Heegaard diagram $\mathcal{H}$ and
spin$^c$ class $\mathfrak{s}$ is defined as follows. Choose a base generator
$\mathbf{x_0}$ with spin$^c$ class $\mathfrak{s}$. Let $P'(\mathbf{x_0})$ be the
set of $g'(P)$ for all $P\in\pi_2(\mathbf{x_0}, \mathbf{x_0})$ (the domains in
$\pi_2(\mathbf{x_0}, \mathbf{x_0})$ are called \emph{periodic domains}). Then
\[ S'_D(\mathcal{H},\mathfrak{s}) = G'(-\slz) / R(P'(\mathbf{x_0})). \]

This grading set has an obvious left action by $G'(\slz)$. For another generator
$\mathbf{x}$ in the same spin$^c$ class, choose a domain
$B_0\in\pi_2(\mathbf{x_0}, \mathbf{x})$, and set
\[ \mathrm{gr}'(\mathbf{x}) = R(g'(B_0)) \cdot R(P'(\mathbf{x_0})). \]

The type $D$ action respects this relative grading in the sense that, for each
arrow $\delta^1: \mathbf{x}\to a\otimes\mathbf{y}$ in the action, we have
\[ \lambda^{-1}\mathrm{gr}'(\mathbf{x}) =
\mathrm{gr}'(a)\mathrm{gr}'(\mathbf{y}). \]

Relative gradings on type $A$ invariants are similar. The grading set is
\[ S'_A(\mathcal{H},\mathfrak{s}) = P'(\mathbf{x_0}) \backslash G'(\slz). \]
This carries a natural right action by $G'(\slz)$. For any generator
$\mathbf{x}$ in the spin$^c$ class $\mathfrak{s}$, choose a domain
$B_0\in\pi_2(\mathbf{x_0}, \mathbf{x})$ and set
\[ \mathrm{gr}'(\mathbf{x}) = P'(\mathbf{x_0}) \cdot g'(B_0). \]

The $A_\infty$-action respects the relative grading in the sense that, for each
arrow
\[ m_{1,k}: (\mathbf{x}; a_1, \dots, a_k) \to \mathbf{y}, \] we have
\[ \lambda^{k-1}\mathrm{gr}'(\mathbf{x})\mathrm{gr}'(a_1)\cdots\mathrm{gr}'(a_k)
= \mathrm{gr}'(\mathbf{y}). \]

Gradings on bimodules are defined similarly. In particular, a domain in a
Heegaard diagram with two boundary components parametrized by $\slz_1$ and
$\slz_2$ gives rise to an element of $G'(\slz_1)\times_\lambda
G'(\slz_2)=G'(\slz_1)\times G'(\slz_2) / (\lambda_1=\lambda_2)$. The grading set
is a certain coset of $G'(\slz_1)\times_\lambda G'(\slz_2)$.

Now we briefly discuss refined gradings, which contain essentially the same
information, but are cleaner to work with theoretically.

The group $G(\slz)$ can be considered as a subgroup of $G'(\slz)$, generated by
$\lambda$ and elements of the form $(0, [p, q])$, where $p, q$ is a pair of
matched points, and $[p, q]$ denotes the interval in $H_1(Z', \mathbf{a})$
between $p$ and $q$. An element $(k,\alpha)$ of $G'(\slz)$ is in $G(\slz)$ if
and only if $M_*(\partial\alpha)=0$ where $\partial: H_1(Z',\mathbf{a})\to
H_0(\mathbf{a})$ is the boundary operator and
$M_*:H_0(\mathbf{a})\to\mathbb{Z}^{2k}$ is a map sending each matched pair of
points to the same basis element of $\mathbb{Z}^{2k}$.

To construct the refined grading on $\sla(\slz)$, we first choose a base
idempotent $s_0$ in $\sla(\slz)$. Then for every idempotent $s$, choose a
grading element $\psi(s)=(k,\alpha)\in G'(\slz)$ such that
$M_*(\partial\alpha)=s-s_0$. For an algebra element $a$ with left idempotent $s$
and right idempotent $t$, we set
\[ \mathrm{gr}(a)=\psi(s)\mathrm{gr}'(a)\psi(t)^{-1}. \] It is easy to check
that this element lies in $G(\slz)$ and that the two conditions on the grading
are satisfied.

Similarly, we can ``refine'' the grading on the bordered invariants to use
$G(\slz)$ rather than $G'(\slz)$. We will omit the details here.

We will not perform any detailed grading computations in this paper, but will
simply note that all such computations can be done combinatorially from the
Heegaard diagram. For a module (or bimodule) $M$ of any type, grading imposes a
constraint on what kind of arrows can appear in the $A_\infty$ or type $D$
action on $M$. One such constraint is as follows: if a domain in a Heegaard
diagram with two boundary components touches the two boundaries at intervals $i$
and $i'$, respectively, then for each arrow in the algebra action of a bimodule
corresponding to that Heegaard diagram, its multiplicities at $i$ and at $i'$
must be the same. Such constraints are crucial in establishing uniqueness
properties of bimodule invariants, to be discussed in the following sections.

\subsection{The strongly-based mapping class groupoid}

An important class of 3-manifolds with two boundary components is the mapping
cylinders of surface diffeomorphisms. Gluing with these 3-manifolds can be
considered as ``changing the parametrization'' on the boundary of a bordered
3-manifold.

The strongly-based mapping class groupoid of genus $g$ is a category whose
objects are pointed matched circles with $4g$ points. Each object $\slz$
corresponds to a surface $F^\circ(\slz)$ of genus $g$, with standard
parametrization by $\slz$. The morphisms from $\slz_1$ to $\slz_2$ in the
category are isotopy classes of diffeomorphisms $\phi:F^\circ(\slz_1)\to
F^\circ(\slz_2)$, sending the basepoint $z_1\in F^\circ(\slz_1)$ to the
basepoint $z_2\in F^\circ(\slz_2)$. Identity and composition in the category
correspond to the identity diffeomorphism and composition of diffeomorphisms,
respectively.

If we fix a pointed matched circle $\slz$ and only consider morphisms from
$\slz$ to itself, we obtain the strongly-based mapping class group of $F_{g,1}$
(where $F_{g,1}$ denotes a genus $g$ surface with one circle boundary). This is
simply the group of isotopy classes of boundary-preserving self-diffeomorphisms
of $F_{g,1}$.

Given a diffeomorphism $\phi: F^\circ(\slz_1)\to F^\circ(\slz_2)$, we can
construct its mapping cylinder $Y(\phi)=F(\slz_2)\times [0,1]$ as a strongly
bordered 3-manifold with two boundary components, parametrized by $-\slz_1$ and
$\slz_2$. The left boundary $\partial_LY(\phi)=F(\slz_2)\times\{0\}$ is
parametrized by the induced map $\phi_*: -F(\slz_1)\to -F(\slz_2)$ (reverse
orientation and extend over the disk filling the boundary), while the right
boundary is parametrized by the identity map on $F(\slz_2)$. The map
$\gamma:(S^1,z)\times[0,1]\to Y(\phi)$ is simply the inclusion $(Z,z)\times
[0,1] \to F(\slz_2)\times [0,1]$.

This establishes an one-to-one correspondence between strongly bordered
3-manifolds that are topologically $F_g\times [0,1]$, and morphisms in the
strongly-based mapping class groupoid with genus $g$. For a morphism
$\phi:F^\circ(\slz_1)\to F^\circ(\slz_2)$, we denote
$\widehatit{CFAA}(\phi)_{\sla(-\slz_1),\sla(\slz_2)}$ to be the type $\AA$
invariant $\widehatit{CFAA}(Y(\phi))_{\sla(-\slz_1),\sla(\slz_2)}$ associated to
the mapping cylinder of $\phi$. Likewise we have notations
$^{\sla(\slz_1)}\widehatit{CFDA}(\phi)_{\sla(\slz_2)}$ and
$^{\sla(\slz_1),\sla(-\slz_2)}\widehatit{CFDD}(\phi)$ for the other invariants
corresponding to $Y(\phi)$.

For future reference, we write down the pairing theorems involving $\DA$
invariants. For morphisms $\phi_1:F^\circ(\slz_1)\to F^\circ(\slz_2)$ and
$\phi_2:F^\circ(\slz_2)\to F^\circ(\slz_3)$, the $\DA$ invariant for
$\phi_2\circ\phi_1:F^\circ(\slz_1)\to F^\circ(\slz_3)$ is given by:
\begin{equation} \label{eq:da-mcg-pairing}
  ^{\sla(\slz_1)}\widehatit{CFDA}(\phi_2\circ\phi_1)_{\sla(\slz_3)} =
  \tensor[^{\sla(\slz_1)}]{\widehatit{CFDA}(\phi_1)}{_{\sla(\slz_2)}} \boxtimes
  \tensor[^{\sla(\slz_2)}]{\widehatit{CFDA}(\phi_2)}{_{\sla(\slz_3)}}.
\end{equation}
For a morphism $\phi:F^\circ(\slz_1)\to F^\circ(\slz_2)$ and a 3-manifold $Y$
with boundary parametrized by $\psi:F(-\slz_2)\to\partial Y$, let $Y'$ be the
same manifold with boundary parametrized by
$\psi\circ\phi_*:F(-\slz_1)\to\partial Y$, then
\begin{equation} \label{eq:d-mcg-pairing}
  ^{\sla(\slz_1)}\widehatit{CFD}(Y') =
  \tensor[^{\sla(\slz_1)}]{\widehatit{CFDA}(\phi)}{_{\sla(\slz_2)}} \boxtimes
  \tensor[^{\sla(\slz_2)}]{\widehatit{CFD}(Y)}{}.
\end{equation}

\subsection{Invariants of the identity diffeomorphism}

Let $\iz$ be the identity morphism $F^\circ(\slz)\to F^\circ(\slz)$. All
bimodule invariants associated to $\iz$ have special significance in the
theory. First, it can be shown that (\cite[Section 8.1]{LOT10a}):
\[ \widehatit{CFDA}(\iz) \simeq
\tensor[^{\sla(\slz)}]{\mathbb{I}}{_{\sla(\slz)}}, \] where the latter denotes
the identity type $\DA$ bimodule over $\sla(\slz)$. This is the bimodule
generated over $\mathbb{F}_2$ by idempotents of $\sla(\slz)$, and with the
algebra action given by:
\[ \delta_2^1(i, a) = a\otimes j, \] for any generator $a\in\sla(\slz)$, where
$i$ and $j$ are the left and right idempotents of $a$.

The type $\DD$ invariant $^{\sla(\slz),\sla(-\slz)}\widehatit{CFDD}(\iz)$ and type
$\AA$ invariant $\widehatit{CFAA}(\iz)_{\sla(-\slz),\sla(\slz)}$ relate the type
$A$ and type $D$ invariants through taking the tensor product. For any
3-manifold $Y$ with one boundary component parametrized by $\slz$, the relations
are:
\begin{eqnarray}
  \widehatit{CFA}(Y)_{\sla(\slz)} =
  \widehatit{CFAA}(\iz)_{\sla(-\slz),\sla(\slz)} \boxtimes
  \tensor*[^{\sla(-\slz)}]{\widehatit{CFD}(Y)}{} \\
  ^{\sla(-\slz)}\widehatit{CFD}(Y) =
  \widehatit{CFA}(Y)_{\sla(\slz)} \boxtimes
  \tensor*[^{\sla(\slz),\sla(-\slz)}]{\widehatit{CFDD}(\iz)}{}.
\end{eqnarray}

One implication is that $\widehatit{CFD}(Y)$ and $\widehatit{CFA}(Y)$ contain the
same information about $Y$. Likewise, there are relations among the bimodule
invariants, showing that all bimodule invariants also contain the same
information. For any 3-manifold $Y$ with two boundary components parametrized by
$\slz_1$ and $\slz_2$:
\begin{eqnarray}
  ^{\sla(-\slz_1)}\widehatit{CFDA}(Y)_{\sla(\slz_2)} =
  \widehatit{CFAA}(\mathbb{I}_{\slz_2})_{\sla(-\slz_2),\sla(\slz_2)}
  \boxtimes_{\sla(-\slz_2)}
  \tensor*[^{\sla(-\slz_1),\sla(-\slz_2)}]{\widehatit{CFDD}(Y)}{} \label{eq:dafromddaa} \\
  \widehatit{CFAA}(Y)_{\sla(\slz_1),\sla(\slz_2)} =
  \widehatit{CFAA}(\mathbb{I}_{\slz_1})_{\sla(-\slz_1),\sla(\slz_1)}
  \boxtimes_{\sla(-\slz_1)}
  \tensor[^{\sla(-\slz_1)}]{\widehatit{CFDA}(Y)}{_{\sla(\slz_2)}} \\
  \tensor[^{\sla(-\slz_1)}]{\widehatit{CFDA}(Y)}{_{\sla(\slz_2)}} =
  \widehatit{CFAA}(Y)_{\sla(\slz_1),\sla(\slz_2)} \boxtimes_{\sla(\slz_1)}
  \tensor*[^{\sla(\slz_1),\sla(-\slz_1)}]{\widehatit{CFDD}(\mathbb{I}_{\slz_1})}{}
  \\
  \tensor*[^{\sla(-\slz_1),\sla(-\slz_2)}]{\widehatit{CFDD}(Y)}{} =
  \tensor[^{\sla(-\slz_1)}]{\widehatit{CFDA}(Y)}{_{\sla(\slz_2)}}
  \boxtimes_{\sla(\slz_2)}
  \tensor*[^{\sla(\slz_2),\sla(-\slz_2)}]{\widehatit{CFDD}(\mathbb{I}_{\slz_2})}{}.
\end{eqnarray}

The above equations are special cases of the pairing theorems (where one of the
bordered 3-manifolds is a cylinder with trivial parametrization). It indicates
the importance of finding combinatorial descriptions of $\widehatit{CFDD}(\iz)$
and $\widehatit{CFAA}(\iz)$, which we will now consider.

First, we describe the combinatorial model $\hatdd(\iz)$ of
$^{\sla(\slz),\sla(-\slz)}\widehatit{CFDD}(\iz)$, given in \cite[Theorem
1]{LOT10c}. It is generated over $\mathbb{F}_2$ by the set of pairs of
complementary idempotents $i\otimes i'$, with $i\in\sla(\slz)$ and
$i'=\overline{o(i)}\in\sla(-\slz)$. The type $\DD$ action is given by:

\begin{equation} \label{eq:ddop} \delta^1(i\otimes i') =
  \sum_{\xi\in\mathcal{C}, ia(\xi)=a(\xi)j,
    i'\overline{a(\xi)}=\overline{a(\xi)}j'} (a(\xi) \otimes \overline{a(\xi)})
  \otimes (j\otimes j'),
\end{equation}
where $\mathcal{C}$ is the set of chords on $\mathcal{Z}$ whose two endpoints
are not matched. Intuitively, the arrows in the type $\DD$ action are exactly
those whose two algebra outputs both contain exactly one chord connecting two
unpaired points, and covering corresponding intervals in $\sla(\slz)$ and
$\sla(-\slz)$.

Next, we consider the invariant
$\widehatit{CFAA}(\iz)_{\sla(-\slz),\sla(\slz)}$.  A formula for it is given in
\cite[Proposition 9.2]{LOT10a} as follows (in this equation we simplify
$\sla(\slz)$ to $\sla$ and $\sla(-\slz)$ to $\sla'$).
\begin{eqnarray}\label{eq:typeaaid}
  \widehatit{CFAA}(\iz)_{\sla',\sla}
  &=& \mathrm{Mor}^{\sla}(\tensor[_{\sla'}]{{\sla'}}{_{\sla'}} \boxtimes_{\sla'}
  \tensor[^{\sla,\sla'}]{\widehatit{CFDD}(\iz)}{}, \tensor[^\sla]{\mathbb{I}}{_\sla}). \nonumber \\
  &=& \left(\overline{\widehatit{CFDD}(\iz)}^{\sla',\sla} \boxtimes_{\sla'}
    \tensor[_{\sla'}]{\overline{\sla'}}{_{\sla'}}\right) \boxtimes_{\sla}
  \tensor[_\sla]{\sla}{_\sla}.
\end{eqnarray}

This bimodule has a large number of generators, making it difficult to use for
the computations needed in this paper. The main result of \cite{BZ1} is to
describe a bimodule $\hataa(\iz)$ homotopy equivalent to this (and hence is also
a combinatorial model of $\widehatit{CFAA}(\iz)$), but with a minimal number of
generators. The bimodule $\hataa(\iz)$ is generated over $\mathbb{F}_2$ by the
set of pairs of complementary idempotents, but with much more complex
$A_\infty$-bimodule actions. We will briefly review this construction in Section
\ref{sec:cfaaconstruction}.

One of the pairing theorems imply the following relation among the combinatorial
models for $\iz$:
\begin{equation} \label{eq:ddtensoraa}
  \tensor[^{\sla(\slz)}]{\mathbb{I}}{_{\sla(\slz)}} \simeq
  \hataa(\iz)_{\sla(-\slz),\sla(\slz)} \boxtimes_{\sla(-\slz)}
  \tensor[^{\sla(\slz),\sla(-\slz)}]{\hatdd(\iz)}{}.
\end{equation}

The two sides are not equal but only homotopy equivalent. This homotopy
equivalence is proved combinatorially as Corollary \ref{cor:ddtimesaais1} in
Section \ref{sec:daoperations}.

\subsection{Invariants of arcslides}\label{sec:invarcslidesintro}

The strongly-based mapping class groupoid is generated by a particularly simple
class of morphisms called \emph{arcslides}. We will now review their definitions
and the invariants associated to them. The relations among arcslides will be
described at the beginning of Section \ref{sec:computationda}.

Given a pointed matched circle $\slz_1$, and two matched pairs of points
$B=(b_1,b_2)$ and $C=(c_1,c_2)$ in $\mathbf{a}\subset Z_1$, such that $b_1$ and
$c_1$ are adjacent in $\mathbf{a}$, an arcslide of $b_1$ over $c_1$ moves $b_1$
to be adjacent to $c_2$, on the side opposite to its original position with
respect to $c_1$. This results in a new pointed matched circle $\slz_2$. Such a
move corresponds to a certain diffeomorphism $F^\circ(\slz_1)\to
F^\circ(\slz_2)$, which we will also call an arcslide. See Figure
\ref{fig:exarcslides} for two examples of arcslides. The first example is an
\emph{overslide} meaning $b_1$ is outside the interval $[c_1,c_2]$. The second
example is an \emph{underslide} meaning $b_1$ is inside that interval.

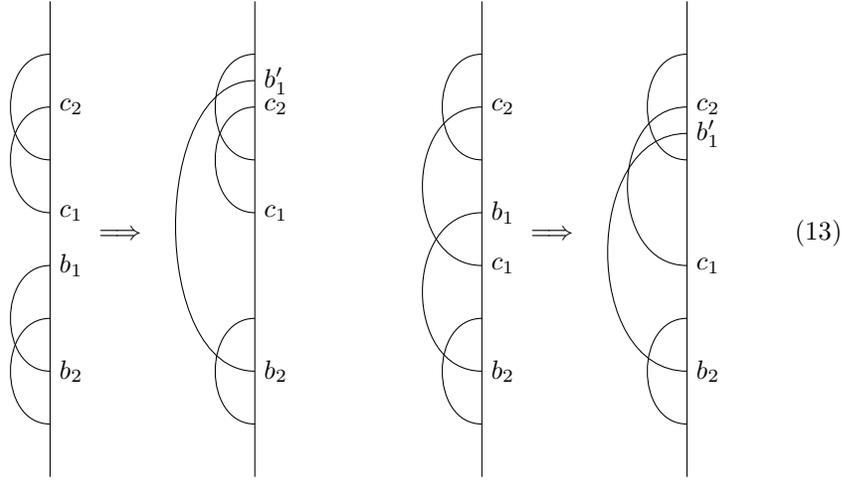
\begin{figure}[h!tb]
  \centering
  \begin{equation}
    \begin{tikzpicture} [x=20,y=20,baseline=(current bounding box.center)]
      \draw (0,0) to (0,9);
      \draw (0,1) .. controls (-1,1) and (-1,3) .. (0,3);
      \draw (0,2) .. controls (-1,2) and (-1,4) .. (0,4);
      \draw (0,5) .. controls (-1,5) and (-1,7) .. (0,7);
      \draw (0,6) .. controls (-1,6) and (-1,8) .. (0,8);
      \draw (0,2) node[right] {$b_2$};
      \draw (0,4) node[right] {$b_1$};
      \draw (0,5) node[right] {$c_1$};
      \draw (0,7) node[right] {$c_2$};
    \end{tikzpicture}
    \Longrightarrow
    \begin{tikzpicture} [x=20,y=20,baseline=(current bounding box.center)]
      \draw (0,0) to (0,9);
      \draw (0,1) .. controls (-1,1) and (-1,3) .. (0,3);
      \draw (0,2) .. controls (-2,2) and (-2,7.5) .. (0,7.5);
      \draw (0,5) .. controls (-1,5) and (-1,7) .. (0,7);
      \draw (0,6) .. controls (-1,6) and (-1,8) .. (0,8);
      \draw (0,2) node[right] {$b_2$};
      \draw (0,7.5) node[right] {$b_1'$};
      \draw (0,5) node[right] {$c_1$};
      \draw (0,7) node[right] {$c_2$};
    \end{tikzpicture}
    \quad\quad\quad\quad
    \begin{tikzpicture} [x=20,y=20,baseline=(current bounding box.center)]
      \draw (0,0) to (0,9);
      \draw (0,1) .. controls (-1,1) and (-1,3) .. (0,3);
      \draw (0,2) .. controls (-1.5,2) and (-1.5,5) .. (0,5);
      \draw (0,4) .. controls (-1.5,4) and (-1.5,7) .. (0,7);
      \draw (0,6) .. controls (-1,6) and (-1,8) .. (0,8);
      \draw (0,2) node[right] {$b_2$};
      \draw (0,4) node[right] {$c_1$};
      \draw (0,5) node[right] {$b_1$};
      \draw (0,7) node[right] {$c_2$};
    \end{tikzpicture}
    \Longrightarrow
    \begin{tikzpicture} [x=20,y=20,baseline=(current bounding box.center)]
      \draw (0,0) to (0,9);
      \draw (0,1) .. controls (-1,1) and (-1,3) .. (0,3);
      \draw (0,2) .. controls (-2,2) and (-2,6.5) .. (0,6.5);
      \draw (0,4) .. controls (-1.5,4) and (-1.5,7) .. (0,7);
      \draw (0,6) .. controls (-1,6) and (-1,8) .. (0,8);
      \draw (0,2) node[right] {$b_2$};
      \draw (0,4) node[right] {$c_1$};
      \draw (0,6.5) node[right] {$b_1'$};
      \draw (0,7) node[right] {$c_2$};
    \end{tikzpicture}
  \end{equation}
  \caption{Two examples of arcslides.}
  \label{fig:exarcslides}
\end{figure}

Given an arcslide $\tau:F^\circ(\slz_1)\to F^\circ(\slz_2)$, the invariant
$\widehatit{CFDD}(\tau)$ is a left type $\DD$ bimodule over $A(\slz_1)$ and
$A(-\slz_2)$. Constructing a combinatorial model of this bimodule, denoted
$\hatdd(\tau)$, is the main subject of \cite{LOT10c}. This model is computed
from a standard Heegaard diagram for the mapping cylinder $Y(\tau)$. For the two
arcslides in Figure \ref{fig:exarcslides}, these standard Heegaard diagrams are
shown in Figure \ref{fig:exhdiagram}. The tiny circles in the diagrams are
1-handle attachment points, paired according to their vertical positions. The
larger circles are $\beta$ circles, and all other arcs inside the boundary are
$\alpha$-arcs. Later on, we will draw more schematic versions of these diagrams,
omitting some of the $\beta$ circles and attaching points of 1-handles.

\begin{figure}[h!tb]
  \begin{tikzpicture} [x=18,y=18,baseline=(current bounding box.center)]
    \draw (0,0) to (0,10);
    \draw (6,0) to (6,10);
    \draw (0,1) to (6,1);\draw (0,2) to (6,2);\draw (0,3) to (6,3);
    \draw (0,5) to (6,5);\draw (0,6) to (6,6);\draw (0,7) to (6,7);\draw (0,9) to (6,9);
    \draw (0,1) .. controls (-1,1) and (-1,3) .. (0,3);
    \draw (0,2) .. controls (-1,2) and (-1,4) .. (0,4);
    \draw (0,5) .. controls (-1,5) and (-1,7) .. (0,7);
    \draw (0,6) .. controls (-1,6) and (-1,9) .. (0,9);
    \draw (6,1) .. controls (7,1) and (7,3) .. (6,3);
    \draw (6,2) .. controls (8,2) and (8,8) .. (6,8);
    \draw (6,5) .. controls (7,5) and (7,7) .. (6,7);
    \draw (6,6) .. controls (7,6) and (7,9) .. (6,9);
    \draw (0,4) .. controls (2,4) .. (2,5);
    \draw (2,7) .. controls (2,8) .. (6,8);
    \filldraw[fill=white] (1,1) circle (0.1);
    \filldraw[fill=white] (1,3) circle (0.1);
    \draw (1,1) circle (0.5);
    \filldraw[fill=white] (2,5) circle (0.1);
    \filldraw[fill=white] (2,7) circle (0.1);
    \draw (2,5) circle (0.5);
    \filldraw (1.97,4.5) circle (0.1);
    \draw (1.97,4.5) node[below right] {$y$};
    \filldraw[fill=white] (4,2) circle (0.1);
    \filldraw[fill=white] (4,8) circle (0.1);
    \draw (4,2) circle (0.5);
    \filldraw[fill=white] (5,6) circle (0.1);
    \filldraw[fill=white] (5,9) circle (0.1);
    \draw (5,6) circle (0.5);
  \end{tikzpicture}
  \quad\quad
  \begin{tikzpicture} [x=18,y=18,baseline=(current bounding box.center)]
    \draw (0,0) to (0,10);
    \draw (6,0) to (6,10);
    \draw (0,1) to (6,1);\draw (0,2) to (6,2);\draw (0,3) to (6,3);\draw (0,4) to (6,4);
    \draw (0,6) to (6,6);\draw (0,8) to (6,8);\draw (0,9) to (6,9);
    \draw (0,1) .. controls (-1,1) and (-1,3) .. (0,3);
    \draw (0,2) .. controls (-1,2) and (-1,5) .. (0,5);
    \draw (0,4) .. controls (-1,4) and (-1,7) .. (0,7);
    \draw (0,6) .. controls (-1,6) and (-1,9) .. (0,9);
    \draw (6,1) .. controls (7,1) and (7,3) .. (6,3);
    \draw (6,2) .. controls (8,2) and (8,7) .. (6,7);
    \draw (6,4) .. controls (7.5,4) and (7.5,8) .. (6,8);
    \draw (6,6) .. controls (7,6) and (7,9) .. (6,9);
    \draw (0,5) .. controls (2,5) .. (2,4);
    \draw (2,8) .. controls (2,7) .. (6,7);
    \filldraw[fill=white] (1,1) circle (0.1);
    \filldraw[fill=white] (1,3) circle (0.1);
    \draw (1,1) circle (0.5);
    \filldraw[fill=white] (2,4) circle (0.1);
    \filldraw[fill=white] (2,8) circle (0.1);
    \draw (2,4) circle (0.5);
    \filldraw (1.97,4.5) circle (0.1);
    \draw (1.97,4.5) node[above right] {$y$};
    \filldraw[fill=white] (4,2) circle (0.1);
    \filldraw[fill=white] (4,7) circle (0.1);
    \draw (4,2) circle (0.5);
    \filldraw[fill=white] (5,6) circle (0.1);
    \filldraw[fill=white] (5,9) circle (0.1);
    \draw (5,6) circle (0.5);
  \end{tikzpicture}
  \caption{Examples of Heegaard diagrams of arcslides.}
  \label{fig:exhdiagram}
\end{figure}
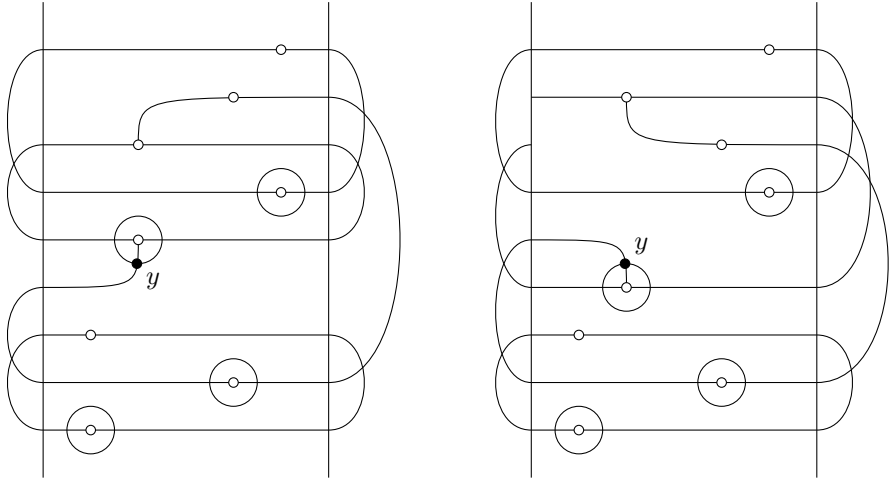

Each generator of $\hatdd(\tau)$ corresponds to a $g$-tuple of intersection
points between $\alpha$ and $\beta$ curves, where each $\beta$ circle contains
exactly one point, and each $\alpha$-arc contains at most one point. The (type
$D$) idempotent of the generator specifies which pairs of $\alpha$-arcs are
\emph{not} occupied by the generator. For the standard Heegaard diagram of
arcslides, a generator is uniquely specified by its idempotents on the two
sides.

There is an obvious identification between pairs of points on the two sides,
using which we can talk about two idempotents on different sides being
complementary, etc. There are two types of generators in $\hatdd(\tau)$. A
generator of type $X$ has complementary idempotents, and a generator of type $Y$
has idempotents that are complementary except for both containing the $C$ pair
and neither containing the $B$ pair. The type $X$ generators are those that do
not occupy the intersection point $y$ in Figure \ref{fig:exhdiagram}, while the
type $Y$ generators do.

The type $\DD$ action on the bimodule can be described as follows. Given a
pointed matched circle $\slz$, let $C(\slz)$ denote the collection of sets of
chords in $\slz$. For some $\xi\in C(\slz)$, let $a(\xi)\in \sla(\slz)$ denote
the sum of all generators of $\sla(\slz)$ produced by adding horizontal strands
to $\xi$ (this definition extends the case where $\xi$ is a chord). For any
arcslide $\tau:F^\circ(\slz_1)\to F^\circ(\slz_2)$, there is a collection of
pairs $C_{\tau}\subseteq \{(\xi_i,\eta_i)\,|\,\xi_i\in C(\slz_1), \eta_i\in
C(\slz_2)\}$ such that
\begin{equation}
  \delta^1(i\otimes i') = \sum_{\substack{
      (\xi_k,\eta_k)\in C_{\tau}\\
      ia(\xi_k)=a(\xi_k)j\\
      i'\overline{a(\eta_k)}=\overline{a(\eta_k)}j'\\
      j\otimes j' \mathrm{is~a~generator}}}
  (a(\xi_k)\otimes\overline{a(\eta_k)}) \otimes (j\otimes j'),
\end{equation}
where generators are represented by their type $D$ idempotents.

Intuitively, there is a term in the type $\DD$ action whenever the idempotent
agrees, and the moving strands part of the algebra coefficients match one of the
fixed patterns. Depending on whether the arcslide is an underslide or an
overslide, there are six or eight types of elements in $C_{\tau}$. See Figure 21
and 28 in \cite{LOT10c} for diagrams of these patterns. Note that not all pairs
in Figure 28 are actually in $C_{\tau}$ -- there is an additional choice
involved. In the following computations we will only use some of the simpler
pairs, involving algebra elements that have small total length. In particular we
will not need to consider any pair where a choice is necessary.

The following properties of $\hatdd(\tau)$ can be directly verified for the
above description:

\begin{itemize}
\item (Relation with Heegaard diagram) Every arrow in the type $\DD$ action
  comes from a domain in the Heegaard diagram. For a Heegaard diagram
  $\mathcal{H}$, denote $\bm{\alpha}$ and $\bm{\beta}$ to be the union of
  $\alpha$ and $\beta$ curves, respectively. A \emph{domain} in $\mathcal{H}$ is
  a non-negative integral linear combination of connected components of
  $\mathcal{H}\setminus\{\bm{\alpha},\bm{\beta}\}$.  Each arrow $\mathbf{x}\to
  a_1\otimes a_2\otimes\mathbf{y}$ in the type $\DD$ action comes from a domain
  $B$, such that $a_1$ and $a_2$ have multiplicities equal to the multiplicities
  of $B$ at the corresponding boundaries. Moreover, let $\partial^\alpha B$ be
  the part of the boundary of $B$ on the $\alpha$ curves, and let
  $\partial(\partial^\alpha B)$ be the part of the boundary of $\partial^\alpha
  B$ in the interior of the diagram, as a signed sum of intersection points,
  then $\partial(\partial^\alpha B)=\mathbf{x}-\mathbf{y}$. Intuitively, the
  $\alpha$-boundaries of $B$ start at points of $\mathbf{x}$ and end at points
  of $\mathbf{y}$, and vice versa for the $\beta$-boundary.

\item (Grading) There is a refined grading on the generators of $\hatdd(\tau)$
  to a particular grading set $S_\tau$, which has left-right actions by
  $G(\slz_1)$ and $G(\slz_2)$. Both actions are free and transitive, which means
  $S_\tau$ induces a group isomorphism $G(\slz_1)\to G(\slz_2)$. This group
  isomorphism is an invariant of $\tau$, up to composing by inner automorphisms
  of $G(\slz_1)$ and $G(\slz_2)$. In other words, $\tau$ induces an element in
  the set of outer isomorphisms $\mathrm{Out}(G(\slz_1),G(\slz_2))$. In fact,
  this outer isomorphism corresponds to the actions of $\tau$ on the homology of
  the surface (see \cite[Section 6.2]{LOT10c} for details).

\item (Stabilization) Given arcslide $\tau:F^\circ(\slz_1)\to F^\circ(\slz_2)$,
  let $\mathring{\slz_1}=\slz_1\#\slz^1$ and $\mathring{\slz_2}=\slz_2\#\slz^1$,
  where $\slz^1$ is the genus 1 pointed matched circle, and $\cdot~\#~\cdot$
  denotes connect sum on pointed matched circles. Let
  $\mathring{\tau}:F(\mathring{\slz_1})\to F(\mathring{\slz_2})$ be the arcslide
  acting as identity on the new handle, and as $\tau$ elsewhere. This is called
  the \emph{stabilization} of $\tau$. Then $\hatdd(\tau)$ and
  $\hatdd(\mathring{\tau})$ are related as follows: fix any idempotent $i_o$ on
  $\slz^1$ (occupying one of the two possible pairs), then there is an injection
  from generators of $\hatdd(\tau)$ into generators of
  $\hatdd(\mathring{\tau})$, sending $i\otimes i'$ in $\hatdd(\tau)$ to
  $(i\#i_o)\otimes (i'\#\overline{o(i_o)})$ in $\hatdd(\mathring{\tau})$. For
  any generator $\mathbf{x}$ in $\hatdd(\tau)$, let $\mathring{\mathbf{x}}$ be
  the corresponding generator in $\hatdd(\mathring{\tau})$. Then for any two
  generators $\mathbf{x}$, $\mathbf{y}$ in $\hatdd(\tau)$, there is a one-to-one
  correspondence between arrows from $\mathbf{x}$ to $\mathbf{y}$ in the $\DD$
  action of $\hatdd(\tau)$ and arrows from $\mathbf{\mathring{x}}$ to
  $\mathbf{\mathring{y}}$ in the $\DD$ action of $\hatdd(\mathring{\tau})$ that
  do not cover any region around the adjoined $\slz^1$, with $\mathbf{x}\to
  a\otimes b\otimes\mathbf{y}$ corresponding to $\mathbf{\mathring{x}}\to
  \mathring{a}\otimes\mathring{b}\otimes\mathbf{\mathring{y}}$, where
  $\mathring{a}$ and $\mathring{b}$ are obtained from $a$ and $b$ by adjoining
  the appropriate idempotents.

\item (Duality) For any arcslide $\tau:F^\circ(\slz_1)\to F^\circ(\slz_2)$, let
  $\overline{\tau}:F^\circ(-\slz_1)\to F^\circ(-\slz_2)$ be the arcslide with
  reversed orientation. Then $^{\sla(\slz_1),\sla(-\slz_2)}\hatdd(\tau)$ and
  $^{\sla(-\slz_1),\sla(\slz_2)}\hatdd(\overline{\tau})$ are dual to each other
  (using definition of dual at the end of Section \ref{sec:intropmc}). That is,
  \begin{equation} \label{eq:cfddarcslidedual}
    \overline{^{\sla(\slz_1),\sla(-\slz_2)}\hatdd(\tau)} \simeq
    \tensor[^{\sla(-\slz_1),\sla(\slz_2)}]{\hatdd(\overline{\tau})}{}.
  \end{equation}
  Furthermore, the invariant $^{\sla(\slz_2),\sla(-\slz_1)}\hatdd(\tau^{-1})$ is
  homotopy equivalent to the right side of Equation (\ref{eq:cfddarcslidedual}),
  after switching the two algebra actions (this comes from the fact that the
  mapping cylinder of $\tau^{-1}$ is the mirror image of the mapping cylinder of
  $\tau$).
\end{itemize}

\subsection{Main constructions}

We now summarize the combinatorial constructions that will be studied in this
paper. From here on, we will no longer use the analytical definitions of
invariants, but define everything combinatorially from scratch. We will use
notations such as $\widehatit{CFAA}$ to denote (combinatorially defined)
homotopy equivalence classes of bimodules, and notations such as $\hatdd$,
$\hatda$ to denote particular combinatorial models in the equivalence classes of
bimodules. For modules with one algebra action, we will use $\widehatit{CFA}$
and $\widehatit{CFD}$ for both models and equivalence classes, as no confusion
will arise there. Since all combinatorial definitions below use either
constructions derived from the analytical definition, or the appropriate box
tensor product, it is clear that the entire construction agrees with the
analytical definitions.

First, we use models $\hatdd(\iz)$, $^{\sla(\slz)}\mathbb{I}_{\sla(\slz)}$, and
$\hataa(\iz)$ to define $\widehatit{CFDD}(\iz)$, $\widehatit{CFDA}(\iz)$, and
$\widehatit{CFAA}(\iz)$, respectively. Then Corollary \ref{cor:ddtimesaais1}
shows Equation (\ref{eq:ddtensoraa}) holds for our combinatorial
construction. This means box tensoring with $\widehat{CFDD}(\iz)$ and
$\widehat{CFAA}(\iz)$ are inverse operations on equivalence classes of
bimodules.

Next, we define $\hatda(\tau)$ as the box tensor product
\begin{equation} \label{eq:arcslidedadef}
  \tensor[^{\sla(\slz_1)}]{\hatda(\tau)}{_{\sla(\slz_2)}} =
  \hataa(\mathbb{I}_{\slz_2})_{\sla(-\slz_2),\sla(\slz_2)}
  \boxtimes_{\sla(-\slz_2)}
  \tensor[^{\sla(\slz_1),\sla(-\slz_2)}]{\hatdd(\tau)}{}
\end{equation}

Given this, we can define $\hatda(\phi)$ for an arbitrary element $\phi$ of the
strongly-based mapping class groupoid, by factoring $\phi$ into arcslides. The
precise statement is:

\begin{construction}\label{constr:invmappingclass}
  Given an element $\phi:F^\circ(\slz_1)\to F^\circ(\slz_{n+1})$ of the
  strongly-based mapping class groupoid, with factorization
  $\phi=\tau_n\circ\cdots\circ\tau_1$, where $\tau_i:F^\circ(\slz_i)\to
  F^\circ(\slz_{i+1})$. Write $\bm{\tau}$ for the sequence
  $\tau_1,\dots,\tau_n$. Define
  \[ ^{\sla(\slz_1)}\hatda(\phi, \bm{\tau})_{\sla(\slz_{n+1})} =
  \tensor[^{\sla(\slz_1)}]{\hatda(\tau_1)}{_{\sla(\slz_2)}}
  \boxtimes\cdots\boxtimes
  \tensor[^{\sla(\slz_n)}]{\hatda(\tau_n)}{_{\sla(\slz_{n+1})}}. \]
\end{construction}

\begin{theorem}\label{thm:invmappingclass}
  The homotopy type of $\hatda(\phi,\bm{\tau})$ does not depend on the choice of
  factorization $\bm{\tau}$. Hence, $\hatda(\phi,\bm{\tau})$ is an invariant of
  $\phi$ up to homotopy equivalence.
\end{theorem}

This theorem is proved in Section \ref{sec:relmcg}. Given this, we can define
$\widehatit{CFDA}(\phi)$ to be the equivalence class of
$\hatda(\phi,\bm{\tau})$, for any choice of factorization $\bm{\tau}$.

The other bimodule invariants $\widehatit{CFDD}(\phi)$,
$\widehatit{CFAA}(\phi)$, and $\widehatit{CFAD}(\phi)$ for a general morphism
$\phi:F^\circ(\slz_1)\to F^\circ(\slz_2)$ can be defined as:
\begin{eqnarray}
  ^{\sla(\slz_1),\sla(-\slz_2)}\widehatit{CFDD}(\phi) =
  \tensor[^{\sla(\slz_1)}]{\widehatit{CFDA}(\phi)}{_{\sla(\slz_2)}}
  \boxtimes_{\sla(\slz_2)}
  \tensor[^{\sla(\slz_2),\sla(-\slz_2)}]{\widehatit{CFDD}(\mathbb{I}_{\slz_2})}{},
  \nonumber \\
  \widehatit{CFAA}(\phi)_{\sla(-\slz_1),\sla(\slz_2)} =
  \widehatit{CFAA}(\mathbb{I}_{\slz_1})_{\sla(-\slz_1),\sla(\slz_1)}
  \boxtimes_{\sla(\slz_1)}
  \tensor[^{\sla(\slz_1)}]{\widehatit{CFDA}(\phi)}{_{\sla(\slz_2)}}, \nonumber \\
  ^{\sla(-\slz_2)}\widehatit{CFAD}(\phi)_{\sla(-\slz_1)} =
  \widehatit{CFAA}(\mathbb{I}_{\slz_1})_{\sla(-\slz_1),\sla(\slz_1)}
  \boxtimes_{\sla(\slz_1)}
  \tensor[^{\sla(\slz_1),\sla(-\slz_2)}]{\widehatit{CFDD}(\phi)}{}. \nonumber
\end{eqnarray}

Since $\hatdd(\iz)$ is the quasi-inverse of $\hataa(\iz)$ (that is, inverse up
to homotopy equivalence), we know $\widehatit{CFDD}(\tau)$ for an arcslide
$\tau$ can also be represented by $\hatdd(\tau)$. Also, expanding out the
definitions, we see $\widehatit{CFAD}(\phi)\simeq
\widehatit{CFAA}(\phi)\boxtimes\widehatit{CFDD}(\mathbb{I}_{\slz_2})$.

This concludes our construction of bimodule invariants (we will not need
bimodule invariants other than those for mapping classes of surface
diffeomorphisms). To construct invariants of closed 3-manifolds, we need one
more building block: $\widehatit{CFD}$ of the \emph{0-framed handlebody}
$\hg$. Here $\hg$ is the 3-manifold with one parametrized boundary given by the
Heegaard diagram in Figure \ref{fig:diaghg}.

\begin{figure} [h!tb]
  \centering
  \begin{tikzpicture} [x=18,y=18,baseline=(current bounding box.center)]
    \draw (-0.5,0) to (15,0);
    \draw (0,-0.1) node[below] {$z$};
    \foreach \x in {1,...,8} {
      \draw (\x, 0) to (\x, 1);
      \draw (\x, 0) node[below] {$\x$};
    }
    \foreach \x in {11,...,14} {
      \draw (\x, 0) to (\x, 1);
    }
    \draw (14, 0) node[below] {$4g$};
    \foreach \x in {2,4,6,8,12,14} {
      \filldraw[fill=white] (\x, 1) circle (0.1);
    }
    \foreach \x in {2,6,12} {
      \draw (\x, 1) circle (0.5);
      \draw (\x+1, 1) arc (0:180:1);
      \filldraw[fill=black] (\x, 0.5) circle (0.1);
    }
    \draw (9.5,1) node {$\cdots$};
    \draw (10.5,0) node[below] {$\cdots$};
  \end{tikzpicture}
  \caption{Heegaard diagram for the 0-framed handlebody. The numbers at bottom
    label points in $\slz^g$. Points in $-\slz^g$ are labeled in the reverse
    order.}
  \label{fig:diaghg}
\end{figure}

In this diagram, the small circles are 1-handle attachment points, paired
consecutively. The larger circles are $\beta$ circles, and all other arcs inside
the boundary are $\alpha$ arcs. From the way the $\alpha$ arcs meet the
boundary, we see that the boundary of $\hg$ is parametrized by the split pointed
matched circle of genus $g$, denoted as $\slz^g$. This is the pointed matched
circle with matching
\[ (1,3),(2,4),(5,7),(6,8),\dots,(4g-3,4g-1),(4g-2,4g). \] While it is true that
$-\slz^g=\slz^g$, we will usually distinguish them in order to emphasize
orientation changes.

The orientation reversal $-\hg$ is called the \emph{$\infty$-framed
  handlebody}. Its boundary is parametrized by $-\slz^g$. The Heegaard diagram
for $-\hg$ is the reflection of that for $\hg$.

The invariant $\widehatit{CFD}(\hg)$ has left type $D$ action by
$\sla(-\slz^g)$. It can be defined using the following model: there is a single
generator $\mathbf{x}$, corresponding to the set of intersection points
indicated in Figure \ref{fig:diaghg}. The idempotent of $\mathbf{x}$ contains
pairs $(2,4),(6,8),\dots$ in $-\slz^g$ (pairs corresponding to $\alpha$-arcs not
occupied by $\mathbf{x}$; note the labeling of points in $-\slz^g$ is
reversed). The type $D$ action is
\[ \delta^1(\mathbf{x}) = \sum_{\xi\in\mathcal{D}} a(\xi)\cdot\mathbf{x}, \]
where $\mathcal{D}$ is the set of chords $\{2\to 4, 6\to 8, \dots\}$. The
invariant $\widehatit{CFD}(-\hg)$ can be defined to be the dual of
$\widehatit{CFD}(\hg)$.

We now give the combinatorial construction of $\widehatit{HF}(Y)$ for a closed
3-manifold $Y$, following the spirit of the construction in \cite{LOT10c}.

\begin{construction} \label{constr:invclosed} Let $Y$ be a closed
  3-manifold. Choose a Heegaard splitting $Y_1\cup_u Y_2$ of $Y$, where
  $u:\partial Y_1\to -\partial Y_2$ is the gluing map. Fix circle and basepoint
  $(Z,z)$ on the gluing boundary $Y_1\cap Y_2$, and diffeomorphisms $f_1:\hg\to
  Y_1$ and $f_2:-\hg\to Y_2$, preserving $(Z,z)$, from the standard handlebodies
  to $Y_1$ and $Y_2$. Let $f_{1*}:F^\circ(\slz^g)\to\partial Y_1$ and
  $f_{2*}:F^\circ(-\slz^g)\to\partial Y_2$ be the restrictions of $f_1$ and
  $f_2$ to the boundary. Let $\psi=\overline{f_{2*}}^{-1}\circ u\circ f_{1*}$ be
  the induced gluing map. This is an element of the strongly-based mapping class
  group on $F^\circ(\slz^g)$. Define:
  \[ \widehatit{HF}(Y,Y_1,Y_2,u,f_1,f_2) =
  \left(\widehatit{CFAA}(\psi)_{\sla(-\slz^g),\sla(\slz^g)} \boxtimes
    \tensor[^{\sla(-\slz^g)}]{\widehatit{CFD}(\hg)}{} \right) \boxtimes
  \tensor[^{\sla(\slz^g)}]{\widehatit{CFD}(-\hg)}{}. \]
\end{construction}

\begin{theorem} \label{thm:invclosed} The homotopy type of
  $\widehatit{HF}(Y,Y_1,Y_2,u,f_1,f_2)$ does not depend on the choice of
  Heegaard splitting $Y=Y_1\cup_u Y_2$ or the parametrizations
  $f_1,f_2$. Therefore it gives an invariant of $Y$ up to homotopy equivalence.
\end{theorem}

We will prove Theorem \ref{thm:invclosed} combinatorially in Section
\ref{sec:3manifold}. Given this theorem, we can write $\widehatit{HF}(Y)$ for
$\widehatit{HF}(Y,Y_1,Y_2,u,f_1,f_2)$, for some choice of Heegaard splitting and
parametrizations. From the construction, it is clear that this is equivalent to
the definition of $\widehatit{HF}(Y)$ using holomorphic curves.

%%% Local Variables:
%%% mode: latex
%%% TeX-master: "dacalc"
%%% End:

\section{Computations on $\DA$ invariants}
\label{sec:computationda}

In this section, we prove some preliminary results on type $\DA$ bimodules, and
perform some computations on the type $DA$ invariants of arcslides, in
preparation for the proof of Theorem \ref{thm:invmappingclass} in Section \ref
{sec:relmcg}.

First, we give an outline for the proof of Theorem \ref{thm:invmappingclass}. We
want to show that the combinatorial construction of $\hatda(\phi,\bm{\tau})$
does not depend on the choice of factorization of $\phi$ into arcslides
$\bm{\tau}$. For this purpose, it is necessary to understand relations among
arcslides. This is studied in detail in \cite{Bene} and \cite{ABP}. The notions
of pointed matched circles and arcslides correspond to linear chord diagrams and
chord slides in these papers. We now give a summary of the results.

Locally, an arcslide can be viewed as one end of the $B$ pair sliding along the
$C$ pair:
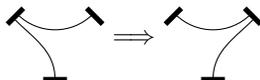
\begin{equation}
  \begin{threebase}
    \draw (0,-1.6) to [out=90,in=-45] (-1.05,-0.05);
    \draw (-0.95,0.05) to [out=-45,in=-135] (0.95,0.05);
  \end{threebase}
  \Longrightarrow
  \begin{threebase}
    \draw (0,-1.6) to [out=90,in=-135] (1.05,-0.05);
    \draw (-0.95,0.05) to [out=-45,in=-135] (0.95,0.05);
  \end{threebase} \nonumber
\end{equation}
In this diagram, the three short segments denote portions of the straight line
in the pointed matched circle. The upper, stationary arc denotes the $C$ pair;
and the lower, moving arc denotes the $B$ pair.

There are five types of relations on arcslides, and together they generate all
relations. The local diagrams for the five types of relations are as follows
(see \cite[Theorem 6.2, Figure 6.1]{Bene}):
\begin{itemize}
\item Triangle
  \begin{equation}
    \begin{threebase}
      \draw (-1.05,-0.05) to [out=-45,in=90] (-0.1,-1.6);
      \draw (1.05,-0.05) to [out=-135,in=90] (0.1,-1.6);
    \end{threebase}
    \Longrightarrow
    \begin{threebase}
      \draw (-1.05,-0.05) to [out=-45,in=90] (-0.1,-1.6);
      \draw (-0.95,0.05) to [out=-45,in=-135] (0.95,0.05);
    \end{threebase}
    \Longrightarrow
    \begin{threebase}
      \draw (-0.95,0.05) to [out=-45,in=-135] (0.95,0.05);
      \draw (1.05,-0.05) to [out=-135,in=90] (0.1,-1.6);
    \end{threebase}
    \Longrightarrow
    \begin{threebase}
      \draw (-1.05,-0.05) to [out=-45,in=90] (-0.1,-1.6);
      \draw (1.05,-0.05) to [out=-135,in=90] (0.1,-1.6);
    \end{threebase} \nonumber
  \end{equation}
\item Involution
  \begin{equation}
    \begin{threebase}
      \draw (-0.95,0.05) to [out=-45,in=-135] (0.95,0.05);
      \draw (1.05,-0.05) to [out=-135,in=90] (0.1,-1.6);
    \end{threebase}
    \Longrightarrow
    \begin{threebase}
      \draw (-1.05,-0.05) to [out=-45,in=90] (-0.1,-1.6);
      \draw (-0.95,0.05) to [out=-45,in=-135] (0.95,0.05);
    \end{threebase}
    \Longrightarrow
    \begin{threebase}
      \draw (-0.95,0.05) to [out=-45,in=-135] (0.95,0.05);
      \draw (1.05,-0.05) to [out=-135,in=90] (0.1,-1.6);
    \end{threebase} \nonumber
  \end{equation}
\item Commutativity
  \begin{eqnarray}&&
    \begin{threebase}
      \draw (-1.05,-0.05) to [out=-45,in=90] (-0.1,-1.6);
      \draw (1.05,-0.05) to [out=-135,in=90] (0.1,-1.6);
    \end{threebase}
    \begin{threebase}
      \draw (-1.05,-0.05) to [out=-45,in=90] (-0.1,-1.6);
      \draw (1.05,-0.05) to [out=-135,in=90] (0.1,-1.6);
    \end{threebase}
    \Rightarrow
    \begin{threebase}
      \draw (-0.95,0.05) to [out=-45,in=-135] (0.95,0.05);
      \draw (1.05,-0.05) to [out=-135,in=90] (0.1,-1.6);
    \end{threebase}
    \begin{threebase}
      \draw (-1.05,-0.05) to [out=-45,in=90] (-0.1,-1.6);
      \draw (1.05,-0.05) to [out=-135,in=90] (0.1,-1.6);
    \end{threebase}
    \Rightarrow
    \begin{threebase}
      \draw (-0.95,0.05) to [out=-45,in=-135] (0.95,0.05);
      \draw (1.05,-0.05) to [out=-135,in=90] (0.1,-1.6);
    \end{threebase}
    \begin{threebase}
      \draw (-0.95,0.05) to [out=-45,in=-135] (0.95,0.05);
      \draw (1.05,-0.05) to [out=-135,in=90] (0.1,-1.6);
    \end{threebase} \nonumber \\
    &\Rightarrow&
    \begin{threebase}
      \draw (-1.05,-0.05) to [out=-45,in=90] (-0.1,-1.6);
      \draw (1.05,-0.05) to [out=-135,in=90] (0.1,-1.6);
    \end{threebase}
    \begin{threebase}
      \draw (-0.95,0.05) to [out=-45,in=-135] (0.95,0.05);
      \draw (1.05,-0.05) to [out=-135,in=90] (0.1,-1.6);
    \end{threebase}
    \Rightarrow
    \begin{threebase}
      \draw (-1.05,-0.05) to [out=-45,in=90] (-0.1,-1.6);
      \draw (1.05,-0.05) to [out=-135,in=90] (0.1,-1.6);
    \end{threebase}
    \begin{threebase}
      \draw (-1.05,-0.05) to [out=-45,in=90] (-0.1,-1.6);
      \draw (1.05,-0.05) to [out=-135,in=90] (0.1,-1.6);
    \end{threebase} \nonumber
  \end{eqnarray}
\item Left Pentagon
  \begin{equation}
    \begin{fourbase}
      \draw (0,1) to (0,-1);
      \draw (-1,0) to [out=0,in=90] (-0.15,-1);
      \draw (1,0) to [out=180,in=90] (0.15,-1);
    \end{fourbase}
    \Rightarrow
    \begin{fourbase}
      \draw (0,1) to (0,-1);
      \draw (0.15,1) to [out=270,in=180] (1,0);
      \draw (-1,0) to [out=0,in=90] (-0.15,-1);
    \end{fourbase}
    \Rightarrow
    \begin{fourbase}
      \draw (-0.1,1) to [out=270,in=0] (-1,0.1);
      \draw (-1,-0.1) to [out=0,in=90] (-0.1,-1);
      \draw (0.1,1) to [out=270,in=180] (1,0.1);
    \end{fourbase}
    \Rightarrow
    \begin{fourbase}
      \draw (0,1) to [out=270,in=0] (-1,0.15);
      \draw (-1,0) to (1,0);
      \draw (-1,-0.15) to [out=0,in=90] (0,-1);
    \end{fourbase}
    \Rightarrow
    \begin{fourbase}
      \draw (-0.1,1) to [out=270,in=0] (-1,0.1);
      \draw (-1,-0.1) to [out=0,in=90] (-0.1,-1);
      \draw (0.1,-1) to [out=90,in=180] (1,-0.1);
    \end{fourbase}
    \Rightarrow
    \begin{fourbase}
      \draw (0,1) to (0,-1);
      \draw (-1,0) to [out=0,in=90] (-0.15,-1);
      \draw (1,0) to [out=180,in=90] (0.15,-1);
    \end{fourbase} \nonumber
  \end{equation}
\item Right Pentagon
  \begin{equation}
    \begin{fourbase}
      \draw (0,1) to (0,-1);
      \draw (-1,0) to [out=0,in=90] (-0.15,-1);
      \draw (1,0) to [out=180,in=90] (0.15,-1);
    \end{fourbase}
    \Rightarrow
    \begin{fourbase}
      \draw (0,1) to (0,-1);
      \draw (0.15,-1) to [out=90,in=180] (1,0);
      \draw (-1,0) to [out=0,in=270] (-0.15,1);
    \end{fourbase}
    \Rightarrow
    \begin{fourbase}
      \draw (-0.1,1) to [out=270,in=0] (-1,0.1);
      \draw (1,-0.1) to [out=180,in=90] (0.1,-1);
      \draw (0.1,1) to [out=270,in=180] (1,0.1);
    \end{fourbase}
    \Rightarrow
    \begin{fourbase}
      \draw (0,1) to [out=270,in=180] (1,0.15);
      \draw (-1,0) to (1,0);
      \draw (1,-0.15) to [out=180,in=90] (0,-1);
    \end{fourbase}
    \Rightarrow
    \begin{fourbase}
      \draw (0.1,-1) to [out=90,in=180] (1,-0.1);
      \draw (1,0.1) to [out=180,in=270] (0.1,1);
      \draw (-0.1,-1) to [out=90,in=0] (-1,-0.1);
    \end{fourbase}
    \Rightarrow
    \begin{fourbase}
      \draw (0,1) to (0,-1);
      \draw (-1,0) to [out=0,in=90] (-0.15,-1);
      \draw (1,0) to [out=180,in=90] (0.15,-1);
    \end{fourbase} \nonumber
  \end{equation}
\end{itemize}

Each relation gives one way to factor the identity morphism $\iz$ starting and
ending at some pointed matched circle $\slz$. For proving Theorem
\ref{thm:invmappingclass}, it suffices to check that for each such factorization
\[ \iz = \tau_n\circ\cdots\circ\tau_1, \] the corresponding homotopy equivalence
\begin{equation}\label{eq:tocheck}
  ^{\sla(\slz)}\mathbb{I}_{\sla(\slz)} \simeq
  \hatda(\tau_1)\boxtimes\cdots\boxtimes\hatda(\tau_n)
\end{equation}
holds. Note that in general, the starting and ending pointed matched circles of
each $\tau_i$ may be different from $\slz$. This is the main reason why we need
to consider strongly-based mapping class groupoids, even if we are only
interested in statements about strongly-based mapping class groups.

The overall strategy for verifying Equation (\ref{eq:tocheck}) is as
follows. From the description of $\hatda(\tau_i)$, we can readily enumerate the
set of generators on the right side of the equation. There are, however, more
generators on the right side than on the left side. The \emph{cancellation
  lemma} for type $\DA$ bimodules describes conditions under which we can prove
that a bimodule is homotopy equivalent to one with two fewer generators. Using
it, we can remove generators from the right side in pairs, so that the set of
remaining generators matches that on the left side. The cancellation lemma is
stated and proved in Section \ref{sec:cancellationlemma}.

It turns out that a type $\DA$ bimodule with the same set of generators as
$^{\sla(\slz)}\mathbb{I}_{\sla(\slz)}$, and with a few more properties in common
with $^{\sla(\slz)}\mathbb{I}_{\sla(\slz)}$, must be homotopy equivalent to
$^{\sla(\slz)}\mathbb{I}_{\sla(\slz)}$. We prove two lemmas of this kind, which
we call \emph{rigidity lemmas}, in Section \ref{sec:idrecognize}. The first
lemma will be used to prove the involution relation, and the second lemma will
be used for all other relations. The idea here is that once the involution
relation is proved, we can show that $\hatda(\tau)$ is quasi-invertible for any
arcslide $\tau$, which implies that any box tensor product of such bimodules is
also quasi-invertible (recall that a type $\DA$ bimodule $^AM_B$ is
\emph{quasi-invertible} if there exists $^BN_A$ such that $^AM_B\boxtimes
\tensor[^B]{N}{_A}\simeq \tensor[^A]{\mathbb{I}}{_A}$). This means checking the
quasi-invertibility condition in the second lemma becomes trivial, and we can
avoid checking the more involved condition in the first lemma that it
replaces. We note here that the rigidity lemmas depend on specific properties of
$\sla(\slz)$, and is not applicable to dg-algebras in general.

To apply the cancellation lemma, and in the case of the involution relation, the
rigidity lemma, we need to compute certain arrows in the type $\DA$ action of the
bimodule on the right side. To prepare for this, we review the construction of
$\hataa(\iz)$ in Section \ref{sec:cfaaconstruction}, and compute in Section
\ref{sec:daoperations} some arrows in the type $\DA$ action of $\hatda(\tau)$ for
arcslides $\tau$ (the components in the tensor product).

\subsection{Cancellation lemmas}\label{sec:cancellationlemma}

In this section we state cancellation lemmas for type $D$ modules and type $\DA$
bimodules over dg-algebras. Both are generalizations of the cancellation lemma
in the case of chain complexes. These results are well-known (see, for example,
\cite[Section 2.6]{Levine10}).

Let $A$ be a dg-algebra over a ground ring $\mathbf{k}$, where $\mathbf{k}$ is a
direct sum of copies of $\mathbb{F}_2$ (in our application, $A=\sla(\slz)$ and
$\mathbf{k}$ is generated by the indecomposable idempotents). Let $M$ be a left
type $D$ module over $A$ with a fixed set of generators $\mathcal{G}$. We can
describe the action $\delta^1$ on $M$ in terms of coefficients as follows: for
any $x\in\mathcal{G}$, expand $\delta^1(x)$ as
\[ \delta^1(x) = \sum_{\mathbf{y}\in\mathcal{G}} c_{xy} \otimes \mathbf{y}, \]
for some choice of $c_{xy}$. Here the tensor product is implicitly taken over
$\mathbf{k}$, and as a result, there is some flexibility in the choice of
$c_{xy}$. We generally choose $c_{xy}$ to consist of as few generators of $A$ as
possible, except when choosing $c_{xy}$ to be invertible whenever possible.

Now suppose that for some $a,b\in\mathcal{G}$, the coefficient $c_{ab}$ is
invertible in $A$, and $d(c_{ab})=0$. Then there is a new type $D$ module $M'$,
generated by $\mathcal{G}'=\mathcal{G}\setminus\{a,b\}$ and with type $D$ action
\begin{equation}
  \delta^{1'}(x) =
  \sum_{y\in\mathcal{G}'} (c_{xy} + c_{xb}c_{ab}^{-1}c_{ay})\otimes y,
\end{equation}
for any $x\in\mathcal{G}'$. The first part of each term in the sum is simply the
original $\delta^1$ (excluding terms involving $a$ and $b$). The second part is
as follows: for each ``zigzag'' in $M$ as shown in Figure \ref{fig:cancelex},
the term $c_{xb}c_{ab}^{-1}c_{ay}\otimes y$ is added to $\delta^{1'}{x}$. The
coefficient can be read out by ``following the arrows'' from $x$ to $y$,
treating a reversed arrow as taking inverse.

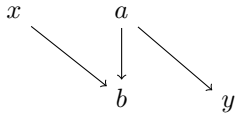
\begin{figure} [h!]
  \centering
  \begin{tikzpicture}
    \matrix (cancellation) [matrix of math nodes]
    {
      x & [1cm] a & [1cm]  \\ [7mm]
      & b & y \\
    };
    \draw [->] (cancellation-1-1) -- (cancellation-2-2);
    \draw [->] (cancellation-1-2) -- (cancellation-2-2);
    \draw [->] (cancellation-1-2) -- (cancellation-2-3);
  \end{tikzpicture}
  \caption{Standard example of a ``zigzag'' in $M$. This becomes $x\to
    c_{xb}c_{ab}^{-1}c_{ay}\otimes y$ in $M'$.}
  \label{fig:cancelex}
\end{figure}

\begin{theorem} [Cancellation lemma for type $D$ modules]
  With the above definitions, the action $\delta^{1'}$ on $M'$ satisfies the
  type $D$ structure equation, and the resulting type $D$ module $M'$ is
  homotopy equivalent to $M$.
  \label{thm:canceld}
\end{theorem}

\begin{proof}
  We prove this by giving explicit type $D$ morphisms and homotopies, in terms
  of coefficients as for the type $D$ action. The necessary data are morphisms
  $f:M\to M'$ and $g:M'\to M$, and homotopy $h:M\to M$, satisfying the
  identities $f\circ g=\mathbb{I}_{M'}$ and $g\circ f=\mathbb{I}_{M} +
  h\circ\delta^1 + \delta^1\circ h$. The morphisms $f$ and $g$ are given by
  \begin{equation}
    f(a)=0,~f(b)=\sum_{y\in\mathcal{G'}}c_{ab}^{-1}c_{ay}\otimes y,
    \mathrm{~and~} f(x)=1\otimes x \mathrm{~for~} x\in\mathcal{G}'.
  \end{equation}
  \begin{equation}
    g(x)=1\otimes x+c_{xb}c_{ab}^{-1}\otimes a.
  \end{equation}
  Part of $f$ can be visualized using the zigzag by following arrows from $b$ to
  $y$. Likewise, part of $g$ can be visualized by following arrows from $x$ to
  $a$.

  The homotopy $h:M\to M$ is given by $h(b)=c_{ab}^{-1}\otimes a$ and $h(x)=0$
  for any $x\neq b$.

  It remains to verify that $\delta^{1'}$ satisfies the type $D$ structure
  equations, and the maps $f, g$, and $h$ satisfy the required identities. This
  can be done by converting the equations into coefficient form. For example,
  the type $D$ structure equation
  \begin{equation}
    (\mu_2\otimes\mathbb{I}_M)\circ(\mathbb{I}_A\otimes\delta^1)\circ\delta^1
    +(\mu_1\otimes\mathbb{I}_M)\circ\delta^1=0, \nonumber
  \end{equation}
  when written in terms of coefficients, becomes
  \begin{equation}
    d(c_{xy})+\sum_{z\in\mathcal{G}} c_{xz}c_{zy}=0.
    \label{eq:strm}
  \end{equation}
  for any $x, y\in\mathcal{G}$.

  The structure equation for a type $D$ morphism $\phi:M\to N$ is:
  \begin{equation}
    (\mu_2\otimes\mathbb{I}_N)\circ(\mathbb{I}_A\otimes\delta^1_N)\circ\phi^1+
    (\mu_2\otimes\mathbb{I}_N)\circ(\mathbb{I}_A\otimes\phi^1)\circ\delta^1_M+
    (\mu_1\otimes\mathbb{I}_N)\circ\phi^1 = 0. \nonumber
  \end{equation}
  For any $x\in\mathcal{G}(M)$ and $y\in\mathcal{G}(N)$, let $\phi_{xy}$ be the
  coefficient of $y$ in $\phi^1(x)$. Then, applying the above equation to an
  arbitrary generator $x$ of $M$, we see that the structure equation is
  equivalent to
  \begin{equation}
    d(\phi_{xy})+\sum_{z'\in\mathcal{G}(N)} \phi_{xz'}c_{z'y,N}+\sum_{z\in\mathcal{G}(M)} c_{xz,M}\phi_{zy}=0
    \label{eq:dmorstr}
  \end{equation}
  for any $y\in\mathcal{G}(N)$.

  The composition of two morphisms $\phi:M\to N$ and $\psi:N\to P$ is given by:
  \[ (\psi\circ\phi)^1 =
  (\mu_2\otimes\mathbb{I}_P)\circ(\mathbb{I}_A\otimes\psi^1)\circ\phi^1. \] In
  terms of coefficients, this is:
  \begin{equation}
    (\psi\circ\phi)_{xy} = \sum_{z\in\mathcal{G}(N)} \phi_{xz}\psi_{zy},
  \end{equation}
  for any $x\in\mathcal{G}(M)$ and $y\in\mathcal{G}(P)$.

  It is then routine to verify these equations, using the assumption that
  $c_{ab}$ is invertible and $d(c_{ab})=0$.
\end{proof}

The cancellation lemma for type $\DA$ bimodules follows from that for type $D$
modules, by viewing type $\DA$ bimodules over $A'$ and $A$ as type $D$ modules
over $\cob(A)\otimes A'$ (see Remark 2.2.35 in \cite{LOT10a}).

\begin{definition}
  Given a strand algebra $A$, let $A_+$ be the sub-$dg$-algebra of $A$ generated
  by the non-idempotent generators. The cobar resolution $\mathrm{Cob}(A)$ is
  defined as $T^*(A_+[1]^*)$, the tensor algebra of the dual of $A_+$. This can
  be given the structure of a dg-algebra, with product being the one on the
  tensor algebra, and the differential consisting of the following arrows:
  \begin{itemize}
  \item $a_1^*\otimes\cdots\otimes b^*\otimes\cdots\otimes a_k^* \to
    a_1^*\otimes\cdots\otimes a_i^*\otimes\cdots\otimes a_k^*$, for each $i$ and
    term $b$ in $da_i$,
  \item $a_1^*\otimes\cdots\otimes a_i^*\otimes\cdots\otimes a_k^* \to
    a_1^*\otimes\cdots\otimes b^*\otimes b'^*\otimes\cdots\otimes a_k^*$, for
    each $i$ and generators $b, b'$ such that $bb'=a_i$.
  \end{itemize}
  Furthermore, we denote $\cob(A)$ to be the completion of $\mathrm{Cob}(A)$
  with respect to the length filtration (that is, an element of $\cob(A)$ is a
  formal sum of elements in $(A_+[1]^*)^{\otimes i}$ for possibly infinitely
  many $i$).
\end{definition}

The category of type $\DA$ bimodules over $A'$ on the $D$-side and $A$ on the
$A$-side is equivalent to the category of type $D$ modules over $\cob(A)\otimes
A'$, where the arrow
\[ \delta^1_{1+i}: (x; a_1, \dots, a_i) \to a'\otimes y \] in the action of a
type $\DA$ bimodule $M$ corresponds to the arrow
\[ \delta^1: x \to (a_1^*\otimes\cdots\otimes a_i^*)\otimes a'\otimes y \] in
the action of the type $D$ module corresponding to $M$.

Using this correspondence, we can define ``coefficients'' on a type $\DA$
bimodule:

\begin{definition}\label{def:dacoeff}
  Given two generators $x,y$ of $M$, define the coefficient $C_{xy}$ to be the
  formal sum, in $\cob(A')\otimes A$, of $(a_1^*\otimes\dots\otimes
  a_i^*)\otimes a'$ over all arrows of the form $\delta^1_{1+i}: (x;
  a_1,\dots,a_i)\to a'\otimes y$. As in the type $D$ case, we choose $a'$ to be
  invertible whenever possible when writing the action in terms of arrows.
\end{definition}

This allows us to state the cancellation lemma for type $\DA$ bimodules,
following immediately from the cancellation lemma in the type $D$ case, and the
equivalence of categories.

\begin{theorem} [Cancellation lemma for type $\DA$ bimodules]
  Let $^{A'}M_A$ be a type $\DA$ bimodule, with a fixed set $\mathcal{G}$ of
  generators. Suppose there are $x,y\in\mathcal{G}$ such that $C_{xy}=1\otimes
  a$ with $a\in A$ invertible and $da=0$. Then $C_{xy}^{-1}=1\otimes a^{-1}$,
  and the type $\DA$ bimodule $M'$ generated by
  $\mathcal{G}'=\mathcal{G}\setminus\{x,y\}$, and with coefficients
  $C'_{ab}=C_{ab}+C_{ay}C_{xy}^{-1}C_{xb}$, is homotopy equivalent to $M$.
  \label{thm:cancelda}
\end{theorem}

We end with a remark on grading. If $M$ is graded by a grading set $S_M$, and if
every generator being cancelled is homogeneous in grading, then $M'$ is also
graded by $S_M$, with the grading of each generator in $M'$ equal to the grading
of the corresponding generator in $M$. The arrows that are added to $M'$ satisfy
the grading constraints, because they come from traversing a zig-zag as in
Figure \ref{fig:cancelex}, where each of the three arrows in the zig-zag satisfy
the grading constraints. The homogeneity condition of the cancelled generators
will be automatically satisfied in our case.

\subsection{Characterization of the identity bimodule}
\label{sec:idrecognize}

In this section, we prove two lemmas describing conditions under which we can
assert a type $\DA$ bimodule $^{\sla(\slz)}M_{\sla(\slz)}$ is homotopy
equivalent to the identity bimodule $^{\sla(\slz)}\mathbb{I}_{\sla(\slz)}$. The
main result we use is the characterization of $\hatdd(\iz)$ given in
\cite{LOT10c}. We will start by reviewing that result here.

\begin{definition}[{\cite[Definition 3.1]{LOT10c}}]\label{def:diagonalalg}
  The \emph{diagonal subalgebra} of $\sla(\slz)\otimes\sla(-\slz)$ is the
  algebra generated by $a\otimes b$, where $a$ and $b$ satisfy the following
  conditions: $\mathrm{mult}(a)=\mathrm{mult}(\overline{b})$, the left
  idempotents of $a$ and $b$ are complementary, and the right idempotents of $a$
  and $b$ are complementary.
\end{definition}

\begin{proposition}[{\cite[Proposition 3.8, Proof of Theorem 1]{LOT10c}}]
  \label{thm:ddrecognize}
  Let $M$ be a left type $\DD$ bimodule over $\sla(\slz)$ and $\sla(-\slz)$,
  where $\slz$ has genus greater than one. Suppose $M$ satisfies the following
  conditions, then $M$ is isomorphic to $\hatdd(\iz)$.
  \begin{enumerate}
  \item The generators of $M$ are in one-to-one correspondence with the
    idempotents of $\sla(\slz)$, so that the generator corresponding to
    idempotent $i$ has (type $D$) idempotents $i$ and $\overline{o(i)}$.
  \item For any arrow $\mathbf{x}\to a\otimes b\otimes\mathbf{y}$ in the
    differential of $M$, the element $a\otimes b$ lies in the diagonal
    subalgebra.
  \item $M$ is graded by a $\lambda$-free grading set $S$, with a left-right
    $G(\slz)$-$G(\slz)$ action.
  \item The differential in $M$ contains all arrows of the form
    \[ \mathbf{x}\to a(\xi)\otimes\overline{a(\xi)}\otimes\mathbf{y}, \]
    where $\xi$ is a length-1 chord.
  \end{enumerate}
  In the case where $\slz$ has genus one, if $M$ satisfies an additional
  stability condition, in the sense of \cite[Definition 1.8]{LOT10c}, then we
  can still conclude that $M=\hatdd(\iz)$.
\end{proposition}

The following result will be used in the proof of the second lemma:
\begin{proposition}\label{prop:idtoexistshortaction}
  Suppose a type $DA$ bimodule $\tensor[^{\sla(\slz)}]{M}{_{\sla(\slz)}}$
  satisfies the following two conditions:
  \begin{enumerate}
  \item $M$ is homotopy equivalent to the identity bimodule
    $^{\sla(\slz)}\mathbb{I}_{\sla(\slz)}$.
  \item The generators of $M$ are in one-to-one correspondence with the
    idempotents of $\sla(\slz)$, so that the generator corresponding to
    idempotent $i$ has both left (type $D$) and right (type $A$) idempotent
    equal to $i$.
  \end{enumerate}
  Then the type $\DA$ action on $M$ contains all arrows of the form
  \begin{equation}\label{eq:arrowtoshowexist}
    \delta_2^1: (\mathbf{x}, a(\xi)) \to a(\xi) \otimes \mathbf{y},
  \end{equation}
  where $\xi$ is a length-1 chord.
\end{proposition}
\begin{proof}
  Consider generators $\mathbf{x},\mathbf{y}$ corresponding to idempotents
  $i,j\in\sla(\slz)$, and $\xi$ a length-1 chord, such that the idempotent
  matches in the arrow (\ref{eq:arrowtoshowexist}). We want to show that
  (\ref{eq:arrowtoshowexist}) does exist as an arrow.

  Let $T_D$ be a type $D$ module over $\sla(\slz)$ with two generators
  $\mathbf{x}_D$ and $\mathbf{y}_D$, whose idempotents are $i$ and $j$, such
  that $\delta^1(\mathbf{x}_D)=a(\xi)\otimes\mathbf{y}_D$ and
  $\delta^1(\mathbf{y}_D)=0$. Since $d(a(\xi))=0$, it is clear that $\delta^1$
  satisfies the type $D$ structure equation.

  Likewise, let $T_A$ be the $A_\infty$-module over $\sla(\slz)$ with two
  generators $\mathbf{x}_A$ and $\mathbf{y}_A$ whose idempotents are $i$ and
  $j$, and $m_{1,1}: (\mathbf{x}_A,a(\xi))\to \mathbf{y}_A$ is the only arrow in
  the $A_\infty$ action.

  Consider the tensor product $T_A\boxtimes N\boxtimes T_D$, with $N=M$ or
  $N=\mathbb{I}$. This is a chain complex with two generators
  $\mathbf{x}_A\otimes\mathbf{x}\otimes\mathbf{x}_D$ and
  $\mathbf{y}_A\otimes\mathbf{y}\otimes\mathbf{y}_D$, and there is an arrow
  between these two if and only if the arrow (\ref{eq:arrowtoshowexist}) exists
  in $N$ for the given $\mathbf{x},\mathbf{y}$ and $a(\xi)$. In particular,
  $T_A\boxtimes\mathbb{I}\boxtimes T_D$ has zero homology. By assumption,
  $M\simeq\mathbb{I}$, so $T_A\boxtimes M\boxtimes T_D$ must also have zero
  homology. This shows the arrow (\ref{eq:arrowtoshowexist}) exists in $M$.
\end{proof}

\paragraph{Note:} The argument in the above proof only works when $\xi$ has
length 1. If otherwise, we may have $d(a(\xi))\neq 0$, and $\delta^1$ on $T_D$
no longer satisfies the type $D$ structure equation. Indeed, in the case where
$\xi$ has length 2, we may have:
\[
d\left(\begin{sdn}{3}\strandup{1}{3}\singlehor{2}\end{sdn}\right) =
\begin{sdn}{3}\strandup{1}{2}\strandup{2}{3}\end{sdn} =
\begin{sdn}{3}\strandup{2}{3}\singlehor{1}\end{sdn} \cdot
\begin{sdn}{3}\strandup{1}{2}\singlehor{3}\end{sdn}.
\]
Hence, if there are generators $\mathbf{x}_D$ and $\mathbf{y}_D$ in $T_D$ with
arrow $\mathbf{x}_D\to a(\xi)\otimes\mathbf{y}_D$, where the idempotents of
$\mathbf{x}_D$ and $\mathbf{y}_D$ contain the middle point, then there must be
an additional generator $\mathbf{z}_D$ with appropriate arrows from
$\mathbf{x}_D$ to $\mathbf{z}_D$ and from $\mathbf{z}_D$ to $\mathbf{y}_D$, so
that the type $D$ structure equation remains satisfied. This is why we may have,
for example, arrow (\ref{eq:dalen2.3}) instead of (\ref{eq:dalen2.2}) in
$^{\sla(\slz)}M_{\sla(\slz)}$, according the computations in Section
\ref{subsec:typeaaid}.

We now state and prove the lemmas on the characterization of
$^{\sla(\slz)}\mathbb{I}_{\sla(\slz)}$.

\begin{lemma}\label{lem:idrecognize1}
  Let $M={^{\sla(\slz)}M_{\sla(\slz)}}$ be a left-right type $\DA$ bimodule over
  $\sla(\slz)$-$\sla(\slz)$.  Suppose $M$ satisfies the following properties,
  then $M$ is homotopy equivalent to the identity bimodule
  $^{\sla(\slz)}\mathbb{I}_{\sla(\slz)}$.
  \begin{itemize}
  \item (ID-1). The generators of $M$ are in one-to-one correspondence with the
    idempotents of $\sla(\slz)$, so that the generator corresponding to
    idempotent $i$ has both left (type $D$) and right (type $A$) idempotent
    equal to $i$.
  \item (ID-2). $M$ can be graded by a principal left-right
    $G(\mathcal{Z})$-$G(\mathcal{Z})$ set, such that the induced map
    $\phi\in\mathrm{Out}(G(\mathcal{Z}),G(\mathcal{Z}))$ (as in \cite[Lemma
    6.4]{LOT10c}) is the identity map, and there is a choice of refined relative
    grading with every generator having grading zero. (The choice of grading
    refinement for $G(\mathcal{Z})$ is arbitrary but must be the same on both
    sides).
  \item (ID-3). The type $\DA$ action on $M$ contains all arrows of the form
    \[ \delta_2^1: (\mathbf{x}, a(\xi)) \to a(\xi) \otimes \mathbf{y}, \] where
    $\xi$ is a length-1 chord.
  \item (ID-4). $M$ is stable in the sense of \cite[Definition 1.8]{LOT10c}
    (this condition is only necessary when $\mathcal{Z}$ is the unique genus-1
    pointed matched circle).
  \end{itemize}
\end{lemma}
\begin{proof}
  Consider the type $\DD$ bimodule $M_{\DD}=M\boxtimes\hatdd(\iz)$. We check
  that $M_{\DD}$ satisfies all the conditions of Proposition
  \ref{thm:ddrecognize}, which would show that $M_{\DD}$ is isomorphic to
  $\hatdd(\iz)$. Since $\hatdd(\iz)$ is quasi-invertible, this implies
  $M\simeq\mathbb{I}$.

  Using the fact that relative grading can be chosen on $\hatdd(\iz)$ so that
  every generator has grading zero, Condition (ID-2) on the grading of $M$
  implies a similar condition on the grading of $M_{\DD}$. The constraint that
  the type $\DD$ action must respect the grading implies the following: for each
  arrow
  \[ \mathbf{x}\to a\otimes b\otimes\mathbf{y}, \] the multiplicities of $a$ and
  $\overline{b}$ must be the same. The idempotent conditions on the diagonal
  subalgebra follow from the constraints on idempotents on each arrow, and the
  fact that both $\mathbf{x}$ and $\mathbf{y}$ have complementary
  idempotents. This verifies condition (2) in Proposition \ref{thm:ddrecognize}.

  The other deductions are trivial. (ID-1), (ID-2), and (ID-3) imply conditions
  (1), (3), and (4), respectively. Condition (ID-4) implies the stability of
  $M_{\DD}$, needed for the genus 1 case.
\end{proof}

The condition (ID-3) in the previous lemma can still be difficult to verify in
actual computations. It is possible to replace it as follows.

\begin{lemma}\label{lem:idrecognize2}
  With the same notation as in Lemma \ref{lem:idrecognize1}, if $M$ satisfies
  the conditions (ID-1), (ID-2), (ID-4), and the following condition, then it is
  homotopy equivalent to $\mathbb{I}$.
  \begin{itemize}
  \item (ID-3') $M$ is invertible, with a quasi-inverse $M'$ that satisfies
    conditions (ID-1) and (ID-2).
  \end{itemize}
\end{lemma}
\begin{proof}
  It suffices to show that (ID-3'), together with the other conditions, imply
  (ID-3). We first show that $\delta_1^1=0$ on both $M$ and $M'$. That is, there
  are no arrows of the form
  \[ \delta_1^1: \mathbf{x} \to a\otimes \mathbf{y}. \] By the grading
  constraints on any arrow, the algebra generator $a$ must have multiplicity
  zero. That is, it must be an idempotent in $\sla(\slz)$. However, this would
  mean that the grading of $\mathbf{x}$ and $\mathbf{y}$ differ by $\lambda$,
  contradicting the assumption that all generators in $M$ (or $M')$ have grading
  zero.

  Both $M$ and its quasi-inverse $M'$ also satisfy condition (ID-1), so by
  \cite[Lemma 2.2.50]{LOT10a}, they can be represented as
  $^{\sla(\slz)}[\phi]_{\sla(\slz)}$ and $^{\sla(\slz)}[\phi']_{\sla(\slz)}$
  respectively, for $A_\infty$-algebra morphisms
  $\phi,\phi':\sla(\slz)\to\sla(\slz)$. Then $M'\boxtimes M$ is represented by
  $^{\sla(\slz)}[\phi'\circ\phi]_{\sla(\slz)}$.

  Since $M'\boxtimes M$ satisfies the grading condition (ID-2), the map
  $\phi'\circ\phi$ must preserve gradings. This means that for $a=a(\xi)$ where
  $\xi$ is a length-1 chord, the only possible term in $(\phi'\circ\phi)(a)$ is
  $a$. Since $M\boxtimes M'$ is homotopy equivalent to identity, by Proposition
  \ref{prop:idtoexistshortaction}, we have $(\phi'\circ\phi)(a)=a$, which
  implies $\phi(a)\neq 0$. By the same grading argument, either $\phi(a)=0$ or
  $\phi(a)=a$. So we must have $\phi(a)=a$. This shows $\phi$ is the identity
  map on length-1 chords, which implies condition (ID-3).
\end{proof}

\subsection{Combinatorial model of
  $\widehatit{CFAA}(\iz)$}\label{sec:cfaaconstruction}
In this section, we review the construction of the combinatorial model
$\hataa(\iz)$ of $\widehatit{CFAA}(\iz)$ given in \cite{BZ1}, in preparation for
computing some arrows in $\hatda(\tau)$ for arcslides $\tau$ in the next
section.

The construction begins with Equation (\ref{eq:typeaaid}). After expanding the
definitions, this gives a model of $\widehatit{CFAA}(\iz)$ generated by the set
of pairs $[a_1,a_2]$, where $a_1$ and $a_2$ are generators of $\sla(\slz)$, such
that the initial idempotents of $a_1$ and $a_2$ are complementary. The
differential and type $\AA$ action on these generators are given as Proposition
1 in \cite{BZ1}. The smaller model $\hataa(\iz)$ is obtained from this using
homological perturbation theory. This involves finding the homology $C'$ of $C$,
the chain complex underlying the larger model, and giving chain maps $f:C\to
C'$, $g:C'\to C$, and homotopy $H:C\to C$ verifying the homotopy equivalence
between $C$ and $C'$. The homology $C'$ is generated by those $[a_1,a_2]$ where
both $a_1$ and $a_2$ are idempotents (which are then complementary). The chain
maps $f$ and $g$ are the obvious projection and inclusion maps. The homotopy $H$
is summarized in Figures 6 and 9 in \cite{BZ1}.

From homological perturbation theory, we obtain the following description of the
smaller model: the $A_\infty$-bimodule $\hataa(\iz)_{\sla(-\slz),\sla(\slz)}$ is
generated by pairs of complementary idempotents $i'\otimes i$, where
$i\in\sla(\slz)$ and $i'=\overline{o(i)}\in\sla(-\slz)$. The generator
$i'\otimes i$ has type $A$ idempotents $i'$ and $i$. Each arrow in the type
$\AA$ action of $\hataa(\iz)$ comes from a sequence of ``moves'' between
generators of $C$. There are three types of moves, the first two of which carry
a coefficient.
\begin{itemize}
\item Move $A_1$: If $c\overline{b'}\neq 0$, with $b'\in\sla(-\slz)$, move from
  $[c\overline{b'},a_2]$ to $[c,a_2]$ with coefficient $b'$.
\item Move $A_2$: If $a_2b\neq 0$, with $b\in\sla(\slz)$, move from $[a_1,a_2]$
  to $[a_1,a_2b]$ with coefficient $b$.
\item Move $H$: Apply one of the arrows in the homotopy map $H$.
\end{itemize}

Each arrow then corresponds to a sequence
$[a_{1,1},a_{1,2}],\dots,[a_{2n,1},a_{2n,2}]$ of generators of $C$, satisfying
the following conditions:
\begin{itemize}
\item $[a_{1,1},a_{1,2}]=[o(i),i]$, and $[a_{2n,1},a_{2n,2}]=[o(j),j]$, for some
  idempotents $i,j\in\sla(\slz)$.
\item Each $[a_{2k,1},a_{2k,2}]$ is obtained from $[a_{2k-1,1},a_{2k-1,2}]$ by
  applying either move $A_1$ or $A_2$.
\item Each $[a_{2k+1,1},a_{2k+1,2}]$ is obtained from $[a_{2k,1},a_{2k,2}]$ by
  applying move $H$.
\end{itemize}

Let $b_1',\dots,b_p'$ be the ordered sequence of coefficients for moves of type
$A_1$, and $b_1,\dots,b_q$ be the ordered sequence of coefficients for moves of
type $A_2$, then such a sequence of generators of $C$ gives rise to an arrow
\[ m_{1,p,q}: (i'\otimes i;b_1',\dots,b_p';b_1,\dots,b_q) \to j'\otimes j, \]
where $i'=\overline{o(i)}$ and $j'=\overline{o(j)}$.

An important property of $\hataa(\iz)$ which follows directly from this
construction is that for any arrow in the type $\AA$ action, the total
multiplicity of the $\sla(-\slz)$ inputs (that is, the sum of multiplicities of
$b_1',\dots,b_p'$) equals that of the $\sla(\slz)$ inputs (the sum of
multiplicities of $b_1,\dots,b_q$). From the definition using holomorphic
curves, this is clear since each arrow comes from a domain in the standard
Heegaard diagram of the identity diffeomorphism. We also note that $\hataa(\iz)$
can be given a refined relative grading where all generators have grading zero.

The definition of the homotopy map $H$ involves first defining a specific
ordering $<_{\slz}$ on the $4g-1$ intervals of the pointed matched circle
$\slz$. This means the determination of arrows is not ``local'', in the sense
that if we restrict to a certain interval of $\slz$, containing points paired
outside the interval, then the type $\AA$ arrows restricted to that interval may
depend on how $\slz$ is configured outside the interval. However, we note that
if all points are paired within the interval, then the ordering $<_{\slz}$ on
these points (and therefore the type $\AA$ arrows) is independent of outside
configurations (this follows directly from how the ordering $<_{\slz}$ is
defined). In particular, $\hataa(\mathbb{I}_\slz)$ behaves well with respect to
stabilization. That is, if $\mathring{\slz}=\slz\#\slz^1$, then
$\hataa(\mathbb{I}_\slz)$ is isomorphic to the appropriate restriction of
$\hataa(\mathbb{I}_{\mathring{\slz}})$.

%%% Local Variables:
%%% mode: latex
%%% TeX-master: "dacalc"
%%% End:

\subsection{Certain arrows in $\hatda$ of arcslides}
\label{sec:daoperations}

In this section we compute some of the arrows in $\hatda(\tau)$ for a general
arcslide $\tau$, using Equation (\ref{eq:arcslidedadef}). From its description
in the previous section, one can expect arrows in $\hataa(\iz)$ to be extremely
complicated in general. The same would then be true for arrows in
$\hatda(\tau)$. We manage this complexity by focusing only on arrows whose
algebra coefficients have a small total length (say length 1 or 2 on each
side). It turns out that these are sufficient to prove the necessary properties
of the box tensor products of $\hatda(\tau_i)$ that we will need to consider.

Since the algebra coefficients have small total length, the domain corresponding
to the arrow is supported in a small part of the Heegaard diagram. For
arcslides, the parts of the Heegaard diagram that we are particularly interested
in are the differences with the Heegaard diagram for the identity diffeomorphism
-- that is, around the points $b_1,c_1,c_2$ and $b_1'$.

One source of complexity comes from the fact that the definition of the homotopy
map $H$ in the construction of $\hataa(\iz)$ depends on the ordering
$<_\mathcal{Z}$ on the intervals of the pointed matched circle. In a local
situation, if we cannot tell which interval comes first in the ordering, we will
need to cover all possible cases. Note only the restriction of $<_\mathcal{Z}$
to the intervals covered by the algebra coefficients matter for determining the
arrows.

When we show a set of local arrows in a given local situation and restriction of
the ordering $<_\mathcal{Z}$, we intend to make the following assertions:
\begin{itemize}
\item There is an arrow for every way of extending the local arrow by completing
  the pointed matched circle and adding the appropriate number of horizontal
  lines to the algebra coefficients.
\item Every arrow in the bimodule action whose algebra coefficients lie within
  the area shown can be obtained by extending one of the local arrows.
\end{itemize}

We now begin with the simplest case: arrows in the type $\AA$ action on
$\hataa(\iz)$ where the algebra coefficients have length 1 on either side. The
coefficients must then cover the same interval. The sequence of pairs
$[a_{i,1},a_{i,2}]$ is:

\begin{equation}
  \begin{dsdn}{2}\lsinglehor{2}\rsinglehor{1}\end{dsdn}\overto{A_2}
  \begin{dsdn}{2}\lsinglehor{2}\rstrandup{1}{2}\end{dsdn}\overto{H}
  \begin{dsdn}{2}\lstrandup{1}{2}\rsinglehor{2}\end{dsdn}\overto{A_1}
  \begin{dsdn}{2}\lsinglehor{1}\rsinglehor{2}\end{dsdn}, \nonumber
\end{equation}
where the middle $H$ is the Case 3 of the homotopy map in the multiplicity-one
case given in \cite{BZ1}. This gives the arrow:
\begin{aaequation}{equation} \label{eq:aalen1}
  m_{1,1,1}: \left(
    \left[\begin{dsdn}{2}\lsinglehor{2}\rsinglehor{1}\end{dsdn}\right];
    \begin{sdn}{2}{7pt}\stranddown{1}{2}\end{sdn}~;~
    \begin{sdn}{2}{7pt}\strandup{1}{2}\end{sdn}\right) \to
  \left[\begin{dsdn}{2}\lsinglehor{1}\rsinglehor{2}\end{dsdn}\right].
\end{aaequation}

From (\ref{eq:aalen1}), we obtain a simple method of deriving arrows in
$\hatda(\tau)$ from arrows in $\hatdd(\tau)$, in cases where the second
coefficient of the type $\DD$ arrow has length 1 (the second algebra action is
the one that is involved in the box tensor product). For each type $\DD$ arrow
$\delta^1: \mathbf{x}\to a_1\otimes a_2\otimes \mathbf{y}$, where $a_2$ has
length 1, there corresponds a type $\DA$ arrow $\delta_2^1:
(\mathbf{x},\overline{a_2})\to a_1\otimes\mathbf{y}$, where by abuse of notation
we used the same symbol to denote corresponding generators of $\hatda(\tau)$ and
$\hatdd(\tau)$.

As an application, we give a combinatorial proof of the following:
\begin{corollary}\label{cor:ddtimesaais1}
  The tensor product $\hataa(\iz)\boxtimes\hatdd(\iz)$ is homotopy equivalent to
  $\mathbb{I}$.
\end{corollary}
\begin{proof}
  Directly check each of the conditions in Lemma \ref{lem:idrecognize1}. For
  Condition (ID-2), we use the refined relative grading on $\hataa(\iz)$ with
  all generators having grading zero. For Condition (ID-3), use the type $\AA$
  arrows computed here. For (ID-4), use the stabilization property of
  $\hataa(\iz)$ discussed at the end of Section \ref{sec:cfaaconstruction}.
\end{proof}

\subsubsection{Type $\AA$ on a size 2 interval, disjoint pairs case}
\label{subsec:typeaaid}
The next simplest case for type $\AA$ arrows is the size 2 interval. First, we
assume that no two of the three points are paired with each other. There are
four subcases, depending on whether the middle idempotent is occupied on the
left or on the right, and whether the lower or the upper interval comes first in
the ordering $<_\mathcal{Z}$.

\paragraph{Middle idempotent on the left, lower interval first in ordering.}
The only sequence covering the size 2 interval is:
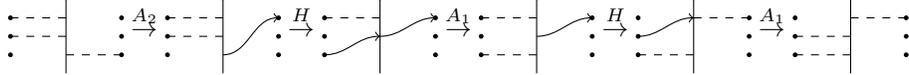
\begin{equation}
  \begin{dsdn}{3}\lsinglehor{2}\lsinglehor{3}\rsinglehor{1}\end{dsdn}\overto{A_2}
  \begin{dsdn}{3}\lsinglehor{2}\lsinglehor{3}\rstrandup{1}{3}\end{dsdn}\overto{H}
  \begin{dsdn}{3}\lsinglehor{3}\lstrandup{1}{2}\rstrandup{2}{3}\end{dsdn}\overto{A_1}
  \begin{dsdn}{3}\lsinglehor{1}\lsinglehor{3}\rstrandup{2}{3}\end{dsdn}\overto{H}
  \begin{dsdn}{3}\lsinglehor{1}\lstrandup{2}{3}\rsinglehor{3}\end{dsdn}\overto{A_1}
  \begin{dsdn}{3}\lsinglehor{1}\lsinglehor{2}\rsinglehor{3}\end{dsdn}, \nonumber
\end{equation}
giving the arrow:
\begin{aaequation}{equation} \label{eq:aalen2.1}
  m_{1,2,1}: \left(
    \left[\begin{dsdn}{3}\lsinglehor{2}\lsinglehor{3}\rsinglehor{1}\end{dsdn}\right];
    \begin{sdn}{3}\stranddown{1}{2}\singlehor{3}\end{sdn},
    \begin{sdn}{3}\stranddown{2}{3}\singlehor{1}\end{sdn}~;~
    \begin{sdn}{3}\strandup{1}{3}\end{sdn}\right) \to
  \left[\begin{dsdn}{3}\lsinglehor{1}\lsinglehor{2}\rsinglehor{3}\end{dsdn}\right].
\end{aaequation}
Note in the first $H$ move, we shift only the lower part of the strand to the
left, since the lower interval comes first in the ordering.

\paragraph{Middle idempotent on the left, upper interval first in ordering.}
The only sequence covering the size 2 interval is:
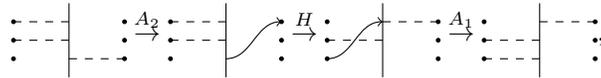
\begin{equation}
  \begin{dsdn}{3}\lsinglehor{2}\lsinglehor{3}\rsinglehor{1}\end{dsdn}\overto{A_2}
  \begin{dsdn}{3}\lsinglehor{2}\lsinglehor{3}\rstrandup{1}{3}\end{dsdn}\overto{H}
  \begin{dsdn}{3}\lsinglehor{2}\lstrandup{1}{3}\rsinglehor{3}\end{dsdn}\overto{A_1}
  \begin{dsdn}{3}\lsinglehor{1}\lsinglehor{2}\rsinglehor{3}\end{dsdn}, \nonumber
\end{equation}
giving the arrow:
\begin{aaequation}{equation} \label{eq:aalen2.2}
  m_{1,1,1}: \left(
    \left[\begin{dsdn}{3}\lsinglehor{2}\lsinglehor{3}\rsinglehor{1}\end{dsdn}\right];
    \begin{sdn}{3}\singlehor{2}\stranddown{1}{3}\end{sdn}~;~
    \begin{sdn}{3}\strandup{1}{3}\end{sdn}\right) \to
  \left[\begin{dsdn}{3}\lsinglehor{1}\lsinglehor{2}\rsinglehor{3}\end{dsdn}\right].
\end{aaequation}
Here the upper interval comes first, so we shift the entire strand to the left
in the first $H$ move.

\paragraph{Middle idempotent on the right, upper interval first in ordering.}
In this case there are two possible sequences covering the size 2 interval. The
first one is,
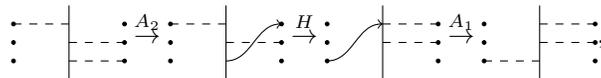
\begin{equation}
  \begin{dsdn}{3}\lsinglehor{3}\rsinglehor{1}\rsinglehor{2}\end{dsdn}\overto{A_2}
  \begin{dsdn}{3}\lsinglehor{3}\rsinglehor{2}\rstrandup{1}{3}\end{dsdn}\overto{H}
  \begin{dsdn}{3}\lstrandup{1}{3}\rsinglehor{2}\rsinglehor{3}\end{dsdn}\overto{A_1}
  \begin{dsdn}{3}\lsinglehor{1}\rsinglehor{2}\rsinglehor{3}\end{dsdn}, \nonumber
\end{equation}
giving the arrow:
\begin{aaequation}{equation} \label{eq:aalen2.3}
  m_{1,1,1}: \left(
    \left[\begin{dsdn}{3}\lsinglehor{3}\rsinglehor{1}\rsinglehor{2}\end{dsdn}\right];
    \begin{sdn}{3}\stranddown{1}{3}\end{sdn}~;~
    \begin{sdn}{3}\strandup{1}{3}\singlehor{2}\end{sdn}\right) \to
  \left[\begin{dsdn}{3}\lsinglehor{1}\rsinglehor{2}\rsinglehor{3}\end{dsdn}\right].
\end{aaequation}

The second one is:
\begin{equation}
  \begin{dsdn}{3}\lsinglehor{3}\rsinglehor{1}\rsinglehor{2}\end{dsdn}\overto{A_2}
  \begin{dsdn}{3}\lsinglehor{3}\rstrandup{1}{2}\rstrandup{2}{3}\end{dsdn}\overto{H}
  \begin{dsdn}{3}\lstrandup{2}{3}\rsinglehor{3}\rstrandup{1}{2}\end{dsdn}\overto{A_1}
  \begin{dsdn}{3}\lsinglehor{2}\rsinglehor{3}\rstrandup{1}{2}\end{dsdn}\overto{H}
  \begin{dsdn}{3}\lstrandup{1}{2}\rsinglehor{2}\rsinglehor{3}\end{dsdn}\overto{A_1}
  \begin{dsdn}{3}\lsinglehor{1}\rsinglehor{2}\rsinglehor{3}\end{dsdn}, \nonumber
\end{equation}
giving the arrow:
\begin{aaequation}{equation} \label{eq:aalen2.4}
  m_{1,2,1}: \left(
    \left[\begin{dsdn}{3}\lsinglehor{3}\rsinglehor{1}\rsinglehor{2}\end{dsdn}\right];
    \begin{sdn}{3}\stranddown{2}{3}\end{sdn},
    \begin{sdn}{3}\stranddown{1}{2}\end{sdn}~;~
    \begin{sdn}{3}\strandup{1}{2}\strandup{2}{3}\end{sdn}\right) \to
  \left[\begin{dsdn}{3}\lsinglehor{1}\rsinglehor{2}\rsinglehor{3}\end{dsdn}\right].
\end{aaequation}

\paragraph{Middle idempotent on the right, lower interval first in ordering.}
The only sequence covering the size 2 interval is:
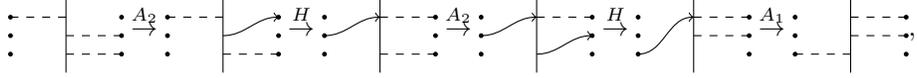
\begin{equation}
  \begin{dsdn}{3}\lsinglehor{3}\rsinglehor{1}\rsinglehor{2}\end{dsdn}\overto{A_2}
  \begin{dsdn}{3}\lsinglehor{3}\rsinglehor{1}\rstrandup{2}{3}\end{dsdn}\overto{H}
  \begin{dsdn}{3}\lstrandup{2}{3}\rsinglehor{1}\rsinglehor{3}\end{dsdn}\overto{A_2}
  \begin{dsdn}{3}\lstrandup{2}{3}\rsinglehor{3}\rstrandup{1}{2}\end{dsdn}\overto{H}
  \begin{dsdn}{3}\lstrandup{1}{3}\rsinglehor{2}\rsinglehor{3}\end{dsdn}\overto{A_1}
  \begin{dsdn}{3}\lsinglehor{1}\rsinglehor{2}\rsinglehor{3}\end{dsdn}, \nonumber
\end{equation}
giving the arrow:
\begin{aaequation}{equation} \label{eq:aalen2.5}
  m_{1,1,2}: \left(
    \left[\begin{dsdn}{3}\lsinglehor{3}\rsinglehor{1}\rsinglehor{2}\end{dsdn}\right];
    \begin{sdn}{3}\stranddown{1}{3}\end{sdn}~;~
    \begin{sdn}{3}\strandup{2}{3}\singlehor{1}\end{sdn},
    \begin{sdn}{3}\strandup{1}{2}\singlehor{3}\end{sdn}\right) \to
  \left[\begin{dsdn}{3}\lsinglehor{1}\rsinglehor{2}\rsinglehor{3}\end{dsdn}\right].
\end{aaequation}

As examples, we show the computation of type $\DA$ arrows in
$\hataa(\iz)\boxtimes\hatdd(\iz)$ that cover a size 2 interval. While the
results in the remainder of this section will not be used directly in what
follows, it serves as a model for the calculations of similar arrows in
$\hatda(\tau)$ for an arcslide $\tau$.

To compute the type $\DA$ arrows, we combine the previous results with what is
known about type $\DD$ arrows in $\hatdd(\iz)$. On the size 2 interval, the
possibilities are given below. On each line, $\delta^1: \mathbf{x}\to
(a,a')\otimes\mathbf{y}$ represents the arrow $\delta^1: \mathbf{x}\to a\otimes
a'\otimes\mathbf{y}$, where $a\in\sla(\slz)$ and $a'\in\sla(-\slz)$.

\begin{ddequation}{alignat}{1}
  \delta^1: \left(\begin{dsdn}{3}\lsinglehor{1}\rsinglehor{2}\rsinglehor{3}\end{dsdn}\right) &\to
  \begin{dsdn}{3}\lstrandup{1}{2}\rstranddown{1}{2}\rsinglehor{3}\end{dsdn} \otimes
  \begin{dsdn}{3}\lsinglehor{2}\rsinglehor{1}\rsinglehor{3}\end{dsdn}
  \label{eq:ddlen2.1} \\
  \delta^1: \left(\begin{dsdn}{3}\lsinglehor{1}\lsinglehor{3}\rsinglehor{2}\end{dsdn}\right) &\to
  \begin{dsdn}{3}\lstrandup{1}{2}\rstranddown{1}{2}\lsinglehor{3}\end{dsdn} \otimes
  \begin{dsdn}{3}\lsinglehor{2}\rsinglehor{1}\lsinglehor{3}\end{dsdn}
  \label{eq:ddlen2.2} \\
  \delta^1: \left(\begin{dsdn}{3}\lsinglehor{2}\rsinglehor{1}\rsinglehor{3}\end{dsdn}\right) &\to
  \begin{dsdn}{3}\lstrandup{2}{3}\rstranddown{2}{3}\rsinglehor{1}\end{dsdn} \otimes
  \begin{dsdn}{3}\lsinglehor{3}\rsinglehor{1}\rsinglehor{2}\end{dsdn}
  \label{eq:ddlen2.3} \\
  \delta^1: \left(\begin{dsdn}{3}\lsinglehor{2}\lsinglehor{1}\rsinglehor{3}\end{dsdn}\right) &\to
  \begin{dsdn}{3}\lstrandup{2}{3}\rstranddown{2}{3}\lsinglehor{1}\end{dsdn} \otimes
  \begin{dsdn}{3}\lsinglehor{3}\lsinglehor{1}\rsinglehor{2}\end{dsdn}
  \label{eq:ddlen2.4} \\
  \delta^1: \left(\begin{dsdn}{3}\lsinglehor{1}\rsinglehor{2}\rsinglehor{3}\end{dsdn}\right) &\to
  \begin{dsdn}{3}\lstrandup{1}{3}\rstranddown{1}{3}\rsinglehor{2}\end{dsdn} \otimes
  \begin{dsdn}{3}\lsinglehor{3}\rsinglehor{1}\rsinglehor{2}\end{dsdn}
  \label{eq:ddlen2.5} \\
  \delta^1: \left(\begin{dsdn}{3}\lsinglehor{1}\lsinglehor{2}\rsinglehor{3}\end{dsdn}\right) &\to
  \begin{dsdn}{3}\lstrandup{1}{3}\lsinglehor{2}\rstranddown{1}{3}\end{dsdn} \otimes
  \begin{dsdn}{3}\lsinglehor{2}\lsinglehor{3}\rsinglehor{1}\end{dsdn}
  \label{eq:ddlen2.6}
\end{ddequation}

It is now a matter of combining these following the rules of the box tensor
product. In the figures below, for both type $\DD$ and type $\AA$ bimodules, we
will show the first algebra action on the left and the second algebra action on
the right. This is purely for ease of visualization, and does not indicate which
side the algebras act on. Indeed, both actions on the type $\DD$ bimodule are on
the left, and both actions on the type $\AA$ bimodule are on the
right. Nevertheless, we will often talk about ``left action'' or ``left
idempotents'' to match how the figures are drawn. Moreover, we will put the
$\DD$ arrows on the left, and $\AA$ arrows on the right, since we are tensoring
the second action in $\hatdd(\mathbb{I}_\slz)$ with the first action in
$\hataa(\mathbb{I}_\slz)$.

Each type $\DA$ arrow comes from a single type $\AA$ arrow and zero or more type
$\DD$ arrows. The left outputs (in $\sla(\slz)$) of the type $\DD$ arrows are
multiplied together to give the overall type $D$ output, while the right outputs
(in $\sla(-\slz)$) are given as the left inputs to the type $\AA$ arrow. The
overall type $A$ inputs in $\sla(\slz)$ are given as the right inputs to the
type $\AA$ arrow.

The right idempotent of the $\DD$ generator must agree with the left idempotent
of the $\AA$ generator. The left idempotent of the $\DD$ generator and the right
idempotent of the $\AA$ generator then combine to form the idempotent of the
resulting $\DA$ generator.

We now look at each of the four cases:

\paragraph{Middle idempotent on the left, lower interval first in ordering.}
The combination is shown as Figure \ref{fig:daform1}, using (\ref{eq:ddlen2.1}),
(\ref{eq:ddlen2.3}), and (\ref{eq:aalen2.1}). The resulting arrow is:

\begin{daequation}{equation}
  \label{eq:dalen2.1}
  \delta_2^1: \left(\begin{dsdn}{3}\lsinglehor{1}\rsinglehor{1}\end{dsdn}~,~
    \begin{sdn}{3}\strandup{1}{3}\end{sdn}\right) \to
  \begin{sdn}{3}\strandup{1}{3}\end{sdn} ~\otimes~
  \begin{dsdn}{3}\lsinglehor{3}\rsinglehor{3}\end{dsdn}.
\end{daequation}

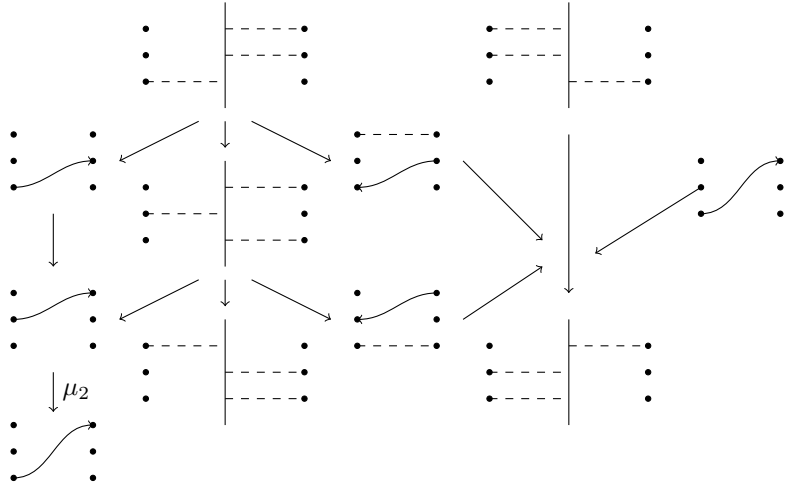
\begin{figure}[h!tb] \centering
  \begin{tikzpicture} [x=10pt,y=10pt]
    \begin{subssdn}{0}{0}{3}\strandup{1}{3}\end{subssdn}
    \begin{subssdn}{0}{5}{3}\strandup{2}{3}\end{subssdn}
    \begin{subssdn}{0}{11}{3}\strandup{1}{2}\end{subssdn}
    \begin{subsdn}{5}{3}{3}\lsinglehor{3}\rsinglehor{1}\rsinglehor{2}\end{subsdn}
    \begin{subsdn}{5}{9}{3}\lsinglehor{2}\rsinglehor{1}\rsinglehor{3}\end{subsdn}
    \begin{subsdn}{5}{15}{3}\lsinglehor{1}\rsinglehor{2}\rsinglehor{3}\end{subsdn}
    \begin{subssdn}{13}{5}{3}\singlehor{1}\stranddown{2}{3}\end{subssdn}
    \begin{subssdn}{13}{11}{3}\singlehor{3}\stranddown{1}{2}\end{subssdn}
    \begin{subsdn}{18}{3}{3}\lsinglehor{1}\lsinglehor{2}\rsinglehor{3}\end{subsdn}
    \begin{subsdn}{18}{15}{3}\lsinglehor{2}\lsinglehor{3}\rsinglehor{1}\end{subsdn}
    \begin{subssdn}{26}{10}{3}\strandup{1}{3}\end{subssdn}
    \draw [->] (1.5,11) to (1.5,9);
    \draw [->] (1.5,5) to node [right] {$\mu_2$} (1.5,3.5);
    \draw [->] (8,14.5) to (8,13.5);
    \draw [->] (8,8.5) to (8,7.5);
    \draw [->] (7,14.5) to (4,13);
    \draw [->] (7,8.5) to (4,7);
    \draw [->] (9,14.5) to (12,13);
    \draw [->] (9,8.5) to (12,7);
    \draw [->] (21,14) to (21,8);
    \draw [->] (17,13) to (20,10);
    \draw [->] (17,7) to (20,9);
    \draw [->] (26,12) to (22,9.5);
  \end{tikzpicture}
  \caption{Formation of type $\DA$ operation, case 1.}
  \label{fig:daform1}
\end{figure}

\paragraph{Middle idempotent on the left, upper interval first in ordering.} The
combination is shown in Figure \ref{fig:daform2}, using (\ref{eq:ddlen2.5}) and
(\ref{eq:aalen2.2}). The resulting arrow is the same as in
(\ref{eq:dalen2.1}). So in this case the order of the two intervals already does
not matter at the $\DA$ level.

\begin{figure}[h!tb] \centering
  \begin{tikzpicture} [x=10pt,y=10pt]
    \begin{subssdn}{0}{4}{3}\strandup{1}{3}\end{subssdn}
    \begin{subsdn}{5}{0}{3}\lsinglehor{3}\rsinglehor{1}\rsinglehor{2}\end{subsdn}
    \begin{subsdn}{5}{8}{3}\lsinglehor{1}\rsinglehor{2}\rsinglehor{3}\end{subsdn}
    \begin{subssdn}{13}{4}{3}\singlehor{2}\stranddown{1}{3}\end{subssdn}
    \begin{subsdn}{18}{0}{3}\lsinglehor{1}\lsinglehor{2}\rsinglehor{3}\end{subsdn}
    \begin{subsdn}{18}{8}{3}\lsinglehor{2}\lsinglehor{3}\rsinglehor{1}\end{subsdn}
    \begin{subssdn}{26}{4}{3}\strandup{1}{3}\end{subssdn}
    \draw [->] (8,7) to (8,5);
    \draw [->] (21,7) to (21,5);
    \draw [->] (7,7) to (4,6);
    \draw [->] (9,7) to (12,6);
    \draw [->] (17,6) to (20,5);
    \draw [->] (25,6) to (22,5);
  \end{tikzpicture}
  \caption{Formation of type $\DA$ operation, case 2.}
  \label{fig:daform2}
\end{figure}
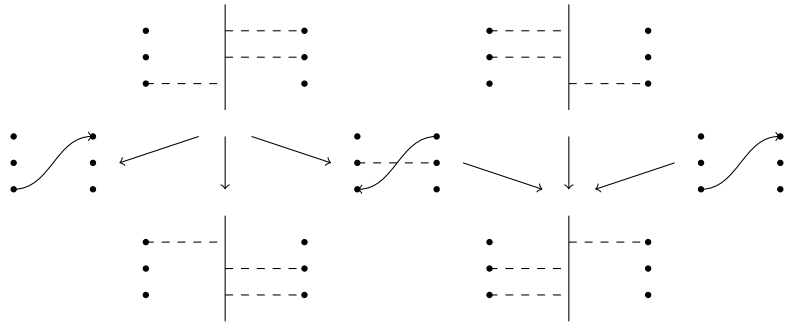

\paragraph{Middle idempotent on the right, upper interval first in ordering.}
The combination is shown in Figure \ref{fig:daform3}, using (\ref{eq:ddlen2.6})
and (\ref{eq:aalen2.3}). The resulting arrow is:
\begin{daequation}{equation}
  \label{eq:dalen2.2}
  \delta_2^1: \left(\begin{dsdn}{3}\lsinglehor{1}\lsinglehor{2}\rsinglehor{1}\rsinglehor{2}\end{dsdn}~,~
    \begin{sdn}{3}\strandup{1}{3}\singlehor{2}\end{sdn}\right) \to
  \begin{sdn}{3}\strandup{1}{3}\singlehor{2}\end{sdn} ~\otimes~
  \begin{dsdn}{3}\lsinglehor{2}\lsinglehor{3}\rsinglehor{2}\rsinglehor{3}\end{dsdn}.
\end{daequation}

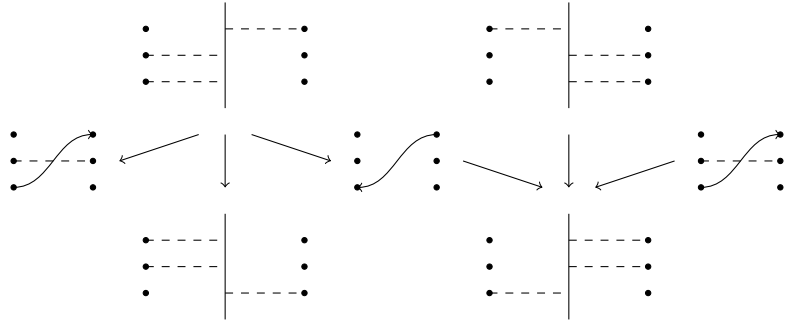
\begin{figure}[h!tb] \centering
  \begin{tikzpicture} [x=10pt,y=10pt]
    \begin{subssdn}{0}{4}{3}\singlehor{2}\strandup{1}{3}\end{subssdn}
    \begin{subsdn}{5}{0}{3}\lsinglehor{2}\lsinglehor{3}\rsinglehor{1}\end{subsdn}
    \begin{subsdn}{5}{8}{3}\lsinglehor{1}\lsinglehor{2}\rsinglehor{3}\end{subsdn}
    \begin{subssdn}{13}{4}{3}\stranddown{1}{3}\end{subssdn}
    \begin{subsdn}{18}{0}{3}\lsinglehor{1}\rsinglehor{2}\rsinglehor{3}\end{subsdn}
    \begin{subsdn}{18}{8}{3}\lsinglehor{3}\rsinglehor{1}\rsinglehor{2}\end{subsdn}
    \begin{subssdn}{26}{4}{3}\singlehor{2}\strandup{1}{3}\end{subssdn}
    \draw [->] (8,7) to (8,5);
    \draw [->] (21,7) to (21,5);
    \draw [->] (7,7) to (4,6);
    \draw [->] (9,7) to (12,6);
    \draw [->] (17,6) to (20,5);
    \draw [->] (25,6) to (22,5);
  \end{tikzpicture}
  \caption{Formation of type $\DA$ operation, case 3.}
  \label{fig:daform3}
\end{figure}

Another combination, using (\ref{eq:ddlen2.4}), (\ref{eq:ddlen2.2}), and
(\ref{eq:aalen2.4}), gives the arrow:
\begin{daequation}{equation}
  \label{eq:dalen2.2'}
  \delta_2^1: \left(\begin{dsdn}{3}\lsinglehor{1}\lsinglehor{2}\rsinglehor{1}\rsinglehor{2}\end{dsdn}~,~
    \begin{sdn}{3}\strandup{1}{2}\strandup{2}{3}\end{sdn}\right) \to
  \begin{sdn}{3}\strandup{1}{2}\strandup{2}{3}\end{sdn} ~\otimes~
  \begin{dsdn}{3}\lsinglehor{2}\lsinglehor{3}\rsinglehor{2}\rsinglehor{3}\end{dsdn}.
\end{daequation}

\paragraph{Middle idempotent on the right, lower interval first in ordering.}
The combination is shown in Figure \ref{fig:daform4}, using (\ref{eq:ddlen2.6})
and (\ref{eq:aalen2.5}). The resulting arrow is:

\begin{daequation}{equation}
  \label{eq:dalen2.3}
  \delta_3^1: \left(\begin{dsdn}{3}\lsinglehor{1}\lsinglehor{2}\rsinglehor{1}\rsinglehor{2}\end{dsdn}~,~
    \begin{sdn}{3}\strandup{2}{3}\singlehor{1}\end{sdn}~,~
    \begin{sdn}{3}\strandup{1}{2}\singlehor{3}\end{sdn}\right) \to
  \begin{sdn}{3}\strandup{1}{3}\singlehor{2}\end{sdn} ~\otimes~
  \begin{dsdn}{3}\lsinglehor{2}\lsinglehor{3}\rsinglehor{2}\rsinglehor{3}\end{dsdn}.
\end{daequation}

\begin{figure}[h!tb] \centering
  \begin{tikzpicture} [x=10pt,y=10pt]
    \begin{subssdn}{0}{5}{3}\singlehor{2}\strandup{1}{3}\end{subssdn}
    \begin{subsdn}{5}{0}{3}\lsinglehor{2}\lsinglehor{3}\rsinglehor{1}\end{subsdn}
    \begin{subsdn}{5}{10}{3}\lsinglehor{1}\lsinglehor{2}\rsinglehor{3}\end{subsdn}
    \begin{subssdn}{13}{5}{3}\stranddown{1}{3}\end{subssdn}
    \begin{subsdn}{18}{0}{3}\lsinglehor{1}\rsinglehor{2}\rsinglehor{3}\end{subsdn}
    \begin{subsdn}{18}{10}{3}\lsinglehor{3}\rsinglehor{1}\rsinglehor{2}\end{subsdn}
    \begin{subssdn}{26}{4}{3}\singlehor{3}\strandup{1}{2}\end{subssdn}
    \begin{subssdn}{26}{8}{3}\singlehor{1}\strandup{2}{3}\end{subssdn}
    \draw [->] (8,9) to (8,5);
    \draw [->] (21,9) to (21,5);
    \draw [->] (7,9) to (4,7);
    \draw [->] (9,9) to (12,7);
    \draw [->] (17,7) to (20,5);
    \draw [->] (25,9) to (22,6);
    \draw [->] (25,6) to (22,5);
  \end{tikzpicture}
  \caption{Formation of type $\DA$ operation, case 4.}
  \label{fig:daform4}
\end{figure}
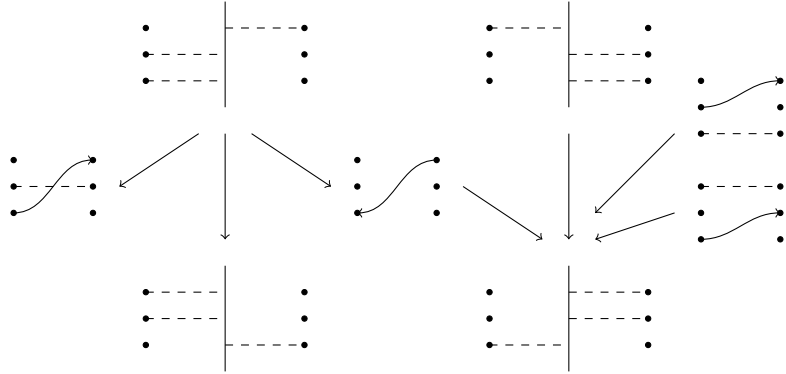

This arrow shows that the model $\hataa(\iz)\boxtimes\hatdd(\iz)$ of
$\widehatit{CFDA}(\iz)$ is not exactly the same, but only homotopy equivalent to
$\mathbb{I}$.

\subsubsection{Pieces of arcslide}
We now compute some simple arrows in $\hatda(\tau)$ for an arcslide
$\tau$. Since the method used here is similar to that in the previous section,
we will show only the results.

First, we consider the case where $b_1$ is directly above $c_1$, and computing
the arrows in $\hatda(\tau)$ corresponding to the region of the Heegaard diagram
around $b_1$. The Heegaard diagram around $b_1$ is:
\[
\begin{tikzpicture}[x=14pt,y=14pt]
  \draw (0,-0.2) to (0,2.2); \draw (4,-0.2) to (4,2.2);
  \draw (0,0) to (4,0); \draw (0,2) to (4,2);
  \draw[rounded corners=5] (0,1) to (2,1) to (2,0);
  \filldraw[fill=white] (2,0) circle (0.15);
  \draw (2,0) circle (0.5);
  \filldraw[fill=white] (1,1) circle (0.15);
  \filldraw[fill=white] (3,2) circle (0.15);
\end{tikzpicture}.\]

The possible type $\DD$ arrows are:
\begin{ddequation}{alignat}{1}
  \delta^1: \left(\begin{dsdmr2}\lsinglehor{1}\rsinglehor{3}\end{dsdmr2}\right) &\to
  \begin{dsdmr2}\lstrandup{1}{3}\rstranddown{1}{3}\end{dsdmr2} \otimes
  \begin{dsdmr2}\lsinglehor{3}\rsinglehor{1}\end{dsdmr2}
  \label{eq:dd1.1} \\
  \delta^1: \left(\begin{dsdmr2}\lsinglehor{1}\lsinglehor{2}\rsinglehor{3}\end{dsdmr2}\right) &\to
  \begin{dsdmr2}\lsinglehor{2}\lstrandup{1}{3}\rstranddown{1}{3}\end{dsdmr2} \otimes
  \begin{dsdmr2}\lsinglehor{2}\lsinglehor{3}\rsinglehor{1}\end{dsdmr2}
  \label{eq:dd1.2} \\
  \delta^1: \left(\begin{dsdmr2}\lsinglehor{1}\lsinglehor{3}\rsinglehor{1}\end{dsdmr2}\right) &\to
  \begin{dsdmr2}\lsinglehor{3}\lstrandup{1}{2}\rsinglehor{1}\end{dsdmr2} \otimes
  \begin{dsdmr2}\lsinglehor{2}\lsinglehor{3}\rsinglehor{1}\end{dsdmr2}
  \label{eq:dd1.3} \\
  \delta^1: \left(\begin{dsdmr2}\lsinglehor{1}\rsinglehor{1}\rsinglehor{3}\end{dsdmr2}\right) &\to
  \begin{dsdmr2}\lstrandup{1}{2}\rsinglehor{1}\rsinglehor{3}\end{dsdmr2} \otimes
  \begin{dsdmr2}\lsinglehor{2}\rsinglehor{1}\rsinglehor{3}\end{dsdmr2}
  \label{eq:dd1.4} \\
  \delta^1: \left(\begin{dsdmr2}\lsinglehor{1}\lsinglehor{2}\rsinglehor{3}\end{dsdmr2}\right) &\to
  \begin{dsdmr2}\lsinglehor{1}\lstrandup{2}{3}\rstranddown{1}{3}\end{dsdmr2} \otimes
  \begin{dsdmr2}\lsinglehor{1}\lsinglehor{3}\rsinglehor{1}\end{dsdmr2}
  \label{eq:dd1.5}
\end{ddequation}

This comes directly from \cite{LOT10c}. The only potentially tricky part is
figuring out the possible locations of idempotents. For example, in the arrow
$\delta^1: \mathbf{x}\to (a(\sigma),1)\otimes\mathbf{y}$ (third and fourth
arrows above; $(\sigma)$ is the chord $c_1\to b_1$), the left idempotent of
$\mathbf{y}$ must be occupied at the $B$ pair and unoccupied at the $C$
pair. Since generators of $\hatdd(\tau)$ either have complementary idempotents
or idempotents that are both occupied at $C$, the right idempotent of
$\mathbf{y}$ must be occupied at $C$ (so $\mathbf{y}$ is of type $X$). From the
idempotent of $\mathbf{y}$, we can deduce that of $\mathbf{x}$, and see that
$\mathbf{x}$ is of type $Y$. Similar arguments are used to list possible
idempotents in the other cases.

Computing the type $\DA$ arrows in this case is relatively straightforward, as
we are combining with the arrow (\ref{eq:aalen1}) on a size 1 interval. The
results are as follows, where (\ref{eq:da1.1}) through (\ref{eq:da1.5}) follow
respectively from (\ref{eq:dd1.1}) through (\ref{eq:dd1.5}).

\begin{daequation}{alignat}{1}
  \delta_2^1: \left(\begin{dsdmr2}\lsinglehor{1}\rsinglehor{1}\end{dsdmr2} ~,~
    \begin{sdbig2}\strandup{1}{3}\end{sdbig2} \right) &\to
  \begin{sdn}{3}\strandup{1}{3}\end{sdn} \otimes
  \begin{dsdmr2}\lsinglehor{3}\rsinglehor{3}\end{dsdmr2}
  \label{eq:da1.1} \\
  \delta_2^1: \left(\begin{dsdmr2}\lsinglehor{1}\lsinglehor{2}\rsinglehor{1}\end{dsdmr2} ~,~
    \begin{sdbig2}\strandup{1}{3}\end{sdbig2} \right) &\to
  \begin{sdn}{3}\singlehor{2}\strandup{1}{3}\end{sdn} \otimes
  \begin{dsdmr2}\lsinglehor{2}\lsinglehor{3}\rsinglehor{3}\end{dsdmr2}
  \label{eq:da1.2} \\
  \delta_1^1: \left(\begin{dsdmr2}\lsinglehor{1}\lsinglehor{3}\rsinglehor{3}\end{dsdmr2}\right) &\to
  \begin{sdn}{3}\singlehor{3}\strandup{1}{2}\end{sdn} \otimes
  \begin{dsdmr2}\lsinglehor{2}\lsinglehor{3}\rsinglehor{3}\end{dsdmr2}
  \label{eq:da1.3} \\
  \delta_1^1: \left(\begin{dsdmr2}\lsinglehor{1}\end{dsdmr2}\right) &\to
  \begin{sdn}{3}\strandup{1}{2}\end{sdn} \otimes
  \begin{dsdmr2}\lsinglehor{2}\end{dsdmr2}
  \label{eq:da1.4} \\
  \delta_2^1: \left(\begin{dsdmr2}\lsinglehor{1}\lsinglehor{2}\rsinglehor{1}\end{dsdmr2} ~,~
    \begin{sdbig2}\strandup{1}{3}\end{sdbig2} \right) &\to
  \begin{sdn}{3}\singlehor{1}\strandup{2}{3}\end{sdn} \otimes
  \begin{dsdmr2}\lsinglehor{1}\lsinglehor{3}\rsinglehor{3}\end{dsdmr2}
  \label{eq:da1.5}
\end{daequation}

Now we consider other side of the same case, computing arrows in $\hatda(\tau)$
corresponding to the region around $b_1'$. Since $b_1$ is directly above $c_1$,
we have $b_1'$ directly below $c_2$, and the Heegaard diagram around $b_1'$ is:
\[ \begin{tikzpicture}[x=14pt,y=14pt]
  \draw (0,-0.2) to (0,2.2); \draw (4,-0.2) to (4,2.2);
  \draw (0,0) to (4,0); \draw (0,2) to (4,2);
  \draw[rounded corners=5] (2,2) to (2,1) to (4,1);
  \filldraw[fill=white] (2,2) circle (0.15);
  \filldraw[fill=white] (1,0) circle (0.15);
\end{tikzpicture}. \]

The type $\DD$ operations are:
\begin{ddequation}{alignat}{1}
  \delta^1: \left(\begin{dsdml2}\lsinglehor{1}\rsinglehor{3}\end{dsdml2}\right) &\to
  \begin{dsdml2}\lstrandup{1}{3}\rstranddown{1}{3}\end{dsdml2} \otimes
  \begin{dsdml2}\lsinglehor{3}\rsinglehor{1}\end{dsdml2}
  \label{eq:dd2.1} \\
  \delta^1: \left(\begin{dsdml2}\lsinglehor{1}\rsinglehor{2}\rsinglehor{3}\end{dsdml2}\right) &\to
  \begin{dsdml2}\lstrandup{1}{3}\rstranddown{1}{3}\rsinglehor{2}\end{dsdml2} \otimes
  \begin{dsdml2}\lsinglehor{3}\rsinglehor{1}\rsinglehor{2}\end{dsdml2}
  \label{eq:dd2.2} \\
  \delta^1: \left(\begin{dsdml2}\lsinglehor{1}\lsinglehor{3}\rsinglehor{3}\end{dsdml2}\right) &\to
  \begin{dsdml2}\lsinglehor{1}\lsinglehor{3}\rstranddown{2}{3}\end{dsdml2} \otimes
  \begin{dsdml2}\lsinglehor{1}\lsinglehor{3}\rsinglehor{2}\end{dsdml2}
  \label{eq:dd2.3} \\
  \delta^1: \left(\begin{dsdml2}\lsinglehor{3}\rsinglehor{1}\rsinglehor{3}\end{dsdml2}\right) &\to
  \begin{dsdml2}\lsinglehor{3}\rsinglehor{1}\rstranddown{2}{3}\end{dsdml2} \otimes
  \begin{dsdml2}\lsinglehor{3}\rsinglehor{1}\rsinglehor{2}\end{dsdml2}
  \label{eq:dd2.4} \\
  \delta^1: \left(\begin{dsdml2}\lsinglehor{1}\rsinglehor{2}\rsinglehor{3}\end{dsdml2}\right) &\to
  \begin{dsdml2}\lstrandup{1}{3}\rsinglehor{3}\rstranddown{1}{2}\end{dsdml2} \otimes
  \begin{dsdml2}\lsinglehor{3}\rsinglehor{1}\rsinglehor{3}\end{dsdml2}
  \label{eq:dd2.5}
\end{ddequation}

This time we will need to combine with type $\AA$ arrows on a size 2 interval,
emulating the method in Section \ref{subsec:typeaaid}. The results are:

\begin{daequation}{alignat}{1}
  (\mathrm{upper~first})\;
  \delta_2^1: \left(\begin{dsdml2}\lsinglehor{1}\rsinglehor{1}\rsinglehor{2}\end{dsdml2} ~,~
    \begin{sdn}{3}\singlehor{2}\strandup{1}{3}\end{sdn}\right) &\to
  \begin{sdbig2}\strandup{1}{3}\end{sdbig2} \otimes
  \begin{dsdml2}\lsinglehor{3}\rsinglehor{2}\rsinglehor{3}\end{dsdml2}
  \label{eq:da2.1} \\
  (\mathrm{lower~first})\;
  \delta_3^1: \left(\begin{dsdml2}\lsinglehor{1}\rsinglehor{1}\rsinglehor{2}\end{dsdml2} ~,~
    \begin{sdn}{3}\singlehor{1}\strandup{2}{3}\end{sdn} ~,~
    \begin{sdn}{3}\singlehor{3}\strandup{1}{2}\end{sdn}\right) &\to
  \begin{sdbig2}\strandup{1}{3}\end{sdbig2} \otimes
  \begin{dsdml2}\lsinglehor{3}\rsinglehor{2}\rsinglehor{3}\end{dsdml2}
  \label{eq:da2.2} \\
  \delta_2^1: \left(\begin{dsdml2}\lsinglehor{1}\rsinglehor{1}\end{dsdml2} ~,~
    \begin{sdn}{3}\strandup{1}{3}\end{sdn}\right) &\to
  \begin{sdbig2}\strandup{1}{3}\end{sdbig2} \otimes
  \begin{dsdml2}\lsinglehor{3}\rsinglehor{3}\end{dsdml2}
  \label{eq:da2.3} \\
  \delta_2^1: \left(\begin{dsdml2}\lsinglehor{1}\lsinglehor{3}\rsinglehor{1}\rsinglehor{2}\end{dsdml2} ~,~
    \begin{sdn}{3}\singlehor{1}\strandup{2}{3}\end{sdn}\right) &\to
  \begin{sdbig2}\singlehor{1}\singlehor{3}\end{sdbig2} \otimes
  \begin{dsdml2}\lsinglehor{1}\lsinglehor{3}\rsinglehor{1}\rsinglehor{3}\end{dsdml2} \label{eq:da2.4} \\
  \delta_2^1: \left(\begin{dsdml2}\lsinglehor{3}\rsinglehor{2}\end{dsdml2} ~,~
    \begin{sdn}{3}\strandup{2}{3}\end{sdn}\right) &\to
  \begin{sdbig2}\singlehor{3}\end{sdbig2} \otimes
  \begin{dsdml2}\lsinglehor{3}\rsinglehor{3}\end{dsdml2}
  \label{eq:da2.5} \\
  \delta_2^1: \left(\begin{dsdml2}\lsinglehor{1}\rsinglehor{1}\end{dsdml2} ~,~
    \begin{sdn}{3}\strandup{1}{2}\end{sdn}\right) &\to
  \begin{sdbig2}\strandup{1}{3}\end{sdbig2} \otimes
  \begin{dsdml2}\lsinglehor{3}\rsinglehor{2}\end{dsdml2}
  \label{eq:da2.6}
\end{daequation}

The first arrow follows from (\ref{eq:dd2.1}) and (\ref{eq:aalen2.3}) only if
the upper interval comes first in the ordering $<_{\slz'}$ for the right pointed
matched circle. The second arrow follows from (\ref{eq:dd2.1}) and
(\ref{eq:aalen2.5}) only if the lower interval comes first in the ordering. The
third arrow does not depend on ordering. However, it is formed in different ways
for the two orderings: if upper interval comes first, it follows from
(\ref{eq:dd2.2}) and (\ref{eq:aalen2.2}), otherwise it follows (\ref{eq:dd2.5}),
(\ref{eq:dd2.4}), and (\ref{eq:aalen2.1}). The last three arrows are independent
of ordering. They follow from (\ref{eq:aalen1}) and respectively
(\ref{eq:dd2.3}) through (\ref{eq:dd2.5}).

The cases where $b_1$ is directly below $c_1$ (and therefore $b_1'$ is directly
above $c_2$) are very similar. The Heegaard diagram around $b_1$ is:
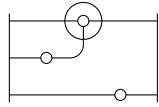
\[
\begin{tikzpicture}[x=14pt,y=14pt]
  \draw (0,-0.2) to (0,2.2); \draw (4,-0.2) to (4,2.2);
  \draw (0,0) to (4,0); \draw (0,2) to (4,2);
  \draw[rounded corners=5] (0,1) to (2,1) to (2,2);
  \filldraw[fill=white] (2,2) circle (0.15);
  \draw (2,2) circle (0.5);
  \filldraw[fill=white] (1,1) circle (0.15);
  \filldraw[fill=white] (3,0) circle (0.15);
\end{tikzpicture}.\] The first two $\DA$ arrows are the same as (\ref{eq:da1.1})
and (\ref{eq:da1.2}), and the last three $\DA$ arrows are modified appropriately
from (\ref{eq:da1.3}) through (\ref{eq:da1.5}).

The Heegaard diagram around $b_1'$ is:
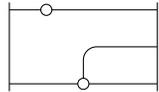
\[ \begin{tikzpicture}[x=14pt,y=14pt]
  \draw (0,-0.2) to (0,2.2); \draw (4,-0.2) to (4,2.2);
  \draw (0,0) to (4,0); \draw (0,2) to (4,2);
  \draw[rounded corners=5] (2,0) to (2,1) to (4,1);
  \filldraw[fill=white] (2,0) circle (0.15);
  \filldraw[fill=white] (1,2) circle (0.15);
\end{tikzpicture}. \] The first three $\DA$ arrows are the same as
(\ref{eq:da2.1}) through (\ref{eq:da2.3}), and the last three $\DA$ arrows are
modified appropriately from (\ref{eq:da2.4}) through (\ref{eq:da2.6}).

\subsubsection{Type $\AA$ on a size 2 interval, paired case}
We now consider the case of a size 2 interval, where the top and bottom points
are paired with each other. Here the lower interval immediately precedes the
upper interval in the ordering $<_\slz$. There are no arrows starting at
generators where the middle idempotent is on the same side as the idempotent
containing the top and bottom points. Starting at generators where the middle
idempotent is to the right, there is a sequence:
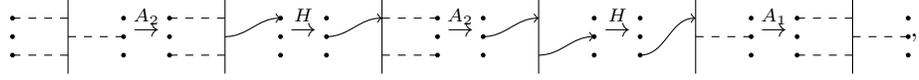
\begin{equation}
  \begin{dsdn}{3}\ldoublehor{1}{3}\rsinglehor{2}\end{dsdn}\overto{A_2}
  \begin{dsdn}{3}\ldoublehor{1}{3}\rstrandup{2}{3}\end{dsdn}\overto{H}
  \begin{dsdn}{3}\lstrandup{2}{3}\rdoublehor{1}{3}\end{dsdn}\overto{A_2}
  \begin{dsdn}{3}\lstrandup{2}{3}\rstrandup{1}{2}\end{dsdn}\overto{H}
  \begin{dsdn}{3}\lstrandup{1}{3}\rsinglehor{2}\end{dsdn}\overto{A_1}
  \begin{dsdn}{3}\ldoublehor{1}{3}\rsinglehor{2}\end{dsdn}, \nonumber
\end{equation}
giving the arrow:
\begin{aaequation}{equation} \label{eq:aalen2.1'}
  m_{1,1,2}: \left(
    \left[\begin{dsdn}{3}\ldoublehor{1}{3}\rsinglehor{2}\end{dsdn}\right];
    \begin{sdn}{3}\stranddown{1}{3}\end{sdn}~;~
    \begin{sdn}{3}\strandup{2}{3}\end{sdn},
    \begin{sdn}{3}\strandup{1}{2}\end{sdn}\right) \to
  \left[\begin{dsdn}{3}\ldoublehor{1}{3}\rsinglehor{2}\end{dsdn}\right].
\end{aaequation}

Starting at generators where the middle idempotent is to the left, there is a
sequence:
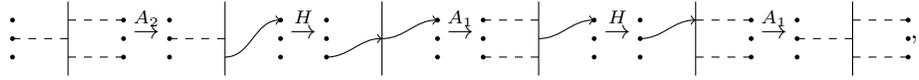
\begin{equation}
  \begin{dsdn}{3}\lsinglehor{2}\rdoublehor{1}{3}\end{dsdn}\overto{A_2}
  \begin{dsdn}{3}\lsinglehor{2}\rstrandup{1}{3}\end{dsdn}\overto{H}
  \begin{dsdn}{3}\lstrandup{1}{2}\rstrandup{2}{3}\end{dsdn}\overto{A_1}
  \begin{dsdn}{3}\ldoublehor{1}{3}\rstrandup{2}{3}\end{dsdn}\overto{H}
  \begin{dsdn}{3}\lstrandup{2}{3}\rdoublehor{1}{3}\end{dsdn}\overto{A_1}
  \begin{dsdn}{3}\lsinglehor{2}\rdoublehor{1}{3}\end{dsdn}, \nonumber
\end{equation}
giving the arrow:
\begin{aaequation}{equation} \label{eq:aalen2.2'}
  m_{1,2,1}: \left(
    \left[\begin{dsdn}{3}\lsinglehor{2}\rdoublehor{1}{3}\end{dsdn}\right];
    \begin{sdn}{3}\stranddown{1}{2}\end{sdn},
    \begin{sdn}{3}\stranddown{2}{3}\end{sdn}~;~
    \begin{sdn}{3}\strandup{1}{3}\end{sdn}\right) \to
  \left[\begin{dsdn}{3}\lsinglehor{2}\rdoublehor{1}{3}\end{dsdn}\right].
\end{aaequation}

Furthermore, there are several infinite series of arrows formed by repeating the
moves used above. We list the two arrows that will be used later in the paper:
\begin{eqnarray}
  \begin{dsdn}{3}\ldoublehor{1}{3}\rsinglehor{2}\end{dsdn}&\overto{A_2}&
  \begin{dsdn}{3}\ldoublehor{1}{3}\rstrandup{2}{3}\end{dsdn}\overto{H}
  \begin{dsdn}{3}\lstrandup{2}{3}\rdoublehor{1}{3}\end{dsdn}\overto{A_2}
  \begin{dsdn}{3}\lstrandup{2}{3}\rstrandup{1}{3}\end{dsdn}\overto{H}
  \begin{dsdn}{3}\lstrandup{1}{3}\rstrandup{2}{3}\end{dsdn} \nonumber \\ &\overto{A_1}&
  \begin{dsdn}{3}\ldoublehor{1}{3}\rstrandup{2}{3}\end{dsdn}\overto{H}
  \begin{dsdn}{3}\lstrandup{2}{3}\rdoublehor{1}{3}\end{dsdn}\overto{A_1}
  \begin{dsdn}{3}\lsinglehor{2}\rdoublehor{1}{3}\end{dsdn}, \nonumber
\end{eqnarray}
giving the arrow:
\begin{aaequation}{equation} \label{eq:aalen2.3'}
  m_{1,2,2}: \left(
    \left[\begin{dsdn}{3}\ldoublehor{1}{3}\rsinglehor{2}\end{dsdn}\right];
    \begin{sdn}{3}\stranddown{1}{3}\end{sdn},
    \begin{sdn}{3}\stranddown{2}{3}\end{sdn}~;~
    \begin{sdn}{3}\strandup{2}{3}\end{sdn},
    \begin{sdn}{3}\strandup{1}{3}\end{sdn}\right) \to
  \left[\begin{dsdn}{3}\lsinglehor{2}\rdoublehor{1}{3}\end{dsdn}\right].
\end{aaequation}
and
\begin{eqnarray}
  \begin{dsdn}{3}\lsinglehor{2}\rdoublehor{1}{3}\end{dsdn}&\overto{A_2}&
  \begin{dsdn}{3}\lsinglehor{2}\rstrandup{1}{3}\end{dsdn}\overto{H}
  \begin{dsdn}{3}\lstrandup{1}{2}\rstrandup{2}{3}\end{dsdn}\overto{A_1}
  \begin{dsdn}{3}\ldoublehor{1}{3}\rstrandup{2}{3}\end{dsdn}\overto{H}
  \begin{dsdn}{3}\lstrandup{2}{3}\rdoublehor{1}{3}\end{dsdn} \nonumber \\ &\overto{A_2}&
  \begin{dsdn}{3}\lstrandup{2}{3}\rstrandup{1}{2}\end{dsdn}\overto{H}
  \begin{dsdn}{3}\lstrandup{1}{3}\rsinglehor{2}\end{dsdn}\overto{A_1}
  \begin{dsdn}{3}\ldoublehor{1}{3}\rsinglehor{2}\end{dsdn}, \nonumber
\end{eqnarray}
giving the arrow:
\begin{aaequation}{equation} \label{eq:aalen2.4'}
  m_{1,2,2}: \left(
    \left[\begin{dsdn}{3}\lsinglehor{2}\rdoublehor{1}{3}\end{dsdn}\right];
    \begin{sdn}{3}\stranddown{1}{2}\end{sdn},
    \begin{sdn}{3}\stranddown{1}{3}\end{sdn}~;~
    \begin{sdn}{3}\strandup{1}{3}\end{sdn},
    \begin{sdn}{3}\strandup{1}{2}\end{sdn}\right) \to
  \left[\begin{dsdn}{3}\ldoublehor{1}{3}\rsinglehor{2}\end{dsdn}\right].
\end{aaequation}

Note that in the derivation of the first arrow, we used Case 3 of the homotopy
map $H$ in the multiplicity greater than one case. From this we see that the
type $\AA$ bimodule can have infinitely many arrows. However, in our examples,
only a finite number of them will be used when constructing the action on type
$\DA$ invariants.

\subsubsection{Short underslide}
Using results from the previous section, we compute type $\DA$ arrows for the
short underslide. These are underslides where $b_1$ is the only point between
$c_1$ and $c_2$. Hence $b_1$ and $b_1'$ are located in the same region of the
Heegaard diagram, which is the only region of interest. The diagram for the case
where $b_1$ is directly above $c_1$ is:
\[
\begin{tikzpicture}[x=14pt,y=14pt]
  \draw (0,-0.2) to (0,2.2); \draw (4,-0.2) to (4,2.2);
  \draw (0,0) to (4,0); \draw (0,2) to (4,2);
  \draw[rounded corners=7] (0,1) to (2,1) to (2,0);
  \draw[rounded corners=7] (2,2) to (2,1) to (4,1);
  \filldraw[fill=white] (2,0) circle (0.15);
  \filldraw[fill=white] (2,2) circle (0.15);
  \draw (2,0) circle (0.5);
  \filldraw[fill=white] (1,1) circle (0.15);
\end{tikzpicture}~.\]

The possible type $\DD$ arrows are:
\begin{ddequation}{alignat}{1}
  \delta^1: \left(\begin{dsd3}\ldoublehor{1}{3}\rdoublehor{1}{3}\end{dsd3}\right) &\to
  \begin{dsd3}\lstrandup{1}{2}\rdoublehor{1}{3}\end{dsd3} \otimes
  \begin{dsd3}\lsinglehor{2}\rdoublehor{1}{3}\end{dsd3}
  \label{eq:dd4.1} \\
  \delta^1: \left(\begin{dsd3}\ldoublehor{1}{3}\rdoublehor{1}{3}\end{dsd3}\right) &\to
  \begin{dsd3}\ldoublehor{1}{3}\rstranddown{2}{3}\end{dsd3} \otimes
  \begin{dsd3}\ldoublehor{1}{3}\rsinglehor{2}\end{dsd3}
  \label{eq:dd4.2} \\
  \delta^1: \left(\begin{dsd3}\ldoublehor{1}{3}\rsinglehor{2}\end{dsd3}\right) &\to
  \begin{dsd3}\lstrandup{1}{3}\rstranddown{1}{2}\end{dsd3} \otimes
  \begin{dsd3}\ldoublehor{1}{3}\rdoublehor{1}{3}\end{dsd3}
  \label{eq:dd4.3} \\
  \delta^1: \left(\begin{dsd3}\lsinglehor{2}\rdoublehor{1}{3}\end{dsd3}\right) &\to
  \begin{dsd3}\lstrandup{2}{3}\rstranddown{1}{3}\end{dsd3} \otimes
  \begin{dsd3}\ldoublehor{1}{3}\rdoublehor{1}{3}\end{dsd3}
  \label{eq:dd4.4}
\end{ddequation}

These give rise to type $\DA$ arrows
\begin{daequation}{alignat}{1}
  \delta_1^1: \left(\begin{dsd3}\ldoublehor{1}{3}\rsinglehor{2}\end{dsd3}\right) &\to
  \begin{sdn}{3}\strandup{1}{2}\end{sdn} \otimes
  \begin{dsd3}\lsinglehor{2}\rsinglehor{2}\end{dsd3}
  \label{eq:da4.1} \\
  \delta_2^1: \left(\begin{dsd3}\ldoublehor{1}{3}\rsinglehor{2}\end{dsd3} ~,~
    \begin{sdn}{3}\strandup{2}{3}\end{sdn}\right) &\to
  \begin{sdn}{3}\doublehor{1}{3}\end{sdn} \otimes
  \begin{dsd3}\ldoublehor{1}{3}\rdoublehor{1}{3}\end{dsd3}
  \label{eq:da4.2} \\
  \delta_2^1: \left(\begin{dsd3}\ldoublehor{1}{3}\rdoublehor{1}{3}\end{dsd3} ~,~
    \begin{sdn}{3}\strandup{1}{2}\end{sdn}\right) &\to
  \begin{sdn}{3}\strandup{1}{3}\end{sdn} \otimes
  \begin{dsd3}\ldoublehor{1}{3}\rsinglehor{2}\end{dsd3}
  \label{eq:da4.3} \\
  \delta_3^1: \left(\begin{dsd3}\lsinglehor{2}\rsinglehor{2}\end{dsd3} ~,~
    \begin{sdn}{3}\strandup{2}{3}\end{sdn} ~,~
    \begin{sdn}{3}\strandup{1}{2}\end{sdn}\right) &\to
  \begin{sdn}{3}\strandup{2}{3}\end{sdn} \otimes
  \begin{dsd3}\ldoublehor{1}{3}\rsinglehor{2}\end{dsd3}
  \label{eq:da4.4} \\
  \delta_2^1: \left(\begin{dsd3}\ldoublehor{1}{3}\rdoublehor{1}{3}\end{dsd3} ~,~
    \begin{sdn}{3}\strandup{1}{3}\end{sdn}\right) &\to
  \begin{sdn}{3}\strandup{1}{3}\end{sdn} \otimes
  \begin{dsd3}\ldoublehor{1}{3}\rdoublehor{1}{3}\end{dsd3}
  \label{eq:da4.5} \\
  \delta_3^1: \left(\begin{dsd3}\lsinglehor{2}\rsinglehor{2}\end{dsd3} ~,~
    \begin{sdn}{3}\strandup{2}{3}\end{sdn} ~,~
    \begin{sdn}{3}\strandup{1}{3}\end{sdn}\right) &\to
  \begin{sdn}{3}\strandup{2}{3}\end{sdn} \otimes
  \begin{dsd3}\ldoublehor{1}{3}\rdoublehor{1}{3}\end{dsd3}
  \label{eq:da4.6}
\end{daequation}

Here arrows (\ref{eq:da4.1}) through (\ref{eq:da4.3}) follow respectively from
(\ref{eq:dd4.1}) through (\ref{eq:dd4.3}). Arrow (\ref{eq:da4.4}) follows from
(\ref{eq:dd4.4}) and (\ref{eq:aalen2.1'}). Arrow (\ref{eq:da4.5}) follows from
(\ref{eq:dd4.3}), (\ref{eq:dd4.2}), and (\ref{eq:aalen2.2'}). Arrow
(\ref{eq:da4.6}) follows from (\ref{eq:dd4.4}), (\ref{eq:dd4.2}), and
(\ref{eq:aalen2.3'}).

The diagram for the case where $b_1$ is directly below $c_1$ is:
\[
\begin{tikzpicture}[x=14pt,y=14pt]
  \draw (0,-0.2) to (0,2.2); \draw (4,-0.2) to (4,2.2);
  \draw (0,0) to (4,0); \draw (0,2) to (4,2);
  \draw[rounded corners=7] (0,1) to (2,1) to (2,2);
  \draw[rounded corners=7] (2,0) to (2,1) to (4,1);
  \filldraw[fill=white] (2,0) circle (0.15);
  \filldraw[fill=white] (2,2) circle (0.15);
  \draw (2,0) circle (0.5);
  \filldraw[fill=white] (1,1) circle (0.15);
\end{tikzpicture}.\]

The possible type $\DD$ arrows are:
\begin{ddequation}{alignat}{1}
  \delta^1: \left(\begin{dsd3}\lsinglehor{2}\rdoublehor{1}{3}\end{dsd3}\right) &\to
  \begin{dsd3}\lstrandup{2}{3}\rdoublehor{1}{3}\end{dsd3} \otimes
  \begin{dsd3}\ldoublehor{1}{3}\rdoublehor{1}{3}\end{dsd3}
  \label{eq:dd5.1} \\
  \delta^1: \left(\begin{dsd3}\ldoublehor{1}{3}\rsinglehor{2}\end{dsd3}\right) &\to
  \begin{dsd3}\ldoublehor{1}{3}\rstranddown{1}{2}\end{dsd3} \otimes
  \begin{dsd3}\ldoublehor{1}{3}\rdoublehor{1}{3}\end{dsd3}
  \label{eq:dd5.2} \\
  \delta^1: \left(\begin{dsd3}\ldoublehor{1}{3}\rdoublehor{1}{3}\end{dsd3}\right) &\to
  \begin{dsd3}\lstrandup{1}{3}\rstranddown{2}{3}\end{dsd3} \otimes
  \begin{dsd3}\ldoublehor{1}{3}\rsinglehor{2}\end{dsd3}
  \label{eq:dd5.3} \\
  \delta^1: \left(\begin{dsd3}\ldoublehor{1}{3}\rdoublehor{1}{3}\end{dsd3}\right) &\to
  \begin{dsd3}\lstrandup{1}{2}\rstranddown{1}{3}\end{dsd3} \otimes
  \begin{dsd3}\lsinglehor{2}\rdoublehor{1}{3}\end{dsd3}
  \label{eq:dd5.4}
\end{ddequation}

These give rise to type $\DA$ arrows
\begin{daequation}{alignat}{1}
  \delta_1^1: \left(\begin{dsd3}\lsinglehor{2}\rsinglehor{2}\end{dsd3}\right) &\to
  \begin{sdn}{3}\strandup{2}{3}\end{sdn} \otimes
  \begin{dsd3}\ldoublehor{1}{3}\rsinglehor{2}\end{dsd3}
  \label{eq:da5.1} \\
  \delta_2^1: \left(\begin{dsd3}\ldoublehor{1}{3}\rdoublehor{1}{3}\end{dsd3} ~,~
    \begin{sdn}{3}\strandup{1}{2}\end{sdn}\right) &\to
  \begin{sdn}{3}\doublehor{1}{3}\end{sdn} \otimes
  \begin{dsd3}\ldoublehor{1}{3}\rsinglehor{2}\end{dsd3}
  \label{eq:da5.2} \\
  \delta_2^1: \left(\begin{dsd3}\ldoublehor{1}{3}\rsinglehor{2}\end{dsd3} ~,~
    \begin{sdn}{3}\strandup{2}{3}\end{sdn}\right) &\to
  \begin{sdn}{3}\strandup{1}{3}\end{sdn} \otimes
  \begin{dsd3}\ldoublehor{1}{3}\rdoublehor{1}{3}\end{dsd3}
  \label{eq:da5.3} \\
  \delta_3^1: \left(\begin{dsd3}\ldoublehor{1}{3}\rsinglehor{2}\end{dsd3} ~,~
    \begin{sdn}{3}\strandup{2}{3}\end{sdn} ~,~
    \begin{sdn}{3}\strandup{1}{2}\end{sdn}\right) &\to
  \begin{sdn}{3}\strandup{1}{2}\end{sdn} \otimes
  \begin{dsd3}\lsinglehor{2}\rsinglehor{2}\end{dsd3}
  \label{eq:da5.4} \\
  \delta_2^1: \left(\begin{dsd3}\ldoublehor{1}{3}\rdoublehor{1}{3}\end{dsd3} ~,~
    \begin{sdn}{3}\strandup{1}{3}\end{sdn}\right) &\to
  \begin{sdn}{3}\strandup{1}{3}\end{sdn} \otimes
  \begin{dsd3}\ldoublehor{1}{3}\rdoublehor{1}{3}\end{dsd3}
  \label{eq:da5.5} \\
  \delta_3^1: \left(\begin{dsd3}\ldoublehor{1}{3}\rdoublehor{1}{3}\end{dsd3} ~,~
    \begin{sdn}{3}\strandup{1}{3}\end{sdn} ~,~
    \begin{sdn}{3}\strandup{1}{2}\end{sdn}\right) &\to
  \begin{sdn}{3}\strandup{1}{2}\end{sdn} \otimes
  \begin{dsd3}\lsinglehor{2}\rsinglehor{2}\end{dsd3}
  \label{eq:da5.6}
\end{daequation}

Here arrows (\ref{eq:da5.1}) through (\ref{eq:da5.3}) follow respectively from
(\ref{eq:dd5.1}) through (\ref{eq:dd5.3}). Arrow (\ref{eq:da5.4}) follows from
(\ref{eq:dd5.4}) and (\ref{eq:aalen2.1'}). Arrow (\ref{eq:da5.5}) follows from
(\ref{eq:dd5.2}), (\ref{eq:dd5.3}), and (\ref{eq:aalen2.2'}). Arrow
(\ref{eq:da5.6}) follows from (\ref{eq:dd5.2}), (\ref{eq:dd5.4}), and
(\ref{eq:aalen2.4'}).

\subsubsection{Type $\AA$ on two separated intervals with pairing}
To consider more local situations for the arcslide, we will need type $\AA$
arrows on two separated intervals, such that either the two inner positions or
the two outer positions are paired. These two cases are very similar, so we will
only write out the first case here.

In this case, the upper interval immediately precedes the lower interval in the
ordering $<_\mathcal{Z}$. If the middle idempotent (consisting of the two paired
inner points) is occupied on the left, then it is not possible to multiply both
intervals to the right as the first step. So the only possible sequence of
$[a_{i,1},a_{i,2}]$ is:

\begin{equation}
  \begin{dsd2sep}\ldoublehor{2}{4}\lsinglehor{5}\rsinglehor{1}\end{dsd2sep}\overto{A_2}
  \begin{dsd2sep}\ldoublehor{2}{4}\lsinglehor{5}\rstrandup{1}{2}\end{dsd2sep}\overto{H}
  \begin{dsd2sep}\lsinglehor{5}\lstrandup{1}{2}\rdoublehor{2}{4}\end{dsd2sep}\overto{A_2}
  \begin{dsd2sep}\lsinglehor{5}\lstrandup{1}{2}\rstrandup{4}{5}\end{dsd2sep}\overto{H}
  \begin{dsd2sep}\lstrandup{1}{2}\lstrandup{4}{5}\rsinglehor{5}\end{dsd2sep}\overto{A_1}
  \begin{dsd2sep}\lsinglehor{1}\ldoublehor{2}{4}\rsinglehor{5}\end{dsd2sep} \nonumber
\end{equation}
giving the arrow:
\begin{aaequation}{equation} \label{eq:aasep1}
  m_{1,1,2}: \left(
    \left[\begin{dsd2sep}\ldoublehor{2}{4}\lsinglehor{5}\rsinglehor{1}\end{dsd2sep}\right] ~;~
    \begin{sd2sep}\stranddown{1}{2}\stranddown{4}{5}\end{sd2sep} ~;~
    \begin{sd2sep}\strandup{1}{2}\end{sd2sep} ~,~ \begin{sd2sep}\strandup{4}{5}\end{sd2sep}\right) \to
  \left[\begin{dsd2sep}\ldoublehor{2}{4}\lsinglehor{1}\rsinglehor{5}\end{dsd2sep}\right].
\end{aaequation}

In these diagrams, the two middle positions are paired, and there can be an
arbitrary number of points between them in the full pointed matched circle.
Since no arrow in $\hatdd(\iz)$ gives off an algebra element with two separate
strands, this cannot be used to form a type $\DA$ arrow for the identity.

If the middle idempotent is occupied on the right, it is possible to multiply
both intervals to the right as the first step, but not possible to multiply only
the lower interval. So the only sequence is:
\begin{equation}
  \begin{dsd2sep}\lsinglehor{5}\rsinglehor{1}\rdoublehor{2}{4}\end{dsd2sep}\overto{A_2}
  \begin{dsd2sep}\lsinglehor{5}\rstrandup{1}{2}\rstrandup{4}{5}\end{dsd2sep}\overto{H}
  \begin{dsd2sep}\lstrandup{4}{5}\rstrandup{1}{2}\rsinglehor{5}\end{dsd2sep}\overto{A_1}
  \begin{dsd2sep}\ldoublehor{2}{4}\rsinglehor{5}\rstrandup{1}{2}\end{dsd2sep}\overto{H}
  \begin{dsd2sep}\lstrandup{1}{2}\rdoublehor{2}{4}\rsinglehor{5}\end{dsd2sep}\overto{A_1}
  \begin{dsd2sep}\lsinglehor{1}\rdoublehor{2}{4}\rsinglehor{5}\end{dsd2sep} \nonumber
\end{equation}
giving the arrow:
\begin{aaequation}{equation} \label{eq:aasep2}
  m_{1,2,1}: \left(
    \left[\begin{dsd2sep}\lsinglehor{5}\rsinglehor{1}\rdoublehor{2}{4}\end{dsd2sep}\right] ~;~
    \begin{sd2sep}\stranddown{4}{5}\end{sd2sep} ~,~
    \begin{sd2sep}\stranddown{1}{2}\end{sd2sep} ~;~
    \begin{sd2sep}\strandup{1}{2}\strandup{4}{5}\end{sd2sep}\right) \to
  \left[\begin{dsd2sep}\lsinglehor{1}\rsinglehor{5}\rdoublehor{2}{4}\end{dsd2sep}\right].
\end{aaequation}

This leads to the following type $\DA$ arrow for identity:
\begin{daequation}{equation} \label{eq:dasep1}
  \delta_2^1: \left(\begin{dsd2sep}\lsinglehor{1}\ldoublehor{2}{4}\rsinglehor{1}\rdoublehor{2}{4}\end{dsd2sep} ~,~
    \begin{sd2sep}\strandup{1}{2}\strandup{4}{5}\end{sd2sep}\right) \to
  \begin{sd2sep}\strandup{1}{2}\strandup{4}{5}\end{sd2sep} \otimes
  \begin{dsd2sep}\lsinglehor{5}\ldoublehor{2}{4}\rsinglehor{5}\rdoublehor{2}{4}\end{dsd2sep}. \nonumber
\end{daequation}

\subsubsection{More pieces of arcslide}

We now compute some arrows whose corresponding domains touch both $c_1$ and
$c_2$. We focus on the overslide cases, with the underslide cases being
similar. First, if $b_1$ is directly above $c_1$, the local Heegaard diagram is
as follows. We focus on arrows whose domain is restricted inside this diagram.
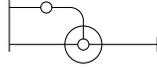
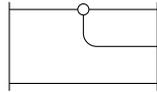
\begin{equation}
  \begin{tikzpicture}[x=14pt,y=14pt]
    \draw (0,-0.2) to (0,2.2); \draw (4,-0.2) to (4,2.2);
    \draw (0,0) to (4,0); \draw (0,2) to (4,2);
    \draw[rounded corners=5] (2,2) to (2,1) to (4,1);
    \draw (0,4) to (4,4); \draw (0,3.8) to (0,5.2); \draw (4,3.8) to (4,4.2);
    \draw[rounded corners=5] (0,5) to (2,5) to (2,4);
    \filldraw[fill=white] (2,2) circle (0.15);
    \filldraw[fill=white] (2,4) circle (0.15);
    \draw (2,4) circle (0.5);
    \filldraw[fill=white] (1,5) circle (0.15);
  \end{tikzpicture}. \nonumber
\end{equation}

Here the two horizontal lines where the 1-handle is attached contain the
$\alpha$-arcs for the $C$ pair. Immediately above and below these two lines are
the points $b_1$ on the left and $b_1'$ on the right.

For clarity, we list all type $\DD$ and $\DA$ arrows in this region, even though
some may already have been covered in previous cases. The type $\DD$ arrows are:
\begin{ddequation}{alignat}{1}
  \delta^1: \left(\begin{dsdlrise}\lsinglehor{1}\ldoublehor{3}{5}\rdoublehor{3}{5}\end{dsdlrise}\right) &\to
  \begin{dsdlrise}\lsinglehor{1}\lstrandup{5}{6}\rdoublehor{3}{5}\end{dsdlrise} \otimes
  \begin{dsdlrise}\lsinglehor{1}\lsinglehor{6}\rdoublehor{3}{5}\end{dsdlrise}
  \label{eq:dd6.1} \\
  \delta^1: \left(\begin{dsdlrise}\ldoublehor{3}{5}\rsinglehor{1}\rdoublehor{3}{5}\end{dsdlrise}\right) &\to
  \begin{dsdlrise}\lstrandup{5}{6}\rsinglehor{1}\rdoublehor{3}{5}\end{dsdlrise} \otimes
  \begin{dsdlrise}\lsinglehor{6}\rsinglehor{1}\rdoublehor{3}{5}\end{dsdlrise}
  \label{eq:dd6.2} \\
  \delta^1: \left(\begin{dsdlrise}\lsinglehor{1}\rsinglehor{2}\rdoublehor{3}{5}\end{dsdlrise}\right) &\to
  \begin{dsdlrise}\lstrandup{1}{3}\rstranddown{1}{2}\rdoublehor{3}{5}\end{dsdlrise} \otimes
  \begin{dsdlrise}\ldoublehor{3}{5}\rsinglehor{1}\rdoublehor{3}{5}\end{dsdlrise}
  \label{eq:dd6.3} \\
  \delta^1: \left(\begin{dsdlrise}\lsinglehor{1}\ldoublehor{3}{5}\rsinglehor{2}\end{dsdlrise}\right) &\to
  \begin{dsdlrise}\lstrandup{1}{3}\lstrandup{5}{6}\rstranddown{1}{2}\end{dsdlrise} \otimes
  \begin{dsdlrise}\ldoublehor{3}{5}\lsinglehor{6}\rsinglehor{1}\end{dsdlrise}
  \label{eq:dd6.4}
\end{ddequation}

These give rise to the following type $\DA$ arrows. Here (\ref{eq:da6.1})
through (\ref{eq:da6.4}) follow respectively from (\ref{eq:dd6.1}) through
(\ref{eq:dd6.4}).

\begin{daequation}{alignat}{1}
  \delta_1^1: \left(\begin{dsdlrise}\lsinglehor{1}\ldoublehor{3}{5}\rsinglehor{1}\rsinglehor{2}\end{dsdlrise}\right) &\to
  \begin{sd6m24}\singlehor{1}\strandup{5}{6}\end{sd6m24} \otimes
  \begin{dsdlrise}\lsinglehor{1}\lsinglehor{6}\rsinglehor{1}\rsinglehor{2}\end{dsdlrise} \label{eq:da6.1} \\
  \delta_1^1: \left(\begin{dsdlrise}\ldoublehor{3}{5}\rsinglehor{2}\end{dsdlrise}\right) &\to
  \begin{sd6m24}\strandup{5}{6}\end{sd6m24} \otimes
  \begin{dsdlrise}\lsinglehor{6}\rsinglehor{2}\end{dsdlrise}
  \label{eq:da6.2} \\
  \delta_2^1: \left(\begin{dsdlrise}\lsinglehor{1}\rsinglehor{1}\end{dsdlrise} ~,~
    \begin{sd6m46}\strandup{1}{2}\end{sd6m46}\right) &\to
  \begin{sd6m24}\strandup{1}{3}\end{sd6m24} \otimes
  \begin{dsdlrise}\ldoublehor{3}{5}\rsinglehor{2}\end{dsdlrise}
  \label{eq:da6.3} \\
  \delta_2^1: \left(\begin{dsdlrise}\lsinglehor{1}\ldoublehor{3}{5}\rsinglehor{1}\rdoublehor{3}{5}\end{dsdlrise} ~,~
    \begin{sd6m46}\strandup{1}{2}\doublehor{3}{5}\end{sd6m46}\right) &\to
  \begin{sd6m24}\strandup{1}{3}\strandup{5}{6}\end{sd6m24} \otimes
  \begin{dsdlrise}\ldoublehor{3}{5}\lsinglehor{6}\rsinglehor{2}\rdoublehor{3}{5}\end{dsdlrise} \label{eq:da6.4}
\end{daequation}

The case where $b_1$ is directly below $c_1$ is again more complicated. The
local Heegaard diagram is
\begin{equation}
  \begin{tikzpicture}[x=14pt,y=14pt]
    \draw (0,-0.2) to (0,2.2); \draw (4,-0.2) to (4,2.2);
    \draw (0,0) to (4,0); \draw (0,2) to (4,2);
    \draw[rounded corners=5] (2,2) to (2,1) to (0,1);
    \draw (0,4) to (4,4); \draw (4,3.8) to (4,5.2); \draw (0,3.8) to (0,4.2);
    \draw[rounded corners=5] (4,5) to (2,5) to (2,4);
    \filldraw[fill=white] (2,2) circle (0.15);
    \filldraw[fill=white] (2,4) circle (0.15);
    \draw (2,4) circle (0.5);
    \filldraw[fill=white] (3,5) circle (0.15);
  \end{tikzpicture}. \nonumber
\end{equation}

The type $\DD$ arrows are:
\begin{ddequation}{alignat}{1}
  \delta^1: \left(\begin{dsdrrise}\lsinglehor{1}\ldoublehor{3}{5}\rsinglehor{6}\end{dsdrrise}\right) &\to
  \begin{dsdrrise}\lsinglehor{1}\ldoublehor{3}{5}\rstranddown{5}{6}\end{dsdrrise} \otimes
  \begin{dsdrrise}\lsinglehor{1}\ldoublehor{3}{5}\rdoublehor{3}{5}\end{dsdrrise}
  \label{eq:dd7.1} \\
  \delta^1: \left(\begin{dsdrrise}\ldoublehor{3}{5}\rsinglehor{1}\rsinglehor{6}\end{dsdrrise}\right) &\to
  \begin{dsdrrise}\ldoublehor{3}{5}\rsinglehor{1}\rstranddown{5}{6}\end{dsdrrise} \otimes
  \begin{dsdrrise}\ldoublehor{3}{5}\rsinglehor{1}\rdoublehor{3}{5}\end{dsdrrise}
  \label{eq:dd7.2} \\
  \delta^1: \left(\begin{dsdrrise}\lsinglehor{1}\ldoublehor{3}{5}\rdoublehor{3}{5}\end{dsdrrise}\right) &\to
  \begin{dsdrrise}\ldoublehor{3}{5}\lstrandup{1}{2}\rstranddown{1}{3}\end{dsdrrise} \otimes
  \begin{dsdrrise}\lsinglehor{2}\ldoublehor{3}{5}\rsinglehor{1}\end{dsdrrise}
  \label{eq:dd7.3} \\
  \delta^1: \left(\begin{dsdrrise}\lsinglehor{1}\rdoublehor{3}{5}\rsinglehor{6}\end{dsdrrise}\right) &\to
  \begin{dsdrrise}\lstrandup{1}{2}\rstranddown{1}{3}\rstranddown{5}{6}\end{dsdrrise} \otimes
  \begin{dsdrrise}\lsinglehor{2}\rsinglehor{1}\rdoublehor{3}{5}\end{dsdrrise}
  \label{eq:dd7.4}
\end{ddequation}

The resulting type $\DA$ arrows are:

\begin{daequation}{alignat}{2}
  \delta_2^1: \left(\begin{dsdrrise}\lsinglehor{1}\ldoublehor{3}{5}\rsinglehor{1}\rdoublehor{3}{5}\end{dsdrrise} ~,~
    \begin{sd6m24}\singlehor{1}\strandup{5}{6}\end{sd6m24}\right) &\to
  \begin{sd6m46}\singlehor{1}\doublehor{3}{5}\end{sd6m46} \otimes
  \begin{dsdrrise}\lsinglehor{1}\ldoublehor{3}{5}\rsinglehor{1}\rsinglehor{6}\end{dsdrrise} \label{eq:da7.1} \\
  \delta_2^1: \left(\begin{dsdrrise}\ldoublehor{3}{5}\rdoublehor{3}{5}\end{dsdrrise} ~,~
    \begin{sd6m24}\strandup{5}{6}\end{sd6m24}\right) &\to
  \begin{sd6m46}\doublehor{3}{5}\end{sd6m46} \otimes
  \begin{dsdrrise}\ldoublehor{3}{5}\rsinglehor{6}\end{dsdrrise}
  \label{eq:da7.2} \\
  \delta_2^1: \left(\begin{dsdrrise}\lsinglehor{1}\ldoublehor{3}{5}\rsinglehor{1}\rsinglehor{6}\end{dsdrrise} ~,~
    \begin{sd6m24}\strandup{1}{3}\singlehor{6}\end{sd6m24}\right) &\to
  \begin{sd6m46}\strandup{1}{2}\doublehor{3}{5}\end{sd6m46} \otimes
  \begin{dsdrrise}\lsinglehor{2}\ldoublehor{3}{5}\rdoublehor{3}{5}\rsinglehor{6}\end{dsdrrise} \label{eq:da7.3} \\
  \delta_3^1: \left(\begin{dsdrrise}\lsinglehor{1}\rsinglehor{1}\end{dsdrrise} ~,~
    \begin{sd6m24}\strandup{1}{3}\end{sd6m24} ~,~ \begin{sd6m24}\strandup{5}{6}\end{sd6m24}\right) &\to
  \begin{sd6m46}\strandup{1}{2}\end{sd6m46} \otimes
  \begin{dsdrrise}\lsinglehor{2}\rsinglehor{6}\end{dsdrrise}
  \label{eq:da7.4} \\
  \delta_2^1: \left(\begin{dsdrrise}\lsinglehor{1}\ldoublehor{3}{5}\rsinglehor{1}\rdoublehor{3}{5}\end{dsdrrise} ~,~
    \begin{sd6m24}\strandup{1}{3}\strandup{5}{6}\end{sd6m24}\right) &\to
  \begin{sd6m46}\strandup{1}{2}\doublehor{3}{5}\end{sd6m46} \otimes
  \begin{dsdrrise}\lsinglehor{2}\ldoublehor{3}{5}\rdoublehor{3}{5}\rsinglehor{6}\end{dsdrrise} \label{eq:da7.5}
\end{daequation}

Here arrows (\ref{eq:da7.1}) through (\ref{eq:da7.3}) follow respectively from
(\ref{eq:dd7.1}) through (\ref{eq:dd7.3}). Arrow (\ref{eq:da7.4}) follows from
(\ref{eq:dd7.4}) and (\ref{eq:aasep1}). Arrow (\ref{eq:da7.5}) follows from
(\ref{eq:dd7.1}), (\ref{eq:dd7.3}), and (\ref{eq:aasep2}).

%%% Local Variables:
%%% mode: latex
%%% TeX-master: "dacalc"
%%% End:

\section{Relations on the mapping class groupoid}
\label{sec:relmcg}

In this section, we conclude the proof of Theorem \ref{thm:invmappingclass}. In
Section \ref{sec:invarcslides}, we describe how to enumerate the set of
generators of a box tensor product of $\hatda(\tau_i)$, where $\tau_i$ are
arcslides, and how properties of $\hatdd(\tau_i)$ carry over to properties of
$\hatda(\tau_i)$ and their box tensor products. With all these preparations in
place, we prove Equation (\ref{eq:tocheck}) for the involution relation in
Section \ref{sec:relinvolution}, and for the other relations in Section
\ref{sec:relothers}.

\subsection{Compositions of arcslides}\label{sec:invarcslides}

Given an arcslide $\tau$, the description of the set of generators of
$\hatda(\tau)$ follows from that for $\hatdd(\tau)$ and $\hataa(\iz)$. The
generators are classified by their idempotents on the two sides (type $D$
idempotent on the left and type $A$ idempotent on the right). As before, we use
the canonical identification of pairs of points between the pointed matched
circles on the two sides. There are two types of generators in $\hatda(\tau)$:
\begin{itemize}
\item Type $X$, where the idempotents on the two sides contain the same pairs.
\item Type $Y$, where the idempotents on the two sides differ at exactly one
  pair, with $C$ pair occupied on the left and $B$ pair occupied on the right.
\end{itemize}
Using this, and the definition of box tensor product, we can enumerate the set
of generators of
\[ \hatda(\tau_1)\boxtimes\cdots\boxtimes\hatda(\tau_n) \] for a sequence of
arcslides $\tau_1,\dots,\tau_n$. We now describe the procedure in detail.

First, we combine the identification between pairs of points on the starting and
ending pointed matched circles of a single arcslide, to obtain an identification
of pairs on all pointed matched circles appearing in the sequence. Note that
even if the starting and ending pointed matched circle of a sequence is the
same, the identification of pairs between the two, induced by the sequence of
arcslides, may not be the identity. See the triangle relation for an example.

With this identification of pairs throughout a sequence, we can talk about a
pair of points in the sequence, which are pairs of points, one for each pointed
matched circle in the sequence, that are identified to be the same. We assign a
number from 1 to $d$ to each pair of points in the sequence that served as
either the $B$ pair or the $C$ pair of some arcslide, where $d$ is the total
number of such pairs.

Each generator of the box tensor product
$\hatda(\tau_1)\boxtimes\cdots\boxtimes\hatda(\tau_n)$ is of the form
$\mathbf{x}_1\otimes\cdots\otimes\mathbf{x}_n$, where $\mathbf{x}_i$ is a
generator of $\hatda(\tau_i)$ for each $1\le i\le n$, and the right idempotent
of $\mathbf{x}_i$ agrees with the left idempotent of $\mathbf{x}_{i+1}$ for each
$1\le i<n$. A generator $\mathbf{x}_1\otimes\cdots\otimes\mathbf{x}_n$ is
determined by the set of occupied pairs at the starting and ending pointed
matched circles, and at each pointed matched circles in the middle. It is clear
that each unnumbered pair must be either occupied throughout or unoccupied
throughout. For the numbered pairs, the only possible changes are as follows:
suppose that for a certain arcslide $\tau_i$ in the sequence, the $B$ pair is
numbered $b_i$ and the $C$ pair is numbered $c_i$, then it is possible to have
$c_i$, but not $b_i$, occupied in the left idempotent of $\mathbf{x}_i$, and
$b_i$, but not $c_i$, occupied in the right idempotent, with all other pairs
staying the same. This corresponds to choosing $\mathbf{x}_i$ to have type $Y$.

We can therefore specify a \emph{type} of generators by specifying which of the
numbered pairs are occupied at each pointed matched circle. At each arcslide,
the generator is either type $X$ or type $Y$. In the first case, the occupied
pairs must be the same before and after, and in the second case, the $C$ pair
occurs before and is replaced by the $B$ pair. To choose a specific generator of
a given type, it remains to choose which unnumbered pairs to occupy throughout,
so that the total number of occupied pairs is $g$ (half of the total $2g$
pairs).

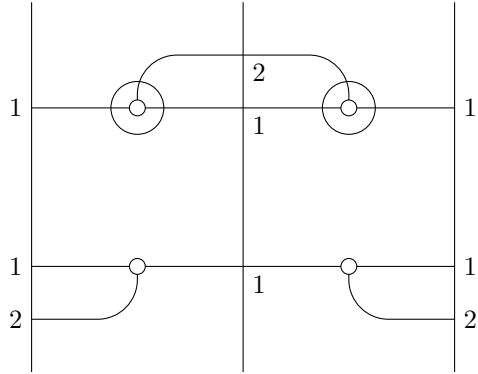
\begin{figure} [h!tb] \centering
  \begin{tikzpicture} [x=20pt,y=20pt]
    \draw (0,1) to (0,8); \draw (4,1) to (4,8); \draw (8,1) to (8,8);
    \draw (0,3) to (8,3); \draw (0,6) to (8,6);
    \draw[rounded corners=15] (0,2) to (2,2) to (2,3);
    \draw[rounded corners=15] (2,6) to (2,7) to (4,7);
    \draw[rounded corners=15] (4,7) to (6,7) to (6,6);
    \draw[rounded corners=15] (6,3) to (6,2) to (8,2);
    \filldraw[fill=white] (2,3) circle (0.15);
    \filldraw[fill=white] (2,6) circle (0.15);
    \filldraw[fill=white] (6,3) circle (0.15);
    \filldraw[fill=white] (6,6) circle (0.15);
    \draw (2,6) circle (0.5); \draw (6,6) circle (0.5);
    \draw (0,3) node[left] {$1$};
    \draw (0,6) node[left] {$1$};
    \draw (0,2) node[left] {$2$};
    \draw (4,7) node[below right] {$2$};
    \draw (4,3) node[below right] {$1$};
    \draw (4,6) node[below right] {$1$};
    \draw (8,3) node[right] {$1$};
    \draw (8,6) node[right] {$1$};
    \draw (8,2) node[right] {$2$};
  \end{tikzpicture}
  \caption{Heegaard diagram for the involution relation.}
  \label{fig:diaginv}
\end{figure}

We now study the involution relation as an example. This is the simplest case,
which nevertheless illustrates most of the reasoning required. One possible
Heegaard diagram for the involution relation is shown in Figure
\ref{fig:diaginv}. There are two numbered pairs ($d = 2$). Pair 1 served as the
$C$ pair and pair 2 served as the $B$ pair for both arcslides. The possible
types of generators are:

\begin{itemize}
\item $_{()}X_{()}X_{()}$
\item $_{(1)}X_{(1)}X_{(1)}$
\item $_{(2)}X_{(2)}X_{(2)}$
\item $_{(12)}X_{(12)}X_{(12)}$
\item $_{(1)}X_{(1)}Y_{(2)}$
\item $_{(1)}Y_{(2)}X_{(2)}$
\end{itemize}

For generators with $X$ at both positions, any combinations of occupying pairs
are possible. For generators with $X$ at first position and $Y$ at second
position, pair 1 and not pair 2 are occupied in the middle, so 1 can be replaced
by 2 at the end. This implies pair 1 and not pair 2 are occupied at the
beginning as well. The same reasoning can be used for type $YX$, and for showing
that type $YY$ is not possible.

Later, we may use $(*)$ to denote an arbitrary subset of the numbered pairs,
that stays the same for a given generator. So for example, we may collect the
first four types above into $_{(*)}X_{(*)}X_{(*)}$.

We now state several general facts about $\hatda(\tau)$ and the box tensor
products of such bimodules. These follow from the corresponding facts about
$\hatdd(\tau)$ in Section \ref{sec:invarcslidesintro}, and the definition of
$\hatda(\tau)$ as $\hataa(\iz)\boxtimes\hatdd(\tau)$.

\begin{remark} [Relation with Heegaard diagram] \label{rem:dacorrespondence}
  Just as in the type $\DD$ case, each generator of $\hatda(\tau)$ corresponds
  to a tuple of points in the standard Heegaard diagram for the arcslide, with
  its left (type $D$) idempotent the set of unoccupied $\alpha$-arcs on the
  left, and its right (type $A$) idempotent the set of occupied $\alpha$-arcs on
  the right. When Heegaard diagrams of arcslides are glued side-by-side along
  their boundaries, the result is a larger Heegaard diagram that now contains
  $\alpha$-circles. Note the boundaries that are glued along are removed from
  the resulting diagram. Each generator
  $\mathbf{x}_1\otimes\cdots\otimes\mathbf{x}_n$ of the box tensor product
  corresponds to a tuple of points, with each $\alpha$ and $\beta$ circles
  containing exactly one point, and each $\alpha$-arc containing at most one
  point.

  As in the type $\DD$ case, each arrow in $\hatda(\tau)$ corresponds to a
  domain away from the basepoint in the Heegaard diagram of the
  arcslide. Likewise, each arrow in the box tensor product
  $\hatda(\tau_1)\boxtimes\cdots\boxtimes\hatda(\tau_n)$ corresponds to a domain
  in the Heegaard diagram obtained by gluing the diagrams for
  $\tau_1,\dots,\tau_n$ in sequence. The multiplicity of the domain on the left
  (resp. right) boundary equals the total multiplicity of the algebra
  coefficients on the left (resp. right) of the arrow. The relation
  $\partial(\partial^\alpha B)=\mathbf{y}-\mathbf{x}$ when domain $B$ represents
  an arrow from $\mathbf{x}$ to $\mathbf{y}$ still holds.
\end{remark}

\begin{remark}[Grading]\label{rem:dagrading}
  The remarks on the grading set of $\hatdd(\tau)$ for an arcslide $\tau$
  extends to a similar statement for $\hatda(\tau)$, and by taking box tensor
  products, extends to $\hatda(\phi,\bm{\tau})$. Given $\phi:F^\circ(\slz_1)\to
  F^\circ(\slz_2)$ and a factorization $\bm{\tau}$ of $\phi$, the bimodule
  $\hatda(\phi,\bm{\tau})$ is graded by a set $S$ having free and transitive
  left-right actions by $G(\slz_1)$ and $G(\slz_2)$. The grading set induces an
  element of $\mathrm{Out}(G(\slz_1),G(\slz_2))$. If $\phi$ begins and ends at
  the same pointed matched circle $\slz$, then it induces an element of the
  outer automorphism group $\mathrm{Out}(G(\slz),G(\slz))$. That element can be
  found from the action of $\phi$ on the homology of the surface. In particular,
  the identity morphism on $F^\circ(\slz)$ induces the identity outer
  isomorphism on $G(\slz)$.
\end{remark}

\begin{remark}[Stabilization]\label{rem:dastabilization}
  Given $\tau:F^\circ(\slz_1)\to F^\circ(\slz_2)$ and its stabilization
  $\mathring{\tau}:F(\mathring{\slz_1})\to F(\mathring{\slz_2})$, the bimodule
  $\hatda(\tau)$ is again an appropriate restriction of
  $\hatda(\mathring{\tau})$. This follows from the corresponding relations
  between $\hataa(\slz)$ and $\hataa(\mathring{\slz})$, for any pointed matched
  circle $\slz$. Taking box tensor products, the stabilization property extends
  to a relation between $\hatda(\phi,\bm{\tau})$ and
  $\hatda(\mathring{\phi},\bm{\mathring{\tau}})$, where $\mathring{\phi}$ is the
  element of the mapping class groupoid that acts as identity on the adjoined
  $\slz^1$, and as $\phi$ elsewhere, and where $\bm{\mathring{\tau}}$ is the
  extension of the factorization $\bm{\tau}$.
\end{remark}

The corresponding duality statements will be left to the end of Section
\ref{sec:relmcg}, where we will have defined all the other bimodule invariants
for surface diffeomorphisms.

\subsection{The involution relation}\label{sec:relinvolution}
In this section we will verify the involution relation. Figure \ref{fig:diaginv}
shows one of the possible cases: overslide in the upward direction. The
computations for overslide in the downward direction, and for underslides over a
pair of points at distance greater than 2 from each other are similar.

Recall that the box tensor product is generated by three types of generators:
\begin{itemize}
\item $_{(*)}X_{(*)}X_{(*)}$
\item $_{(1)}X_{(1)}Y_{(2)}$
\item $_{(1)}Y_{(2)}X_{(2)}$
\end{itemize}
For each type $XY$ generator, there is a corresponding type $YX$ generator that
occupies the same unnumbered pairs. The plan is to cancel out pairs of $XY$ and
$YX$ generators using this correspondence, and show that the resulting bimodule
satisfies the four conditions in Lemma \ref{lem:idrecognize1}.

There are five domains that contribute type $\DA$ arrows of interest. They are
shown in Figures \ref{fig:domainsinv1}.

\begin{figure} [h!tb] \centering
  \begin{tikzpicture} [x=20pt,y=20pt]
    \draw (0,0) to (0,9); \draw (4,0) to (4,9); \draw (8,0) to (8,9);
    \draw (0,3) to (8,3); \draw (0,6) to (8,6);
    \draw (0,1) to (8,1); \draw (0,8) to (8,8);
    \draw[rounded corners=15] (0,2) to (2,2) to (2,3);
    \draw[rounded corners=15] (2,6) to (2,7) to (4,7);
    \draw[rounded corners=15] (4,7) to (6,7) to (6,6);
    \draw[rounded corners=15] (6,3) to (6,2) to (8,2);
    \filldraw[fill=white] (2,3) circle (0.15);
    \filldraw[fill=white] (2,6) circle (0.15);
    \filldraw[fill=white] (6,3) circle (0.15);
    \filldraw[fill=white] (6,6) circle (0.15);
    \draw (2,6) circle (0.5); \draw (6,6) circle (0.5);
    \draw (0,3) node[left] {$1$};
    \draw (0,6) node[left] {$1$};
    \draw (0,2) node[left] {$2$};
    \draw (4,7) node[below right] {$2$};
    \draw (4,3) node[below right] {$1$};
    \draw (4,6) node[below right] {$1$};
    \draw (8,3) node[right] {$1$};
    \draw (8,6) node[right] {$1$};
    \draw (8,2) node[right] {$2$};
    \draw (1,2.5) node {$A$}; \draw (7,2.5) node {$B$};
    \draw (4,6.5) node[left] {$C$};
    \draw (1,7) node {$D$}; \draw (3,2) node {$E$};
  \end{tikzpicture}
  \caption{Domains $A$ through $E$ are connected components of
    $\mathcal{H}\setminus\{\bm{\alpha},\bm{\beta}\}$ containing the respective
    letters. The domains that contribute type $\DA$ arrows of interest are $A$,
    $B$, $C$, $D+C$, and $E+C$. If the upper point in pair 1 (resp. the visible
    point in pair 2) is the topmost (resp. bottommost) point in the pointed
    matched circle, then the domain $D+C$ (resp. $E+C$) does not exist.}
  \label{fig:domainsinv1}
\end{figure}
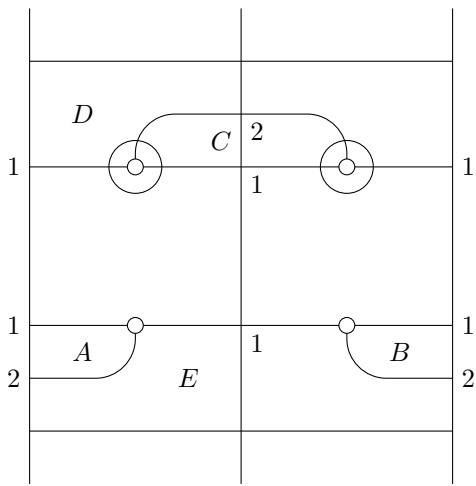

Domain $C$ contributes an arrow with no $A$-side inputs and idempotent $D$-side
output from any $XY$ generator to the corresponding $YX$ generator. This allows
us to cancel all $XY$ and $YX$ generators using the cancellation lemma. We now
focus on the resulting bimodule, with the type $XX$ generators.

This bimodule clearly satisfies (ID-1). Condition (ID-2) can be checked by
explicit grading computations, which uses only the combinatorial features of the
Heegaard diagram. In particular, the fact that the induced map
$\phi\in\mathrm{Out}(G(\mathcal{Z}),G(\mathcal{Z}))$ is identity is equivalent
to the fact that the action of this composition of arcslides on
$H_1(F(\mathcal{Z}))$ is the identity. The stability condition (ID-4) follows
from Remark \ref{rem:dastabilization}.

It remains to verify (ID-3). For this, we need to classify all arrows whose
coefficients have length at most one on both boundaries. Such an arrow either
exists before applying the cancellation lemma, or is produced via a zig-zag.  In
the first case, they correspond to one connected domain between type $XX$
generators. They include trivial horizontal strips in regions away from the
slide, the domain $D+C$, and the domain $E+C$. In the second case, the zig-zag
must be of the form

\[
\begin{tikzpicture} [baseline=(current bounding box.center)]
  \matrix (cancellation) [matrix of math nodes]
  {
    XX & [1cm] XY & [1cm]  \\ [7mm]
    & YX & XX \\
  };
  \draw [->] (cancellation-1-1) to node[below] {$c_1$} (cancellation-2-2);
  \draw [->] (cancellation-1-2) -- (cancellation-2-2);
  \draw [->] (cancellation-1-2) to node[above] {$c_2$} (cancellation-2-3);
\end{tikzpicture}
\]

The coefficient $c_1$ must have length one on the left boundary and length zero
on the right, and $c_2$ must have length zero on the left and length one on the
right, or vice-versa. Looking at the Heegaard diagram, the only possibility is
that $c_1$ is produced by domain $A$ and $c_2$ by domain $B$.

In what follows, we show that for each of the domains $A,B,D+C$, and $E+C$, and
any starting and ending generators with matching idempotents, there is exactly
one arrow. These arrows, together with the ones coming from simple horizontal
strips, cover each length-one interval exactly once, which verifies (ID-3).

Domains $A$ and $B$ are straightforward since they involve only length one
coefficients. The next case is the domain $D+C$. If pair 2 is not occupied, the
arrow follows from (\ref{eq:da1.1}) and (\ref{eq:da2.3}):

\[
\begin{tikzpicture} [x=10pt,y=10pt]
  \begin{subsdbig2}{0}{4}{3}\strandup{1}{3}\end{subsdbig2}
  \begin{subdsdml2}{5}{0}\lsinglehor{3}\rsinglehor{3}\end{subdsdml2}
  \begin{subdsdml2}{5}{8}\lsinglehor{1}\rsinglehor{1}\end{subdsdml2}
  \begin{subssdn}{13}{4}{3}\strandup{1}{3}\end{subssdn}
  \begin{subdsdmr2}{18}{0}{3}\lsinglehor{3}\rsinglehor{3}\end{subdsdmr2}
  \begin{subdsdmr2}{18}{8}{3}\lsinglehor{1}\rsinglehor{1}\end{subdsdmr2}
  \begin{subsdbig2}{26}{4}{3}\strandup{1}{3}\end{subsdbig2}
  \draw [->] (8,7) to (8,5);
  \draw [->] (21,7) to (21,5);
  \draw [->] (7,6) to (4,5);
  \draw [->] (12,7) to (9,6);
  \draw [->] (20,6) to (17,5);
  \draw [->] (25,7) to (22,6);
\end{tikzpicture}
\]

If pair 2 is occupied, then the type $\DA$ arrows on the left side depends on
the ordering. However, in either case we get the same arrow after box
tensoring. If upper interval comes first, it follows from (\ref{eq:da1.2}) and
(\ref{eq:da2.1}) (first diagram below). Otherwise, it follows from
(\ref{eq:da1.5}), (\ref{eq:da1.3}), and (\ref{eq:da2.2}) (second diagram below).

\[
\begin{tikzpicture} [x=10pt,y=10pt]
  \begin{subsdbig2}{0}{4}{3}\strandup{1}{3}\end{subsdbig2}
  \begin{subdsdml2}{5}{0}\lsinglehor{3}\rsinglehor{2}\rsinglehor{3}\end{subdsdml2}
  \begin{subdsdml2}{5}{8}\lsinglehor{1}\rsinglehor{1}\rsinglehor{2}\end{subdsdml2}
  \begin{subssdn}{13}{4}{3}\strandup{1}{3}\singlehor{2}\end{subssdn}
  \begin{subdsdmr2}{18}{0}{3}\lsinglehor{2}\lsinglehor{3}\rsinglehor{3}\end{subdsdmr2}
  \begin{subdsdmr2}{18}{8}{3}\lsinglehor{1}\lsinglehor{2}\rsinglehor{1}\end{subdsdmr2}
  \begin{subsdbig2}{26}{4}{3}\strandup{1}{3}\end{subsdbig2}
  \draw [->] (8,7) to (8,5);
  \draw [->] (21,7) to (21,5);
  \draw [->] (7,6) to (4,5);
  \draw [->] (12,7) to (9,6);
  \draw [->] (20,6) to (17,5);
  \draw [->] (25,7) to (22,6);
\end{tikzpicture}
\]

\[
\begin{tikzpicture} [x=10pt,y=10pt]
  \begin{subsdbig2}{0}{7}{3}\strandup{1}{3}\end{subsdbig2}
  \begin{subdsdml2}{5}{3}{3}\lsinglehor{3}\rsinglehor{2}\rsinglehor{3}\end{subdsdml2}
  \begin{subdsdml2}{5}{15}{3}\lsinglehor{1}\rsinglehor{1}\rsinglehor{2}\end{subdsdml2}
  \begin{subssdn}{13}{5}{3}\singlehor{3}\strandup{1}{2}\end{subssdn}
  \begin{subssdn}{13}{11}{3}\singlehor{1}\strandup{2}{3}\end{subssdn}
  \begin{subdsdmr2}{18}{3}{3}\lsinglehor{2}\lsinglehor{3}\rsinglehor{3}\end{subdsdmr2}
  \begin{subdsdmr2}{18}{9}{3}\lsinglehor{1}\lsinglehor{3}\rsinglehor{3}\end{subdsdmr2}
  \begin{subdsdmr2}{18}{15}{3}\lsinglehor{1}\lsinglehor{2}\rsinglehor{1}\end{subdsdmr2}
  \begin{subsdbig2}{26}{13}{3}\strandup{1}{3}\end{subsdbig2}
  \draw [->] (8,14) to (8,8);
  \draw [->] (7,11) to (4,9);
  \draw [->] (12,13) to (9,11.5);
  \draw [->] (12,7) to (9,10.5);
  \draw [->] (21,14.5) to (21,13.5);
  \draw [->] (21,8.5) to (21,7.5);
  \draw [->] (20,14) to (17,13);
  \draw [->] (20,8) to (17,7);
  \draw [->] (25,15) to (22,14);
\end{tikzpicture}
\]

Now for the domain $E+C$: if pair 1 is occupied, the arrow follows from
(\ref{eq:da6.4}) and (\ref{eq:da7.5}):

\[
\begin{tikzpicture} [x=10pt,y=10pt]
  \begin{subsd6m46}{0}{5}{3}\strandup{1}{2}\doublehor{3}{5}\end{subsd6m46}
  \begin{subdsdrrise}{5}{0}\lsinglehor{2}\ldoublehor{3}{5}\rdoublehor{3}{5}\rsinglehor{6}\end{subdsdrrise}
  \begin{subdsdrrise}{5}{10}\lsinglehor{1}\ldoublehor{3}{5}\rdoublehor{3}{5}\rsinglehor{1}\end{subdsdrrise}
  \begin{subsd6m24}{13}{5}{3}\strandup{1}{3}\strandup{5}{6}\end{subsd6m24}
  \begin{subdsdlrise}{18}{0}{3}\lsinglehor{6}\ldoublehor{3}{5}\rdoublehor{3}{5}\rsinglehor{2}\end{subdsdlrise}
  \begin{subdsdlrise}{18}{10}{3}\lsinglehor{1}\ldoublehor{3}{5}\rdoublehor{3}{5}\rsinglehor{1}\end{subdsdlrise}
  \begin{subsd6m46}{26}{5}{3}\strandup{1}{2}\doublehor{3}{5}\end{subsd6m46}
  \draw [->] (8,9.5) to (8,7.5);
  \draw [->] (21,9.5) to (21,7.5);
  \draw [->] (7,8.5) to (4,7.5);
  \draw [->] (12,9.5) to (9,8.5);
  \draw [->] (20,8.5) to (17,7.5);
  \draw [->] (25,9.5) to (22,8.5);
\end{tikzpicture}
\]
and if pair 1 is unoccupied, the arrow follows from (\ref{eq:da6.3}),
(\ref{eq:da6.2}), and (\ref{eq:da7.4}).

\[
\begin{tikzpicture} [x=10pt,y=10pt]
  \begin{subsd6m46}{0}{7}{3}\strandup{1}{2}\end{subsd6m46}
  \begin{subdsdrrise}{5}{0}\lsinglehor{2}\rsinglehor{6}\end{subdsdrrise}
  \begin{subdsdrrise}{5}{18}\lsinglehor{1}\rsinglehor{1}\end{subdsdrrise}
  \begin{subsd6m24}{13}{4.5}{3}\strandup{5}{6}\end{subsd6m24}
  \begin{subsd6m24}{13}{13.5}{3}\strandup{1}{3}\end{subsd6m24}
  \begin{subdsdlrise}{18}{0}{3}\lsinglehor{6}\rsinglehor{2}\end{subdsdlrise}
  \begin{subdsdlrise}{18}{9}{3}\ldoublehor{3}{5}\rsinglehor{2}\end{subdsdlrise}
  \begin{subdsdlrise}{18}{18}{3}\lsinglehor{1}\rsinglehor{1}\end{subdsdlrise}
  \begin{subsd6m46}{26}{15}{3}\strandup{1}{2}\end{subsd6m46}
  \draw [->] (8,17) to (8,8);
  \draw [->] (21,8.5) to (21,7.5);
  \draw [->] (21,17.5) to (21,16.5);
  \draw [->] (7,12.5) to (4,10.5);
  \draw [->] (12,17) to (9,13);
  \draw [->] (12,8) to (9,12);
  \draw [->] (20,8) to (17,7);
  \draw [->] (20,17) to (17,16);
  \draw [->] (25,18) to (22,17);
\end{tikzpicture}
\]

This finishes the verification of the involution relation, except for the case
of a short underslide. The computations in this case involve size 2 intervals
where the top and bottom points are paired, so we consider them separately. The
diagram is shown in Figure \ref{fig:shortinv}.

\begin{figure} [h!tb] \centering
  \begin{tikzpicture} [x=15pt,y=15pt]
    \draw (0,0) to (0,4); \draw (4,0) to (4,4); \draw (8,0) to (8,4);
    \draw (0,1) to (8,1); \draw (0,3) to (8,3);
    \draw[rounded corners=15] (0,2) to (2,2) to (2,1);
    \draw[rounded corners=15] (2,3) to (2,2) to (4,2);
    \draw[rounded corners=15] (4,2) to (6,2) to (6,3);
    \draw[rounded corners=15] (6,1) to (6,2) to (8,2);
    \draw (4,2.5) node[left] {$B$};
    \draw (4,1.5) node[left] {$A$};
    \filldraw[fill=white] (2,1) circle (0.15);
    \filldraw[fill=white] (2,3) circle (0.15);
    \filldraw[fill=white] (6,1) circle (0.15);
    \filldraw[fill=white] (6,3) circle (0.15);
    \draw (2,3) circle (0.7); \draw (6,3) circle (0.7);
  \end{tikzpicture}
  \caption{Diagram for involution, short underslide case. Domains $A$ and $B$
    are connected components of
    $\mathcal{H}\setminus\{\bm{\alpha},\bm{\beta}\}$ containing the respective
    letters.}
  \label{fig:shortinv}
\end{figure}
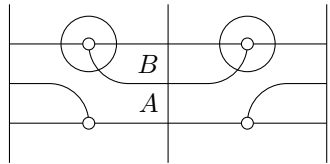

The only difference in the verification is computing the arrow covering the
upper length-one interval. This arrow comes from the domain $A+2B$, and is
produced by (\ref{eq:da5.1}), (\ref{eq:da5.3}), and (\ref{eq:da4.6}):

\[
\begin{tikzpicture} [x=10pt,y=10pt]
  \begin{subssdn}{0}{7}{3}\strandup{2}{3}\end{subssdn}
  \begin{subsdn}{5}{3}{3}\ldoublehor{1}{3}\rdoublehor{1}{3}\end{subsdn}
  \begin{subsdn}{5}{15}{3}\lsinglehor{2}\rsinglehor{2}\end{subsdn}
  \begin{subssdn}{13}{5}{3}\strandup{1}{3}\end{subssdn}
  \begin{subssdn}{13}{11}{3}\strandup{2}{3}\end{subssdn}
  \begin{subsdn}{18}{3}{3}\ldoublehor{1}{3}\rdoublehor{1}{3}\end{subsdn}
  \begin{subsdn}{18}{9}{3}\ldoublehor{1}{3}\rsinglehor{2}\end{subsdn}
  \begin{subsdn}{18}{15}{3}\lsinglehor{2}\rsinglehor{2}\end{subsdn}
  \begin{subssdn}{26}{7}{3}\strandup{2}{3}\end{subssdn}
  \draw [->] (8,14) to (8,8);
  \draw [->] (7,11) to (4,9);
  \draw [->] (12,13) to (9,11.5);
  \draw [->] (12,7) to (9,10.5);
  \draw [->] (21,14.5) to (21,13.5);
  \draw [->] (21,8.5) to (21,7.5);
  \draw [->] (20,14) to (17,13);
  \draw [->] (20,8) to (17,7);
  \draw [->] (25,9) to (22,8);
\end{tikzpicture}
\]
This concludes all cases of the involution relation. Two results follow
immediately from this relation.

\begin{corollary}\label{cor:arcslideinvertible}
  The bimodule $\hatda(\tau)$ is quasi-invertible, and the same is true for any
  box tensor product of such bimodules.
\end{corollary}
\begin{proof}
  For a single arcslide, the quasi-inverse is given by $\hatda(\tau^{-1})$. It
  is clear that box tensor products of quasi-invertible bimodules are also
  quasi-invertible.
\end{proof}

The computations here allow us to prove an uniqueness statement on
$\hatdd(\tau)$. A similar statement is proved in \cite{LOT10c}.
\begin{corollary}
  Let $\tau:F^\circ(\slz_1)\to F^\circ(\slz_2)$ be an arcslide. If a bimodule
  $^{\sla(\slz_1),\sla(-\slz_2)}M$ is stable, has the same generators and
  gradings as $\hatdd(\tau)$, and its type $\DD$ action matches that of
  $\hatdd(\tau)$ on all arrows with total lengths of coefficients at most 3,
  then $M$ is homotopy equivalent to $\hatdd(\tau)$.
\end{corollary}
\begin{proof}
  Let $M_{\DA}=\hataa(\mathbb{I})\boxtimes M$. Since we only used type $\DD$
  arrows whose coefficients have total length at most 3 in this section, we can
  perform the same computations on $M_{\DA}$ as on $\hatda(\tau)$, showing that:
  \[
  M_{\DA}\boxtimes\hatda(\tau^{-1})\simeq \mathbb{I}\simeq
  \hatda(\tau)\boxtimes\hatda(\tau^{-1}). \] Since $\hatda(\tau^{-1})$ is
  quasi-invertible, we see $M_{\DA}$ is homotopy equivalent to
  $\hatda(\tau)$. Since $\hataa(\mathbb{I})$ is also quasi-invertible, we see
  $M$ is homotopy equivalent to $\hatdd(\tau)$.
\end{proof}

\subsection{Other relations on arcslides}\label{sec:relothers}
For each of the other relations on arcslides, we check the conditions in Lemma
\ref{lem:idrecognize2}. Condition (ID-2) is checked by grading computations as
before. (ID-3) follows from Corollary \ref{cor:arcslideinvertible}, with the
quasi-inverse $M'$ being the box tensor product of the inverse arcslides in the
opposite order. (ID-4) follows from Remark \ref{rem:dastabilization} as
before. So it remains to show (ID-1) (with the same technique as given here we
can show (ID-1) for the inverse $M'$).

We now go through each of the remaining relations.

\subsubsection{Triangle}
For the triangle relation, the Heegaard diagram for one of the possible cases is
shown in Figure \ref{fig:diagtriangle}. The other cases differ from this one
only by switching the ordering of the points and underslides with
overslides. The enumeration of generators, and which pairs of generators can be
cancelled, are essentially similar.

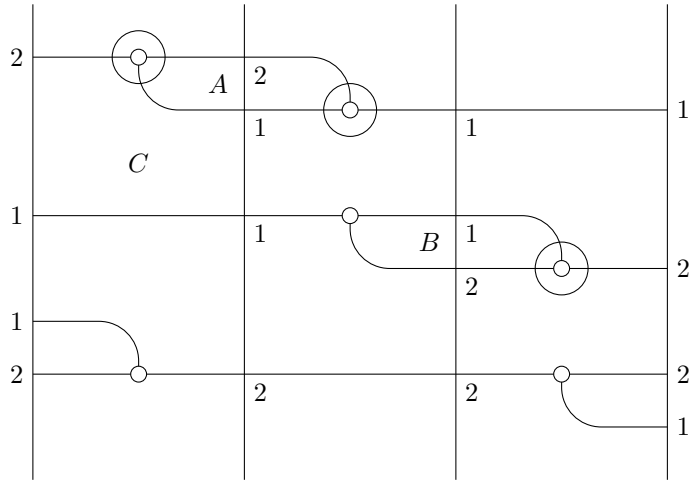
\begin{figure} [h!tb] \centering
  \begin{tikzpicture} [x=20pt,y=20pt]
    \draw (3.5,7.5) node{$A$};
    \draw (7.5,4.5) node{$B$};
    \draw (2,6) node{$C$};
    \draw (0,0) to (0,9); \draw (4,0) to (4,9); \draw (8,0) to (8,9); \draw (12,0) to (12,9);
    \draw (0,2) to (12,2); \draw (0,5) to (8,5); \draw (0,8) to (4,8);
    \draw (8,4) to (12,4); \draw (4,7) to (12,7);
    \draw [rounded corners=15] (0,3) to (2,3) to (2,2);
    \draw [rounded corners=15] (10,2) to (10,1) to (12,1);
    \draw [rounded corners=15] (2,8) to (2,7) to (4,7);
    \draw [rounded corners=15] (4,8) to (6,8) to (6,7);
    \draw [rounded corners=15] (6,5) to (6,4) to (8,4);
    \draw [rounded corners=15] (8,5) to (10,5) to (10,4);
    \filldraw[fill=white] (2,2) circle (0.15);
    \filldraw[fill=white] (2,8) circle (0.15);
    \filldraw[fill=white] (6,5) circle (0.15);
    \filldraw[fill=white] (6,7) circle (0.15);
    \filldraw[fill=white] (10,2) circle (0.15);
    \filldraw[fill=white] (10,4) circle (0.15);
    \draw (2,8) circle (0.5);
    \draw (6,7) circle (0.5);
    \draw (10,4) circle (0.5);
    \draw (0,2) node[left] {$2$};
    \draw (0,3) node[left] {$1$};
    \draw (0,5) node[left] {$1$};
    \draw (0,8) node[left] {$2$};
    \draw (4,2) node[below right] {$2$};
    \draw (4,5) node[below right] {$1$};
    \draw (4,7) node[below right] {$1$};
    \draw (4,8) node[below right] {$2$};
    \draw (8,2) node[below right] {$2$};
    \draw (8,4) node[below right] {$2$};
    \draw (8,5) node[below right] {$1$};
    \draw (8,7) node[below right] {$1$};
    \draw (12,1) node[right] {$1$};
    \draw (12,2) node[right] {$2$};
    \draw (12,4) node[right] {$2$};
    \draw (12,7) node[right] {$1$};
  \end{tikzpicture}
  \caption{Heegaard diagram for the triangle relation. Domains $A$, $B$ and $C$
    are connected components of $\mathcal{H}\setminus\{\bm{\alpha},\bm{\beta}\}$
    containing the respective letters.}
  \label{fig:diagtriangle}
\end{figure}

The roles of the numbered pairs are as follows:
\begin{itemize}
\item Arcslide 1: $C=2, B=1$.
\item Arcslide 2: $C=1, B=2$.
\item Arcslide 3: $C=2, B=1$.
\end{itemize}
Only the sequence $YXY$ is forbidden. For that sequence, pair 1 must be occupied
after the first arcslide, and therefore after the second arcslide, so type $Y$
is not possible at the third arcslide. The possible types are:
\begin{itemize}
\item $_{(*)}X_{(*)}X_{(*)}X_{(*)}$
\item $_{(2)}Y_{(1)}X_{(1)}X_{(1)}$
\item $_{(1)}X_{(1)}Y_{(2)}X_{(2)}$
\item $_{(2)}X_{(2)}X_{(2)}Y_{(1)}$
\item $_{(2)}Y_{(1)}Y_{(2)}X_{(2)}$
\item $_{(1)}X_{(1)}Y_{(2)}Y_{(1)}$
\item $_{(2)}Y_{(1)}Y_{(2)}Y_{(1)}$
\item $_{(1)}X_{(1)}Y_{(2)}X_{(2)}$
\end{itemize}

There are two domains that give rise to cancellable arrows: domain $A$ and $B$
as shown in the figure.

Domain $A$ gives rise to arrows from $YY*$ to $XX*$, and domain $B$ gives rise
to arrows from $*YY$ to $*XX$. So the cancellable arrows are:
\begin{itemize}
\item $YYY \to\,XXY$
\item $YYY \to\,YXX$
\item $_{(1)}XYY_{(1)} \to\,_{(1)}XXX_{(1)}$
\item $_{(2)}YYX_{(2)} \to\,_{(2)}XXX_{(2)}$
\end{itemize}
We choose to cancel everything except the second set $YYY \to\,YXX$ (cancelling
the first set of arrows eliminates the option of cancelling the second). The
remaining generators are:
\[ _{(2)}YXX_{(1)},\quad _{(1)}XYX_{(2)},\quad _{()}XXX_{()}\;\mathrm{and}\;
_{(12)}XXX_{(12)}. \] Since each type of idempotents at the two ends occurs
exactly once, we have verified (ID-1). Note pair 1 at the left becomes pair 2 at
the right, and vice versa, under the bijection of pairs coming from the equality
of pointed matched circles at the two ends.

For the triangle relation, it is not immediately clear that there exists a
refined relative grading where all generators have grading zero, so we give more
details on verifying this condition. Choose a generator in class
$_{(2)}YXX_{(1)}$ as the base generator (with refined grading zero). To verify
that any generator of class $_{(1)}XYX_{(2)}$ has grading zero, it suffices to
check that any potential domain connecting them has the expected grading. The
domain $B+C$ is such a domain. Its grading can be computed to be the same as
that of a simple horizontal strip in the Heegaard diagram for identity, with the
same boundaries at the two sides. If the genus is greater than 2, then
generators of type $_{()}XXX_{()}$ and $_{(12)}XXX_{(12)}$ exist. They are
connected to $_{(2)}YXX_{(1)}$ or $_{(1)}XYX_{(2)}$ by horizontal strips above
either $A$ or $B$. These domains also have the same gradings as the simple
horizontal strips in the diagram for identity with the same boundaries, so the
latter two types of generators must also have grading zero.

\subsubsection{Commutativity}
The Heegaard diagram for one of the cases of the commutativity relation is shown
in Figure \ref{fig:diagcomm} (as in the triangle case, the other possibilities
are similar).

\begin{figure} [h!tb] \centering
  \begin{tikzpicture} [x=15pt,y=15pt]
    \draw (10,1.5) node{$B$};
    \draw (6,6.5) node{$A$};
    \draw (0,0) to (0,11); \draw (4,0) to (4,11); \draw (8,0) to (8,11); \draw (12,0) to (12,11); \draw (16,0) to (16,11);
    \draw (0,2) to (16,2); \draw (0,4) to (16,4); \draw (0,7) to (16,7); \draw (0,9) to (16,9);
    \draw[rounded corners=15] (0,5) to (6,5) to (6,4);
    \draw[rounded corners=15] (6,2) to (6,1) to (14,1) to (14,2);
    \draw[rounded corners=15] (14,4) to (14,5) to (16,5);
    \draw[rounded corners=15] (0,10) to (2,10) to (2,9);
    \draw[rounded corners=15] (2,7) to (2,6) to (10,6) to (10,7);
    \draw[rounded corners=15] (10,9) to (10,10) to (16,10);
    \filldraw[fill=white] (2,9) circle (0.15);
    \filldraw[fill=white] (2,7) circle (0.15);
    \draw (2,7) circle (0.5);
    \filldraw[fill=white] (6,2) circle (0.15);
    \filldraw[fill=white] (6,4) circle (0.15);
    \draw (6,2) circle (0.5);
    \filldraw[fill=white] (10,9) circle (0.15);
    \filldraw[fill=white] (10,7) circle (0.15);
    \draw (10,7) circle (0.5);
    \filldraw[fill=white] (14,2) circle (0.15);
    \filldraw[fill=white] (14,4) circle (0.15);
    \draw (14,2) circle (0.5);
    \draw (0,2) node[left] {$3$}; \draw (0,4) node[left] {$3$}; \draw (0,5) node[left] {$4$};
    \draw (0,7) node[left] {$1$}; \draw (0,9) node[left] {$1$}; \draw (0,10) node[left] {$2$};
    \draw (4,6) node[below right] {$2$}; \draw (4,7) node[below right] {$1$}; \draw (4,9) node[below right] {$1$};
    \draw (8,1) node[below right] {$4$}; \draw (8,2) node[below right] {$3$}; \draw (8,4) node[below right] {$3$};
    \draw (12,7) node[below right] {$1$}; \draw (12,9) node[below right] {$1$}; \draw (12,10) node[below right] {$2$};
    \draw (16,2) node[right] {$3$}; \draw (16,4) node[right] {$3$}; \draw (16,5) node[right] {$4$};
  \end{tikzpicture}
  \caption{Heegaard diagram for the commutativity relation. Domains $A$ and $B$
    are connected components of $\mathcal{H}\setminus\{\bm{\alpha},\bm{\beta}\}$
    containing the respective letters.}
  \label{fig:diagcomm}
\end{figure}
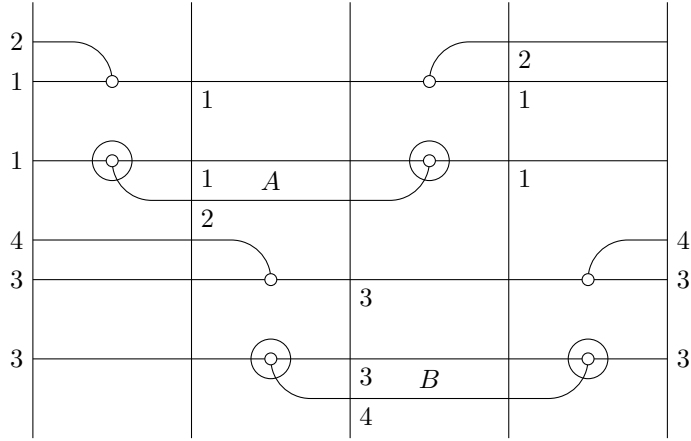

The role of the numbered pairs are as follows:
\begin{itemize}
\item Arcslide 1: $C=1, B=2$.
\item Arcslide 2: $C=3, B=4$.
\item Arcslide 3: $C=1, B=2$.
\item Arcslide 4: $C=3, B=4$.
\end{itemize}

The restriction on the types is that at most one of the types at arcslides 1 and
3 can be $Y$, and at most one at arcslides 2 and 4 can be $Y$. The possibilities
are:

\begin{itemize}
\item $_{(*)}X_{(*)}X_{(*)}X_{(*)}X_{(*)}$
\item $_{(1)}Y_{(2)}X_{(2)}X_{(2)}X_{(2)}$ with pairs 3 and/or 4 possibly added
  to each idempotent (similarly in the next three types).
\item $_{(1)}X_{(1)}X_{(1)}Y_{(2)}X_{(2)}$ with 3 and/or 4 possibly added.
\item $_{(3)}X_{(3)}Y_{(4)}X_{(4)}X_{(4)}$ with 1 and/or 2 possibly added.
\item $_{(3)}X_{(3)}X_{(3)}X_{(3)}Y_{(4)}$ with 1 and/or 2 possibly added.
\item $_{(13)}Y_{(23)}Y_{(24)}X_{(24)}X_{(24)}$
\item $_{(13)}Y_{(23)}X_{(23)}X_{(23)}Y_{(24)}$
\item $_{(13)}X_{(13)}Y_{(14)}Y_{(24)}X_{(24)}$
\item $_{(13)}X_{(13)}X_{(13)}Y_{(23)}Y_{(24)}$
\end{itemize}

The two domains giving rise to cancellable arrows are labelled $A$ and $B$ in
the figure.

Domain $A$ gives rise to arrows from $YaXb$ to $XaYb$, for any valid choice of
$a,b\in\{X,Y\}$. Likewise, domain $B$ gives rise to arrows from $aYbX$ to
$aXbY$. So the cancellable arrows are:
\begin{itemize}
\item $_{(1)}YXXX_{(2)} \to\,_{(1)}XXYX_{(2)}$ with 3 and/or 4 possibly added
\item $_{(3)}XYXX_{(4)} \to\,_{(3)}XXXY_{(4)}$ with 1 and/or 2 possibly added
\end{itemize}
and

\begin{tikzpicture}
  \matrix (cancellation) [matrix of math nodes]
  {
    _{(13)}YYXX_{(24)} & [1cm] _{(13)}YXXY_{(24)} \\ [7mm]
    _{(13)}XYYX_{(24)} & [1cm] _{(13)}XXYY_{(24)} \\
  };
  \draw [->] (cancellation-1-1) -- (cancellation-1-2);
  \draw [->] (cancellation-1-1) -- (cancellation-2-1);
  \draw [->] (cancellation-1-2) -- (cancellation-2-2);
  \draw [->] (cancellation-2-1) -- (cancellation-2-2);
\end{tikzpicture}

The first two arrows cancel all generators with one $Y$. For generators with two
$Y$'s, we can either cancel both horizontal arrows or both vertical arrows in
the square above. In the end, only generators of type $_{(*)}XXXX_{(*)}$ remain,
which checks (ID-1).

\subsubsection{Left and right pentagon}

The Heegaard diagram for one of the cases of the left pentagon relation is shown
in Figure \ref{fig:diagpentagon}. Other cases of the left and right pentagon
relation are similar.

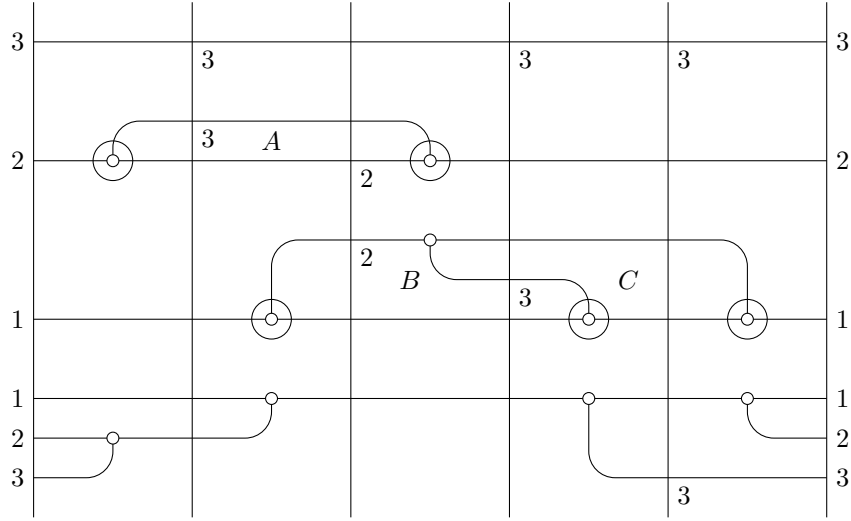
\begin{figure} [h!tb] \centering
  \begin{tikzpicture} [x=15pt,y=15pt]
    \draw (6,9.5) node{$A$};
    \draw (9.5,6) node{$B$};
    \draw (15,6) node{$C$};
    \draw (0,0) to (0,13); \draw (4,0) to (4,13); \draw (8,0) to (8,13);
    \draw (12,0) to (12,13); \draw (16,0) to (16,13); \draw (20,0) to (20,13);
    \draw (0,3) to (20,3); \draw (0,5) to (20,5); \draw (0,9) to (20,9); \draw (0,12) to (20,12);
    \draw[rounded corners=10] (0,2) to (6,2) to (6,3); \draw[rounded corners=10] (0,1) to (2,1) to (2,2);
    \draw[rounded corners=10] (20,2) to (18,2) to (18,3); \draw[rounded corners=10] (20,1) to (14,1) to (14,3);
    \draw[rounded corners=10] (6,5) to (6,7) to (18,7) to (18,5);
    \draw[rounded corners=10] (10,7) to (10,6) to (14,6) to (14,5);
    \draw[rounded corners=10] (2,9) to (2,10) to (10,10) to (10,9);
    \draw (0,1) node[left] {$3$}; \draw (0,2) node[left] {$2$}; \draw (0,3) node[left] {$1$};
    \draw (0,5) node[left] {$1$}; \draw (0,9) node[left] {$2$}; \draw (0,12) node[left] {$3$};
    \draw (4,10) node[below right] {$3$}; \draw (4,12) node[below right] {$3$};
    \draw (8,7) node[below right] {$2$}; \draw (8,9) node[below right] {$2$};
    \draw (12,6) node[below right] {$3$}; \draw (12,12) node[below right] {$3$};
    \draw (16,1) node[below right] {$3$}; \draw (16,12) node[below right] {$3$};
    \draw (20,1) node[right] {$3$}; \draw (20,2) node[right] {$2$}; \draw (20,3) node[right] {$1$};
    \draw (20,5) node[right] {$1$}; \draw (20,9) node[right] {$2$}; \draw (20,12) node[right] {$3$};
    \filldraw[fill=white] (2,2) circle (0.15); \filldraw[fill=white] (2,9) circle (0.15); \draw (2,9) circle (0.5);
    \filldraw[fill=white] (6,3) circle (0.15); \filldraw[fill=white] (6,5) circle (0.15); \draw (6,5) circle (0.5);
    \filldraw[fill=white] (10,7) circle (0.15); \filldraw[fill=white] (10,9) circle (0.15); \draw (10,9) circle (0.5);
    \filldraw[fill=white] (14,3) circle (0.15); \filldraw[fill=white] (14,5) circle (0.15); \draw (14,5) circle (0.5);
    \filldraw[fill=white] (18,3) circle (0.15); \filldraw[fill=white] (18,5) circle (0.15); \draw (18,5) circle (0.5);
  \end{tikzpicture}
  \caption{Heegaard diagram for the pentagon relation. Domains $A$, $B$, and $C$
    are connected components of $\mathcal{H}\setminus\{\bm{\alpha},\bm{\beta}\}$
    containing the respective letters.}
  \label{fig:diagpentagon}
\end{figure}

The role of the numbered pairs are as follows:
\begin{itemize}
\item Arcslide 1: $C=2, B=3$.
\item Arcslide 2: $C=1, B=2$.
\item Arcslide 3: $C=2, B=3$.
\item Arcslide 4: $C=1, B=3$.
\item Arcslide 5: $C=1, B=2$.
\end{itemize}

The possible types are:

\begin{itemize}
\item $_{(*)}X_{(*)}X_{(*)}X_{(*)}X_{(*)}X_{(*)}$
\item $_{(12)}Y_{(13)}Y_{(23)}X_{(23)}X_{(23)}X_{(23)}$
\item $_{(12)}X_{(12)}X_{(12)}Y_{(13)}X_{(13)}Y_{(23)}$
\item $_{(12)}Y_{(13)}X_{(13)}X_{(13)}X_{(13)}Y_{(23)}$
\item $_{(1)}X_{(1)}Y_{(2)}Y_{(3)}X_{(3)}X_{(3)}$
\item $_{(2)}Y_{(3)}X_{(3)}X_{(3)}X_{(3)}X_{(3)}$, $_{(12)}Y_{(13)}X_{(13)}X_{(13)}X_{(13)}X_{(13)}$
\item $_{(1)}X_{(1)}Y_{(2)}X_{(2)}X_{(2)}X_{(2)}$, $_{(13)}X_{(13)}Y_{(23)}X_{(23)}X_{(23)}X_{(23)}$
\item $_{(2)}X_{(2)}X_{(2)}Y_{(3)}X_{(3)}X_{(3)}$, $_{(12)}X_{(12)}X_{(12)}Y_{(13)}X_{(13)}X_{(13)}$
\item $_{(1)}X_{(1)}X_{(1)}X_{(1)}Y_{(3)}X_{(3)}$, $_{(12)}X_{(12)}X_{(12)}X_{(12)}Y_{(23)}X_{(23)}$
\item $_{(1)}X_{(1)}X_{(1)}X_{(1)}X_{(1)}Y_{(2)}$, $_{(13)}X_{(13)}X_{(13)}X_{(13)}X_{(13)}Y_{(23)}$
\end{itemize}

Domains $A, B,$ and $C$ give rise to the following arrows:
\begin{itemize}
\item Domain $A$: $XXY**\to\,YXX**$
\item Domain $B$: $*XXY*\to\,*YYX*$
\item Domain $C$: $**YXY\to\,**XYX$
\end{itemize}

Other domains that may give arrows are $B+C$ and $A+B$. We first analyze $B+C$,
showing that it will always contribute an arrow whenever idempotent matches. The
calculation involves box tensoring the four type $\DA$ bimodules as shown in
Figure \ref{fig:domainbplusc}.

\begin{figure} [h!tb] \centering
  \begin{tikzpicture} [x=15pt,y=15pt]
    \draw (0,-0.5) to (0,2.5); \draw (4,-0.5) to (4,2.5); \draw (8,-0.5) to (8,2.5);
    \draw (12,-0.5) to (12,2.5); \draw (16,-0.5) to (16,2.5);
    \draw (-0.5,0) to (16.5,0);
    \draw[rounded corners=15] (2,0) to (2,2) to (14,2) to (14,0);
    \draw[rounded corners=10] (6,2) to (6,1) to (10,1) to (10,0);
    \filldraw[fill=white] (2,0) circle (0.15);
    \filldraw[fill=white] (6,2) circle (0.15);
    \filldraw[fill=white] (10,0) circle (0.15);
    \filldraw[fill=white] (14,0) circle (0.15);
    \draw (2,0) circle (0.5); \draw (10,0) circle (0.5);
    \draw (14,0) circle (0.5);
    \draw (2,-1) node{$1$};
    \draw (6,-1) node{$2$};
    \draw (10,-1) node{$3$};
    \draw (14,-1) node{$4$};
  \end{tikzpicture}
  \caption{Domain $B+C$.}
  \label{fig:domainbplusc}
\end{figure}
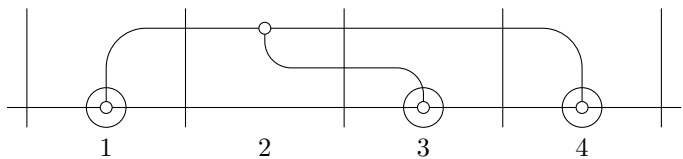

The arrow needed in the fourth piece is simple. For the third piece, there are
several ways to cover the domain. First, if pair 3 is unoccupied in the middle
pieces, we use (\ref{eq:da1.1}):

\[
\begin{tikzpicture} [x=10pt,y=10pt]
  \begin{subssdn}{0}{2}{3}\strandup{1}{3}\end{subssdn}
  \begin{subdsdmr2}{5}{0}\lsinglehor{3}\rsinglehor{3}\end{subdsdmr2}
  \begin{subdsdmr2}{5}{6}\lsinglehor{1}\rsinglehor{1}\end{subdsdmr2}
  \begin{subsdbig2}{13}{4}\strandup{1}{3}\end{subsdbig2}
  \draw [->] (8,5.5) to (8,4.5);
  \draw [->] (7,5) to (4,4);
  \draw [->] (12,6) to (9,5);
\end{tikzpicture}
\]

If pair 3 is occupied, there are two different ways: using (\ref{eq:da1.2}), or
using (\ref{eq:da1.5}) and (\ref{eq:da1.3}):

\[
\begin{tikzpicture} [x=10pt,y=10pt]
  \begin{subssdn}{0}{2}{3}\strandup{1}{3}\singlehor{2}\end{subssdn}
  \begin{subdsdmr2}{5}{0}\lsinglehor{2}\lsinglehor{3}\rsinglehor{3}\end{subdsdmr2}
  \begin{subdsdmr2}{5}{6}\lsinglehor{1}\lsinglehor{2}\rsinglehor{1}\end{subdsdmr2}
  \begin{subsdbig2}{13}{4}\strandup{1}{3}\end{subsdbig2}
  \draw [->] (8,5.5) to (8,4.5);
  \draw [->] (7,5) to (4,4);
  \draw [->] (12,6) to (9,5);
\end{tikzpicture}
\]
and
\[
\begin{tikzpicture} [x=10pt,y=10pt]
  \begin{subssdn}{0}{2}{3}\strandup{1}{2}\singlehor{3}\end{subssdn}
  \begin{subssdn}{0}{8}{3}\strandup{2}{3}\singlehor{1}\end{subssdn}
  \begin{subdsdmr2}{5}{0}\lsinglehor{2}\lsinglehor{3}\rsinglehor{3}\end{subdsdmr2}
  \begin{subdsdmr2}{5}{6}\lsinglehor{1}\lsinglehor{3}\rsinglehor{3}\end{subdsdmr2}
  \begin{subdsdmr2}{5}{12}\lsinglehor{1}\lsinglehor{2}\rsinglehor{1}\end{subdsdmr2}
  \begin{subsdbig2}{13}{10}\strandup{1}{3}\end{subsdbig2}
  \draw [->] (8,5.5) to (8,4.5);
  \draw [->] (8,11.5) to (8,10.5);
  \draw [->] (7,5) to (4,4);
  \draw [->] (7,11) to (4,10);
  \draw [->] (12,12) to (9,11);
\end{tikzpicture}
\]

Now looking at the possible arrows in the second piece, we see there is always
exactly one way to continue forming the arrow in the box tensor product to the
second piece (and then trivially to the first piece). If pair 3 is unoccupied,
we use (\ref{eq:da2.3}). If pair 3 is occupied, we use either (\ref{eq:da2.1})
or (\ref{eq:da2.2}), depending on the ordering $<_\slz$. This shows that the
domain $B+C$ gives rise to arrows:
\begin{itemize}
\item $*_{(1)}X_{(1)}X_{(1)}X_{(1)}Y_{(2)}
  \to\,*_{(1)}Y_{(2)}X_{(2)}X_{(2)}X_{(2)}$
\item $*_{(13)}X_{(13)}X_{(13)}X_{(13)}Y_{(23)}
  \to\,*_{(13)}Y_{(23)}X_{(23)}X_{(23)}X_{(23)}$
\end{itemize}

Finally, we consider the domain $A+B$. This domain potentially contributes
arrows of the form $XXXX*\to YYXX*$. The only possible choice of idempotents is
\begin{itemize}
\item $_{(12)}X_{(12)}X_{(12)}X_{(12)}Y_{(23)}X_{(23)}
  \to\,_{(12)}Y_{(13)}Y_{(23)}X_{(23)}X_{(23)}X_{(23)}$
\end{itemize}
Rather than computing the type $\DA$ arrows for this domain like in the previous
case, we note that the sequence
\[_{(12)}XXYXY_{(23)}\to\, _{(12)}YXXXY_{(23)}\to\, _{(12)}YYXXX_{(23)} \] must
cancel against something in the type $\DA$ structure equation. This is possible
only if the domain $A+B$ contributes an arrow.

In summary, the cancellable arrows are:
\begin{itemize}
\item $_{(1)}XXXXY_{(2)} \to\,_{(1)}XYXXX_{(2)}$
\item $_{(13)}XXXXY_{(23)} \to\,_{(13)}XYXXX_{(23)}$
\item $_{(2)}XXYXX_{(3)} \to\,_{(2)}YXXXX_{(3)}$
\item $_{(12)}XXYXX_{(13)} \to\,_{(12)}YXXXX_{(13)}$
\item $_{(1)}XXXYX_{(3)} \to\,_{(1)}XYYXX_{(3)}$
\end{itemize}
and

\begin{tikzpicture}
  \matrix (cancellation) [matrix of math nodes]
  {
    _{(12)}XXYXY_{(23)} & [1cm] _{(12)}YXXXY_{(23)} \\ [7mm]
    _{(12)}XXXYX_{(23)} & [1cm] _{(12)}YYXXX_{(23)} \\
  };
  \draw [->] (cancellation-1-1) -- (cancellation-1-2);
  \draw [->] (cancellation-1-1) -- (cancellation-2-1);
  \draw [->] (cancellation-1-2) -- (cancellation-2-2);
  \draw [->] (cancellation-2-1) -- (cancellation-2-2);
\end{tikzpicture}

The four types of generators starting with ``12'' and ending with ``23'' form
the square above, and is cancelled using either the horizontal or vertical
arrows. The other ten types of generators containing at least one $Y$ are
cancelled using the first five arrows. So only generators of type
$_{(*)}XXXXX_{(*)}$ remain, which verifies (ID-1).

This concludes the proof of Theorem \ref{thm:invmappingclass}, showing the
bimodule $\hatda(\phi,\bm{\tau})$ is independent of the choice of factorization
$\bm{\tau}$ up to homotopy equivalence. This allows us to write
$\widehatit{CFDA}(\phi)$ for the homotopy equivalence class of
$\hatda(\phi,\bm{\tau})$, and define the other invariants
$\widehatit{CFDD}(\phi), \widehatit{CFAA}(\phi)$, and $\widehatit{CFAD}(\phi)$
combinatorially by box tensoring with appropriate identity bimodules.

We finish with a discussion of how duality on $\hatdd(\tau)$ extends to the
other bimodule invariants.
\begin{lemma} \label{lem:daduality} For any element $\phi:F^\circ(\slz_1)\to
  F^\circ(\slz_2)$ of the strongly-based mapping class groupoid, we have
  \begin{equation} \label{eq:daduality}
    ^{\sla(-\slz_1)}\widehatit{CFAD}(\phi^{-1})_{\sla(-\slz_2)} \simeq
    \overline{^{\sla(\slz_1)}\widehatit{CFDA}(\phi)_{\sla(\slz_2)}}.
  \end{equation}
\end{lemma}
\begin{proof}
  First, we consider the case of an arcslide $\tau$. Using the definition of
  $\widehatit{CFAD}$ and the fact that $\widehatit{CFAA}(\mathbb{I})$ and
  $\widehatit{CFDD}(\mathbb{I})$ are quasi-inverses, we have:
  \[ ^{\sla(\slz_2),\sla(-\slz_1)}\widehatit{CFDD}(\tau^{-1}) \simeq
  \tensor[^{\sla(-\slz_1)}]{\widehatit{CFAD}(\tau^{-1})}{_{\sla(-\slz_2)}}
  \boxtimes
  \tensor[^{\sla(\slz_2),\sla(-\slz_2)}]{\widehatit{CFDD}(\mathbb{I}_{\slz_2})}{}.
  \]
  On the other hand,
  \begin{eqnarray}
    ^{\sla(-\slz_1),\sla(\slz_2)}\overline{\widehatit{CFDD}(\tau)} &\simeq&
    \overline{^{\sla(\slz_1)}\widehatit{CFDA}(\tau)_{\sla(\slz_2)}
      \boxtimes
      \tensor[^{\sla(\slz_2),\sla(-\slz_2)}]
      {\widehatit{CFDD}(\mathbb{I}_{\slz_2})}{}}
    \nonumber \\ &\simeq&
    \tensor[^{\sla(-\slz_1)}]{\overline{\widehatit{CFDA}(\tau)}}{_{\sla(-\slz_2)}}
    \boxtimes
    \tensor[^{\sla(-\slz_2),\sla(\slz_2)}]
    {\overline{\widehatit{CFDD}(\mathbb{I}_{\slz_2})}}{} \nonumber
  \end{eqnarray}
  By the remarks on duality at the end of Section \ref{sec:invarcslidesintro}
  (Equation (\ref{eq:cfddarcslidedual})), we see $\widehatit{CFDD}(\tau^{-1})$
  and $\overline{\widehatit{CFDD}(\tau)}$ are homotopy equivalent after
  switching the two algebra actions. It is also clear from the construction of
  $\widehatit{CFDD}(\mathbb{I}_{\slz_2})$ that it is isomorphic to
  $\overline{\widehatit{CFDD}(\mathbb{I}_{\slz_2})}$ after switching the algebra
  actions. This implies Equation (\ref{eq:daduality}) for arcslides $\tau$.

  For a general surface diffeomorphism $\phi$, factor it into arcslides
  $\tau_i$. The statement then follows from the case of arcslides, and the fact
  that taking duals distributes over the box tensor product.
\end{proof}

%%% Local Variables:
%%% mode: latex
%%% TeX-master: "dacalc"
%%% End:

\section{The 3-manifold invariant}\label{sec:3manifold}

In this section, we prove Theorem \ref{thm:invclosed}, showing that the homotopy
type of the chain complex $\widehatit{HF}$ given in Construction
\ref{constr:invclosed} does not depend on the choices made. There are two main
components of the proof, given by the two lemmas below.

Let $\mathrm{MCG}_0(\slz^g)$ denote the strongly-based mapping class group on
$F_{g,1}$, parametrized by the genus $g$ split pointed matched circle
$\slz^g$. Recall that $\hg$ denotes the 0-framed handlebody, and its orientation
reversal $-\hg$ is the $\infty$-framed handlebody.

\begin{lemma}[Stabilization]\label{lem:stabinv}
  Let $\psi$ be an element of $\mathrm{MCG}_0(\slz^g)$. Consider $F_{g+1,1}$,
  parametrized by $\slz^{g+1}$, as the surface obtained from $F_{g,1}$ by adding
  a handle in a neighborhood of the basepoint. Let $\mathring{\psi}$ be the
  element of $\mathrm{MCG}_0(\slz^{g+1})$ that fixes the new handle and acts as
  $\psi$ elsewhere. Then
  \begin{equation} \label{eq:stabinv}
    \left(\widehatit{CFAA}(\psi)\boxtimes \widehatit{CFD}(\hg) \right)
    \boxtimes \widehatit{CFD}(-\hg) \simeq
    \left(\widehatit{CFAA}(\mathring{\psi})\boxtimes
      \widehatit{CFD}(\mathbf{H}^{g+1}) \right)
    \boxtimes\widehatit{CFD}(-\mathbf{H}^{g+1}).
  \end{equation}
\end{lemma}

\begin{definition}
  Define $\mathrm{MCG}_0^{\beta}(\slz^g)$ to be the subgroup of
  $\mathrm{MCG}_0(\slz^g)$ consisting of maps that extend to automorphisms of
  $\hg$. Likewise, define $\mathrm{MCG}_0^{\alpha}(\slz^g)$ to be the subgroup
  of $\mathrm{MCG}_0(\slz^g)$ consisting of maps that extend to automorphisms of
  $-\hg$ (using identification $\slz^g=-\slz^g$ to consider $-\hg$ as
  parametrized by $\slz^g$).
\end{definition}

\begin{lemma}[Reparametrization of the 0-framed handlebody]
  \label{lem:handlebodyinv}
  For each element $\phi\in\mathrm{MCG}_0^\beta(\slz^g)$, we have
  \begin{equation} \label{eq:reparametrization}
    ^{\sla(-\slz^g)}\widehatit{CFAD}(\phi)_{\sla(-\slz^g)} \boxtimes
    \tensor[^{\sla(-\slz^g)}]{\widehatit{CFD}(\hg)}{} \simeq
    \tensor[^{\sla(-\slz^g)}]{\widehatit{CFD}(\hg)}{}.
  \end{equation}
\end{lemma}

We first show that these two lemmas imply Theorem \ref{thm:invclosed}.

\begin{proof}[Proof of Theorem \ref{thm:invclosed}]
  There are two choices made in Construction \ref{constr:invclosed}: the choice
  of Heegaard splitting $Y=Y_1\cup Y_2$, and choice of parametrizations of $Y_1$
  and $Y_2$ by standard handlebodies. It is well-known that any two Heegaard
  splittings become isotopic after a finite number of stabilizations. Also, any
  stabilization can be isotopied to the standard one, adding a handle in a
  neighborhood of the basepoint. If $\psi$ is a valid choice of element in
  $\mathrm{MCG}_0(\slz^g)$ in the second stage of the construction, then
  $\mathring{\psi}$ is a valid choice of element in $\mathrm{MCG}_0(\slz^{g+1})$
  after a standard stabilization. So Lemma \ref{lem:stabinv} implies that
  Construction \ref{constr:invclosed} is invariant under stabilizations.

  Now we consider choice of parametrizations of $Y_1$ and $Y_2$. Recall
  $\psi=\overline{f_{2*}}^{-1}\circ u\circ f_{1*}$, where $u:\partial Y_1\to
  -\partial Y_2$ is the gluing map, $f_1:\hg\to Y_1$ is the parametrization of
  $Y_1$ by $\hg$, and $f_2:-\hg\to Y_2$ is the parametrization of $Y_2$ by
  $-\hg$. Hence, changing parametrization of $Y_1$ changes $\psi$ to
  $\psi'=\psi\circ\phi_1$, where $\phi_1\in\mathrm{MCG}_0^\beta(\slz^g)$, and
  changing parametrization of $Y_2$ changes $\psi$ to
  $\psi'=\overline{\phi_2}^{-1}\circ\psi$, where
  $\phi_2\in\mathrm{MCG}_0^\alpha(\slz^g)$.

  It remains to show the following:
  \begin{align}
    \widehatit{CFAA}(\psi\circ\phi_1)\boxtimes\widehatit{CFD}(\hg) &\simeq
    \widehatit{CFAA}(\psi)\boxtimes\widehatit{CFD}(\hg), \nonumber \\
    \widehatit{CFAA}(\overline{\phi_2}^{-1}\circ\psi)\boxtimes
    \widehatit{CFD}(-\hg) &\simeq
    \widehatit{CFAA}(\psi)\boxtimes\widehatit{CFD}(-\hg) \nonumber
  \end{align}
  for $\phi_1\in\mathrm{MCG}_0^\beta(\slz^g)$ and
  $\phi_2\in\mathrm{MCG}_0^\alpha(\slz^g)$.

  The first equation follows directly from Lemma \ref{lem:handlebodyinv}:
  \begin{align}
    \widehatit{CFAA}(\psi\circ\phi_1)\boxtimes\widehatit{CFD}(\hg)
    &\simeq \widehatit{CFAA}(\psi)\boxtimes
    \widehatit{CFAD}(\phi_1)\boxtimes \widehatit{CFD}(\hg) \nonumber \\
    &\simeq \widehatit{CFAA}(\psi)\boxtimes \widehatit{CFD}(\hg). \nonumber
  \end{align}

  For the second equation, by taking the dual of Equation
  \ref{eq:reparametrization}, and using Lemma \ref{lem:daduality}, we get
  \begin{equation} \label{eq:reparametrization2}
    ^{\sla(\slz^g)}\widehatit{CFDA}(\phi^{-1})_{\sla(\slz^g)} \boxtimes
    \tensor[^{\sla(\slz^g)}]{\widehatit{CFD}(-\hg)}{} \simeq
    \tensor[^{\sla(\slz^g)}]{\widehatit{CFD}(-\hg)}{}
  \end{equation}
  for any $\phi\in\mathrm{MCG}_0^\beta(\slz^g)$. Then the second equation
  follows as:
  \begin{align}
    \widehatit{CFAA}(\overline{\phi_2}^{-1}\circ\psi)\boxtimes
    \widehatit{CFD}(-\hg) &\simeq
    \widehatit{CFAA}(\psi)\boxtimes\widehatit{CFDA}(\overline{\phi_2}^{-1})
    \boxtimes \widehatit{CFD}(-\hg) \nonumber \\
    &\simeq \widehatit{CFAA}(\psi)\boxtimes \widehatit{CFD}(-\hg). \nonumber
  \end{align}
  since $\phi_2\in\mathrm{MCG}_0^\alpha(\slz^g)$ implies
  $\overline{\phi_2}\in\mathrm{MCG}_0^\beta(\slz^g)$.
\end{proof}

We now prove the two lemmas, starting with stabilization invariance.

\begin{proof}[Proof of Lemma \ref{lem:stabinv}]
  Choose factorization $\bm{\tau}$ for $\psi$, then $\mathring{\bm{\tau}}$ is a
  factorization for $\mathring{\psi}$. Choose $\hatda(\psi,\bm{\tau})$ and
  $\hatda(\mathring{\psi},\mathring{\bm{\tau}})$ as models for the
  $\widehatit{CFDA}$ invariants behind the $\widehatit{CFAA}$ invariants. The
  lemma then follows from the stabilization property for
  $\hatda(\psi,\bm{\tau})$. We can see this by comparing the Heegaard diagrams
  underlying the two sides of Equation (\ref{eq:stabinv}). First, the Heegaard
  diagram for $\hatda(\mathring{\psi},\mathring{\bm{\tau}})$ is constructed from
  that for $\hatda(\psi,\bm{\tau})$ by adjoining a horizontal ``strip'' of
  diagrams for the identity diffeomorphism of the genus 1 surface at the
  top. Likewise, the Heegaard diagrams of $\mathbf{H}^{g+1}$ and
  $-\mathbf{H}^{g+1}$ are obtained from that of $\hg$ and $-\hg$ by adjoining
  diagrams of $\mathbf{H}^1$ and $-\mathbf{H}^1$ to the top. These constructions
  are combined in Figure \ref{fig:stabinv}.

  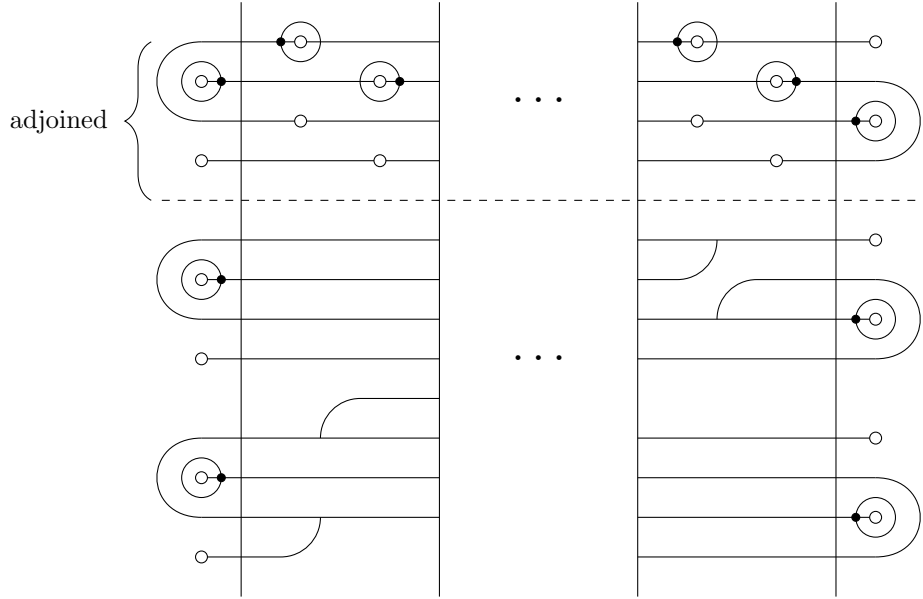
\begin{figure}
    \begin{tikzpicture} [x=15pt,y=15pt]
      \draw (0,0) to (0,15); \draw (5,0) to (5,15); \draw (10,0) to (10,15); \draw (15,0) to (15,15);
      \draw[dashed] (-2,10) to (17,10);
      % Left adjoined portion
      \draw (-1,11) to (5,11); \draw (-1,12) to (5,12); \draw (-1,13) to (5,13); \draw (-1,14) to (5,14);
      \filldraw[fill=white] (1.5,12) circle (0.15);
      \filldraw[fill=white] (1.5,14) circle (0.15);
      \filldraw[fill=white] (3.5,11) circle (0.15);
      \filldraw[fill=white] (3.5,13) circle (0.15);
      \draw (1.5,14) circle (0.5);
      \draw (3.5,13) circle (0.5);
      \filldraw[fill=black] (1,14) circle (0.1);
      \filldraw[fill=black] (4,13) circle (0.1);
      % Right adjoined portion
      \draw (10,11) to (16,11); \draw (10,12) to (16,12); \draw (10,13) to (16,13); \draw (10,14) to (16,14);
      \filldraw[fill=white] (11.5,12) circle (0.15);
      \filldraw[fill=white] (11.5,14) circle (0.15);
      \filldraw[fill=white] (13.5,11) circle (0.15);
      \filldraw[fill=white] (13.5,13) circle (0.15);
      \draw (11.5,14) circle (0.5);
      \draw (13.5,13) circle (0.5);
      \filldraw[fill=black] (11,14) circle (0.1);
      \filldraw[fill=black] (14,13) circle (0.1);
      % Left below
      \draw[rounded corners=15] (-1,1) to (2,1) to (2,2);
      \draw (-1,2) to (5,2); \draw (-1,3) to (5,3); \draw (-1,4) to (5,4);
      \draw[rounded corners=15] (2,4) to (2,5) to (5,5);
      \draw (-1,6) to (5,6); \draw (-1,7) to (5,7); \draw (-1,8) to (5,8); \draw (-1,9) to (5,9);
      % Left type D
      \filldraw[fill=white] (-1,11) circle (0.15);
      \filldraw[fill=white] (-1,13) circle (0.15);
      \draw (-1,13) circle (0.5);
      \filldraw[fill=black] (-0.5,13) circle (0.1);
      \draw (-1,14) .. controls (-2.5,14) and (-2.5,12) .. (-1,12);
      \filldraw[fill=white] (-1,6) circle (0.15);
      \filldraw[fill=white] (-1,8) circle (0.15);
      \draw (-1,8) circle (0.5);
      \filldraw[fill=black] (-0.5,8) circle (0.1);
      \draw (-1,9) .. controls (-2.5,9) and (-2.5,7) .. (-1,7);
      \filldraw[fill=white] (-1,1) circle (0.15);
      \filldraw[fill=white] (-1,3) circle (0.15);
      \draw (-1,3) circle (0.5);
      \filldraw[fill=black] (-0.5,3) circle (0.1);
      \draw (-1,4) .. controls (-2.5,4) and (-2.5,2) .. (-1,2);
      % Right below
      \draw (10,1) to (16,1); \draw (10,2) to (16,2); \draw (10,3) to (16,3); \draw (10,4) to (16,4);
      \draw (10,6) to (16,6); \draw (10,7) to (16,7); \draw (10,9) to (16,9);
      \draw [rounded corners=15] (10,8) to (12,8) to (12,9);
      \draw [rounded corners=15] (12,7) to (12,8) to (16,8);
      % Right type D
      \filldraw[fill=white] (16,12) circle (0.15);
      \filldraw[fill=white] (16,14) circle (0.15);
      \draw (16,12) circle (0.5);
      \filldraw[fill=black] (15.5,12) circle (0.1);
      \draw (16,13) .. controls (17.5,13) and (17.5,11) .. (16,11);
      \filldraw[fill=white] (16,7) circle (0.15);
      \filldraw[fill=white] (16,9) circle (0.15);
      \draw (16,7) circle (0.5);
      \filldraw[fill=black] (15.5,7) circle (0.1);
      \draw (16,8) .. controls (17.5,8) and (17.5,6) .. (16,6);
      \filldraw[fill=white] (16,2) circle (0.15);
      \filldraw[fill=white] (16,4) circle (0.15);
      \draw (16,2) circle (0.5);
      \filldraw[fill=black] (15.5,2) circle (0.1);
      \draw (16,3) .. controls (17.5,3) and (17.5,1) .. (16,1);
      % Dots in the middle
      \draw (7.5,12.5) node {\larger[3]$\mathbf{\cdots}$};
      \draw (7.5,6) node {\larger[3]$\mathbf{\cdots}$};
      % Comment at right
      \draw [decorate,decoration={brace,amplitude=10pt},xshift=-4pt,yshift=0pt]
      (-2,10) -- (-2,14) node [black,midway,xshift=-35]{adjoined};
    \end{tikzpicture}
    \caption{Proof of invariance under stabilization.}
    \label{fig:stabinv}
  \end{figure}

  By Remark \ref{rem:dacorrespondence}, generators in the chain complex
  \begin{equation}\label{eq:stabinvcx1}
    \left( \widehatit{CFAA}(\psi) \boxtimes \widehatit{CFD}(\hg) \right)
    \boxtimes \widehatit{CFD}(-\hg)
  \end{equation}
  correspond to certain tuples of intersection points in the part of the diagram
  below the dashed line in Figure \ref{fig:stabinv}, while generators in the
  chain complex
  \begin{equation}\label{eq:stabinvcx2}
    \left( \widehatit{CFAA}(\mathring{\psi}) \boxtimes
      \widehatit{CFD}(\mathbf{H}^{g+1}) \right)
    \boxtimes \widehatit{CFD}(-\mathbf{H}^{g+1})
  \end{equation}
  correspond to certain tuples of intersection points in the full
  diagram. Likewise, there is a correspondence between arrows in the type $\DA$
  action on the two sides, and domains in appropriate parts of the diagram.

  The choice of intersection points in the adjoined portion of the diagram is
  forced (as marked in the figure), which means that it is the same for all
  generators in (\ref{eq:stabinvcx2}). So there is a one-to-one correspondence
  between generators in (\ref{eq:stabinvcx1}) and
  (\ref{eq:stabinvcx2}). Moreover, since there are no closed domains above the
  dashed line, all arrows in (\ref{eq:stabinvcx2}) automatically have domains
  restricted below the dashed line. By Remark \ref{rem:dastabilization}, there
  is a one-to-one correspondence between these arrows and the arrows in
  (\ref{eq:stabinvcx1}). This shows the chain complexes (\ref{eq:stabinvcx1})
  and (\ref{eq:stabinvcx2}) are isomorphic, proving Lemma \ref{lem:stabinv}.
\end{proof}

For Lemma \ref{lem:handlebodyinv}, we need to show
\[ \widehatit{CFAD}(\phi)\boxtimes\widehatit{CFD}(\hg) \simeq
\widehatit{CFD}(\hg), \] for any $\phi\in\mathrm{MCG}_0^\beta(\slz_g)$. It
suffices to verify the equation for a set of generators of
$\mathrm{MCG}_0^\beta(\slz_g)$.

We find generators for the strongly-based mapping class group by appealing to
results on the usual mapping class group. Let $F_g$ be the genus $g$ surface
with a basepoint. Let $\mathrm{MCG}(F_g)$ be the group of isotopy classes of
diffeomorphisms on $F_g$ that fixes the basepoint, with isotopies also required
to fix the basepoint. It is related to $\mathrm{MCG}_0(\slz_g)$ by a short exact
sequence (see \cite[Section 4.2.5]{FM}):

\begin{center}
  \begin{tikzpicture}[baseline=(current bounding box.center)]
    \matrix(m)[matrix of math nodes,column sep=15]
    {0 & \mathbb{Z} & \mathrm{MCG}_0(\slz_g) & \mathrm{MCG}(F_g) & 0 \\};
    \path[->]
    (m-1-1) edge (m-1-2)
    (m-1-2) edge node[above] {$\tau_\partial$} (m-1-3)
    (m-1-3) edge (m-1-4)
    (m-1-4) edge (m-1-5);
  \end{tikzpicture}
\end{center}

Here $\tau_\partial$ maps the generator of $\mathbb{Z}$ to the \emph{boundary
  Dehn twist} in $\mathrm{MCG}_0(\slz_g)$. This is the element that performs a
Dehn twist along a loop parallel to the boundary of $F_{g,1}$.

There is likewise a short exact sequence:

\begin{center}
  \begin{tikzpicture}[baseline=(current bounding box.center)]
    \matrix(m)[matrix of math nodes,column sep=15]
    {0 & \mathbb{Z} & \mathrm{MCG}_0^\beta(\slz_g) & \mathrm{MCG}^\beta(F_g) & 0 \\};
    \path[->]
    (m-1-1) edge (m-1-2)
    (m-1-2) edge node[above] {$\tau_\partial$} (m-1-3)
    (m-1-3) edge (m-1-4)
    (m-1-4) edge (m-1-5);
  \end{tikzpicture}
\end{center}
where $\mathrm{MCG}^\beta(F_g)$ is the subgroup of $\mathrm{MCG}(F_g)$
consisting of restrictions of automorphisms of the 0-framed handlebody
$\hg$. This exact sequence shows that a generating set of
$\mathrm{MCG}_0^\beta(\slz_g)$ can be obtained by adding the boundary Dehn twist
to the lifting of a generating set of $\mathrm{MCG}^\beta(F_g)$.

A generating set of $\mathrm{MCG}^\beta(F_g)$ is given in \cite{Suzuki} (the
corresponding notation in that paper is $\mathrm{MCG}^*(F_g)$). We reproduce the
list of generators, together with the action of each generator on $\pi_1(F_g)$
here. For an element $\psi\in\mathrm{MCG}^\beta(F_g)$, let
$\psi_\sharp:\pi_1(F_g)\to\pi_1(F_g)$ be its action on $\pi_1(F_g)$. We let
$a_1,b_1,\dots,a_g,b_g$ be a set of standard generators of $\pi_1(F_g)$, with
each $b_i$ contractible in the handlebody, and each $a_i$ intersecting $b_i$
once. Let $s_i=a_i^{-1}b_i^{-1}a_ib_i$, so that $s_n\dots s_2s_1=1$ is a
relation in $\pi_1(F_g)$. In \cite{Suzuki}, a genus $g$ surface is considered as
a sphere with $g$ handles attached. Each handle, together with its immediate
base, is called a \emph{knob}. We refer to that paper for diagrams and geometric
description of these generators.

\begin{theorem}[Suzuki, \cite{Suzuki}]\label{thm:gensuzuki}
  The group $\mathrm{MCG}^\beta(F_g)$ is generated by $\rho, \omega_1, \tau_1,
  \rho_{12}, \theta_{12}$ and $\xi_{12}$, whose actions on $\pi_1(F_g)$ are the
  following:
  \begin{itemize}
  \item Cyclic translation of handles: $\rho_\sharp: a_i\to a_{i+1},~b_i\to
    b_{i+1}$, where indices are taken modulo $g$.
  \item Twisting a knob: $\omega_{1\sharp}: a_1\to a_1^{-1}s_1^{-1},~b_1\to
    a_1^{-1}b_1^{-1}a_1,~ a_j\to a_j,~b_j\to b_j$ for $2\le j\le n$.
  \item Twisting a handle, or Dehn twist: $\tau_{1\sharp}:a_1\to
    a_1b_1^{-1},~b_1\to b_1,~a_j\to a_j,~ b_j\to b_j$ for $2\le j\le n$.
  \item Interchanging two knobs: $\rho_{12\sharp}: a_1\to s_1^{-1}a_2s_1,~a_2\to
    a_1,~b_1\to s_1^{-1}b_2s_1,~ b_2\to b_1,~a_j\to a_j,~b_j\to b_j$ for $3\le
    j\le n$.
  \item Sliding along $a_2$: $\theta_{12\sharp}: a_1\to
    a_1(b_2^{-1}a_2^{-1}b_2),~a_j\to a_j$ for $j\neq 1$, $b_2\to
    a_2b_2(a_1^{-1}b_1a_1)(b_2^{-1}a_2^{-1}b_2),~b_j\to b_j$ for $j\neq 2$.
  \item Sliding along $b_2$: $\xi_{12\sharp}: a_1\to
    b_1a_1b_2^{-1}s_2(a_1^{-1}b_1^{-1}a_1),~ a_2\to
    a_2b_2(a_1^{-1}b_1^{-1}a_1)b_2^{-1}$, $a_j\to a_j$ for $j\neq 1,2,~b_i\to
    b_i$ for $1\le i\le g$.
  \end{itemize}
\end{theorem}

Of these, only $\rho$ is non-local in the sense that it is not restricted to a
part of the surface with fixed genus. All other generators are restricted to a
genus 1 or 2 part of the surface. We can remove $\rho$ in favor of other local
generators, by writing:
\[ \rho^{-1} =
\rho_{12}\circ\rho_{23}\circ\cdots\circ\rho_{g-1,g}\circ(\omega_{g\sharp})^{-2}, \]
where $\omega_{g\sharp}$ is similar to $\omega_{1\sharp}$, except acting on the
$g$th handle, and $\rho_{i,i+1}$ interchanges the $i$th and $(i+1)$th knobs.
The equation can be verified by comparing the actions of two sides on
$\pi_1(F_g)$: the initial $(\omega_{g\sharp})^{-2}$ has the effect of
conjugating $a_g$ and $b_g$ by $s_g^{-1}$. After interchanging the knobs in
succession, the action of the right side on $\pi_1(F_g)$ is $a_1\to
s_1^{-1}s_2^{-1}\cdots s_g^{-1}a_gs_g\cdots s_2s_1$, $a_2\to a_1$, $a_3\to a_2$,
and so on, and similarly for the $b_i$'s. We then apply the relation $s_g\cdots
s_2s_1=1$.

Since the boundary Dehn twist equals $\rho^g$, the $g$th power of the cyclic
translation of handles, the same generators also generate the group
$\mathrm{MCG}_0^\beta(\slz_g)$.

So we have proved the following:
\begin{corollary} \label{cor:genhandlebody} The group
  $\mathrm{MCG}_0^\beta(F_g)$ is generated by $\omega_1, \omega_g, \tau_1,
  \theta_{12}, \xi_{12}$ and $\rho_{i,i+1}$ for $1\le i\le g-1$.
\end{corollary}

Each of the generators in Corollary \ref{cor:genhandlebody} is confined to one
or two knobs on the surface. Our strategy will be to check Equation
(\ref{eq:reparametrization}) on a surface of the corresponding genus (1 or 2),
then extend the result to the general case. First, we compute a decomposition of
these generators into arcslides. An arcslide with $B$ pair $(b_1,b_2)$ and $C$
pair $(c_1,c_2)$, with $b_1$ sliding over $c_1$, is denoted $b_1\to c_1$. The
points are always labeled 0 to $4g-1$ from left to right. The results are:
\begin{eqnarray}
  \rho_{12}:&& 3\to 4~~6\to 7~~5\to 6~~4\to 5~~2\to 3~~5\to 6~~4\to 5~~3\to 4 
  \nonumber \\
  && 1\to 2~~4\to 5~~3\to 4~~2\to 3~~0\to 1~~3\to 4~~2\to 3~~1\to 2
  \nonumber \\
  \theta_{12}:&& 4\to 3~~1\to 0~~1\to 2~~5\to 4~~6\to 5 \nonumber \\
  \xi_{12}:&& 0\to 1~~3\to 4~~6\to 7~~6\to 5~~2\to 3~~1\to 2~~3\to 2
  \nonumber \\
  \omega_1:&& 2\to 3~~1\to 2~~2\to 3~~1\to 2~~2\to 3~~1\to 2 \nonumber \\
  \tau_1:&& 2\to 3 \nonumber
\end{eqnarray}
To verify these decompositions, we compute their actions on
$\pi_1(F^\circ(\slz^g))$. For any pointed matched circle $\slz$, recall that the
surface with circle boundary $F^\circ(\slz)$ is formed by attaching 1-handles to
$Z$ along the matched pairs of points in $\mathbf{a}\subset Z$, then gluing in a
solid disk on the other side. Choosing $z\in Z$ as the basepoint, the
fundamental group of $F^\circ(\slz)$ is generated freely by paths through the
1-handles. We choose the following orientation for the generators of the
fundamental group. For the genus 1 cases:
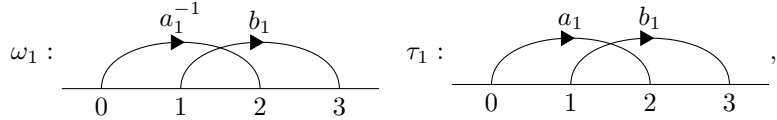
\[ \omega_1:
\begin{tikzpicture} [x=30pt,y=30pt,baseline=(current bounding box.center)]
  \draw (0.5,0) to (4.5,0);
  \draw (1,0) node[below] {$0$};\draw (2,0) node[below] {$1$};
  \draw (3,0) node[below] {$2$};\draw (4,0) node[below] {$3$};
  \draw (1,0) to [out=90,in=90] node {\midarrow} node [above] {$a_1^{-1}$} (3,0);
  \draw (2,0) to [out=90,in=90] node {\midarrow} node [above] {$b_1$} (4,0);
\end{tikzpicture}
\quad
\tau_1:
\begin{tikzpicture} [x=30pt,y=30pt,baseline=(current bounding box.center)]
  \draw (0.5,0) to (4.5,0);
  \draw (1,0) node[below] {$0$};\draw (2,0) node[below] {$1$};
  \draw (3,0) node[below] {$2$};\draw (4,0) node[below] {$3$};
  \draw (1,0) to [out=90,in=90] node {\midarrow} node [above] {$a_1$} (3,0);
  \draw (2,0) to [out=90,in=90] node {\midarrow} node [above] {$b_1$} (4,0);
\end{tikzpicture},
\]
and for all genus 2 cases:
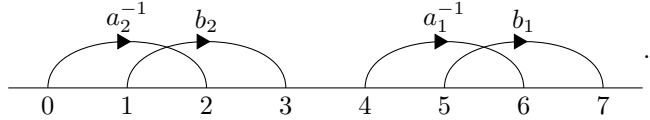
\[\begin{tikzpicture} [x=30pt,y=30pt,baseline=(current bounding box.center)]
  \draw (0.5,0) to (8.5,0);
  \draw (1,0) node[below] {$0$};\draw (2,0) node[below] {$1$};
  \draw (3,0) node[below] {$2$};\draw (4,0) node[below] {$3$};
  \draw (5,0) node[below] {$4$};\draw (6,0) node[below] {$5$};
  \draw (7,0) node[below] {$6$};\draw (8,0) node[below] {$7$};
  \draw (1,0) to [out=90,in=90] node {\midarrow} node [above] {$a_2^{-1}$} (3,0);
  \draw (2,0) to [out=90,in=90] node {\midarrow} node [above] {$b_2$} (4,0);
  \draw (5,0) to [out=90,in=90] node {\midarrow} node [above] {$a_1^{-1}$} (7,0);
  \draw (6,0) to [out=90,in=90] node {\midarrow} node [above] {$b_1$} (8,0);
\end{tikzpicture}.
\]

An arcslide $\tau:\slz_1\to\slz_2$ induces an action
\[ \tau_*:\pi_1(F^\circ(\slz_1)) \to \pi_1(F^\circ(\slz_2)) \] on the
fundamental groups. We describe this action by expressing each generator of
$\pi_1(F^\circ(\slz_2))$ (corresponding to a pair of matched points in $\slz_2$)
in terms of the images under $\tau_*$ of generators of
$\pi_1(F^\circ(\slz_1))$. This can be computed from the definition of arcslides
(for example, see Figure 3 in \cite{LOT10c}). The results are shown in Figures
\ref{fig:actionunderslides} and \ref{fig:actionoverslides}. For example, the
first diagram means that if the two displayed handles in the starting pointed
matched circle correspond to generators $a_1$ and $b_1$, then the two displayed
handles in the ending pointed matched circle correspond to $\tau_*(a_1b_1)$ and
$\tau_*(b_1)$ (the relation for handles unaffected by the arcslide is clear).

\begin{figure}
  \[
  \begin{tikzpicture} [x=13pt,y=13pt,baseline=(current bounding box.center)]
    \draw (0,0) to [out=90,in=90] node {\midarrow} node [above] {$a_1$} (4,0);
    \draw (3,0) to [out=90,in=90] node {\midarrow} node [above] {$b_1$} (7,0);
    \draw (-0.5,0) to (0.5,0); \draw (2.5,0) to (4.5,0);
    \draw (6.5,0) to (7.5,0); \draw [white] (0,3) to (7,3);
    \draw (1.5,0) node {$\cdots$}; \draw (5.5,0) node {$\cdots$};
    \draw [->] (4,-0.5) to (3,-0.5);
  \end{tikzpicture}
  \Rightarrow
  \begin{tikzpicture} [x=13pt,y=13pt,baseline=(current bounding box.center)]
    \draw (0,0) to [out=90,in=90] node {\midarrow} node [above] {$a_1b_1$} (6,0);
    \draw (3,0) to [out=90,in=90] node {\midarrow} node [above] {$b_1$} (7,0);
    \draw (-0.5,0) to (0.5,0); \draw (2.5,0) to (3.5,0);
    \draw (5.5,0) to (7.5,0); \draw [white] (0,3) to (7,3);
    \draw (1.5,0) node {$\cdots$}; \draw (4.5,0) node {$\cdots$};
  \end{tikzpicture}
  \]
  \[
  \begin{tikzpicture} [x=13pt,y=13pt,baseline=(current bounding box.center)]
    \draw (0,0) to [out=90,in=90] node {\midarrow} node [above] {$a_1$} (6,0);
    \draw (3,0) to [out=90,in=90] node {\midarrow} node [above] {$b_1$} (7,0);
    \draw (-0.5,0) to (0.5,0); \draw (2.5,0) to (3.5,0);
    \draw (5.5,0) to (7.5,0); \draw [white] (0,3) to (7,3);
    \draw (1.5,0) node {$\cdots$}; \draw (4.5,0) node {$\cdots$};
    \draw [->] (6,-0.5) to (7,-0.5);
  \end{tikzpicture}
  \Rightarrow
  \begin{tikzpicture} [x=13pt,y=13pt,baseline=(current bounding box.center)]
    \draw (0,0) to [out=90,in=90] node {\midarrow} node [above] {$a_1b_1^{-1}$} (4,0);
    \draw (3,0) to [out=90,in=90] node {\midarrow} node [above] {$b_1$} (7,0);
    \draw (-0.5,0) to (0.5,0); \draw (2.5,0) to (4.5,0);
    \draw (6.5,0) to (7.5,0); \draw [white] (0,3) to (7,3);
    \draw (1.5,0) node {$\cdots$}; \draw (5.5,0) node {$\cdots$};
  \end{tikzpicture}
  \]
  \[
  \begin{tikzpicture} [x=13pt,y=13pt,baseline=(current bounding box.center)]
    \draw (0,0) to [out=90,in=90] node {\midarrow} node [above] {$a_1$} (4,0);
    \draw (3,0) to [out=90,in=90] node {\midarrow} node [above] {$b_1$} (7,0);
    \draw (-0.5,0) to (0.5,0); \draw (2.5,0) to (4.5,0);
    \draw (6.5,0) to (7.5,0); \draw [white] (0,3) to (7,3);
    \draw (1.5,0) node {$\cdots$}; \draw (5.5,0) node {$\cdots$};
    \draw [->] (3,-0.5) to (4,-0.5);
  \end{tikzpicture}
  \Rightarrow
  \begin{tikzpicture} [x=13pt,y=13pt,baseline=(current bounding box.center)]
    \draw (0,0) to [out=90,in=90] node {\midarrow} node [above] {$a_1$} (4,0);
    \draw (1,0) to [out=90,in=90] node {\midarrow} node [above] {$a_1b_1$} (7,0);
    \draw (-0.5,0) to (1.5,0); \draw (3.5,0) to (4.5,0);
    \draw (6.5,0) to (7.5,0); \draw [white] (0,3) to (7,3);
    \draw (2.5,0) node {$\cdots$}; \draw (5.5,0) node {$\cdots$};
  \end{tikzpicture}
  \]
  \[
  \begin{tikzpicture} [x=13pt,y=13pt,baseline=(current bounding box.center)]
    \draw (0,0) to [out=90,in=90] node {\midarrow} node [above] {$a_1$} (4,0);
    \draw (1,0) to [out=90,in=90] node {\midarrow} node [above] {$b_1$} (7,0);
    \draw (-0.5,0) to (1.5,0); \draw (3.5,0) to (4.5,0);
    \draw (6.5,0) to (7.5,0); \draw [white] (0,3) to (7,3);
    \draw (2.5,0) node {$\cdots$}; \draw (5.5,0) node {$\cdots$};
    \draw [->] (1,-0.5) to (0,-0.5);
  \end{tikzpicture}
  \Rightarrow
  \begin{tikzpicture} [x=13pt,y=13pt,baseline=(current bounding box.center)]
    \draw (0,0) to [out=90,in=90] node {\midarrow} node [above] {$a_1$} (4,0);
    \draw (3,0) to [out=90,in=90] node {\midarrow} node [above] {$a_1^{-1}b_1$} (7,0);
    \draw (-0.5,0) to (0.5,0); \draw (2.5,0) to (4.5,0);
    \draw (6.5,0) to (7.5,0); \draw [white] (0,3) to (7,3);
    \draw (1.5,0) node {$\cdots$}; \draw (5.5,0) node {$\cdots$};
  \end{tikzpicture}
  \]
  \caption{Actions of arcslides on the fundamental group (underslide).}
  \label{fig:actionunderslides}
\end{figure}
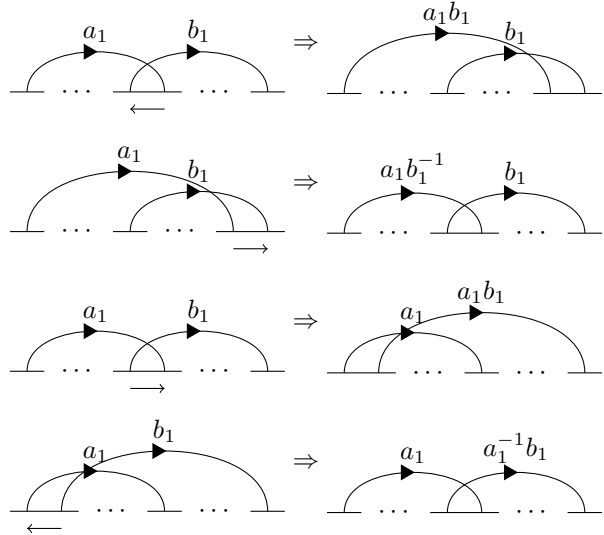

\begin{figure}
  \[
  \begin{tikzpicture} [x=13pt,y=13pt,baseline=(current bounding box.center)]
    \draw (0,0) to [out=90,in=90] node {\midarrow} node [above] {$a_1$} (4,0);
    \draw (1,0) to [out=90,in=90] node {\midarrow} node [above] {$b_1$} (7,0);
    \draw (-0.5,0) to (1.5,0); \draw (3.5,0) to (4.5,0);
    \draw (6.5,0) to (7.5,0); \draw [white] (0,3) to (7,3);
    \draw (2.5,0) node {$\cdots$}; \draw (5.5,0) node {$\cdots$};
    \draw [->] (0,-0.5) to (1,-0.5);
  \end{tikzpicture}
  \Rightarrow
  \begin{tikzpicture} [x=13pt,y=13pt,baseline=(current bounding box.center)]
    \draw (3,0) to [out=90,in=90] node {\midarrow} node [above] {$a_1^{-1}b_1$} (7,0);
    \draw (0,0) to [out=90,in=90] node {\midarrow} node [above] {$b_1$} (6,0);
    \draw (-0.5,0) to (0.5,0); \draw (2.5,0) to (3.5,0);
    \draw (5.5,0) to (7.5,0); \draw [white] (0,3) to (7,3);
    \draw (1.5,0) node {$\cdots$}; \draw (4.5,0) node {$\cdots$};
  \end{tikzpicture}
  \]
  \[
  \begin{tikzpicture} [x=13pt,y=13pt,baseline=(current bounding box.center)]
    \draw (3,0) to [out=90,in=90] node {\midarrow} node [above] {$a_1$} (7,0);
    \draw (0,0) to [out=90,in=90] node {\midarrow} node [above] {$b_1$} (6,0);
    \draw (-0.5,0) to (0.5,0); \draw (2.5,0) to (3.5,0);
    \draw (5.5,0) to (7.5,0); \draw [white] (0,3) to (7,3);
    \draw (1.5,0) node {$\cdots$}; \draw (4.5,0) node {$\cdots$};
    \draw [->] (7,-0.5) to (6,-0.5);
  \end{tikzpicture}
  \Rightarrow
  \begin{tikzpicture} [x=13pt,y=13pt,baseline=(current bounding box.center)]
    \draw (0,0) to [out=90,in=90] node {\midarrow} node [above] {$b_1a_1^{-1}$} (4,0);
    \draw (1,0) to [out=90,in=90] node {\midarrow} node [above] {$b_1$} (7,0);
    \draw (-0.5,0) to (1.5,0); \draw (3.5,0) to (4.5,0);
    \draw (6.5,0) to (7.5,0); \draw [white] (0,3) to (7,3);
    \draw (2.5,0) node {$\cdots$}; \draw (5.5,0) node {$\cdots$};
  \end{tikzpicture}
  \]
  \[
  \begin{tikzpicture} [x=13pt,y=13pt,baseline=(current bounding box.center)]
    \draw (0,0) to [out=90,in=90] node {\midarrow} node [above] {$a_1$} (3,0);
    \draw (4,0) to [out=90,in=90] node {\midarrow} node [above] {$b_1$} (7,0);
    \draw (-0.5,0) to (0.5,0); \draw (2.5,0) to (4.5,0);
    \draw (6.5,0) to (7.5,0); \draw [white] (0,3.5) to (7,3.5);
    \draw (1.5,0) node {$\cdots$}; \draw (5.5,0) node {$\cdots$};
    \draw [->] (3,-0.5) to (4,-0.5);
  \end{tikzpicture}
  \Rightarrow
  \begin{tikzpicture} [x=13pt,y=13pt,baseline=(current bounding box.center)]
    \draw (0,0) to [out=90,in=90] node {\midarrow} node [above] {$a_1b_1$} (7,0);
    \draw (3,0) to [out=90,in=90] node {\midarrow} node [above] {$b_1$} (6,0);
    \draw (-0.5,0) to (0.5,0); \draw (2.5,0) to (3.5,0);
    \draw (5.5,0) to (7.5,0); \draw [white] (0,3.5) to (7,3.5);
    \draw (1.5,0) node {$\cdots$}; \draw (4.5,0) node {$\cdots$};
  \end{tikzpicture}
  \]
  \[
  \begin{tikzpicture} [x=13pt,y=13pt,baseline=(current bounding box.center)]
    \draw (0,0) to [out=90,in=90] node {\midarrow} node [above] {$a_1$} (7,0);
    \draw (3,0) to [out=90,in=90] node {\midarrow} node [above] {$b_1$} (6,0);
    \draw (-0.5,0) to (0.5,0); \draw (2.5,0) to (3.5,0);
    \draw (5.5,0) to (7.5,0); \draw [white] (0,3.5) to (7,3.5);
    \draw (1.5,0) node {$\cdots$}; \draw (4.5,0) node {$\cdots$};
    \draw [->] (7,-0.5) to (6,-0.5);
  \end{tikzpicture}
  \Rightarrow
  \begin{tikzpicture} [x=13pt,y=13pt,baseline=(current bounding box.center)]
    \draw (0,0) to [out=90,in=90] node {\midarrow} node [above] {$a_1$} (3,0);
    \draw (4,0) to [out=90,in=90] node {\midarrow} node [above] {$a_1b_1^{-1}$} (7,0);
    \draw (-0.5,0) to (0.5,0); \draw (2.5,0) to (4.5,0);
    \draw (6.5,0) to (7.5,0); \draw [white] (0,3.5) to (7,3.5);
    \draw (1.5,0) node {$\cdots$}; \draw (5.5,0) node {$\cdots$};
  \end{tikzpicture}
  \]
  \caption{Actions of arcslides on the fundamental group (overslide).}
  \label{fig:actionoverslides}
\end{figure}
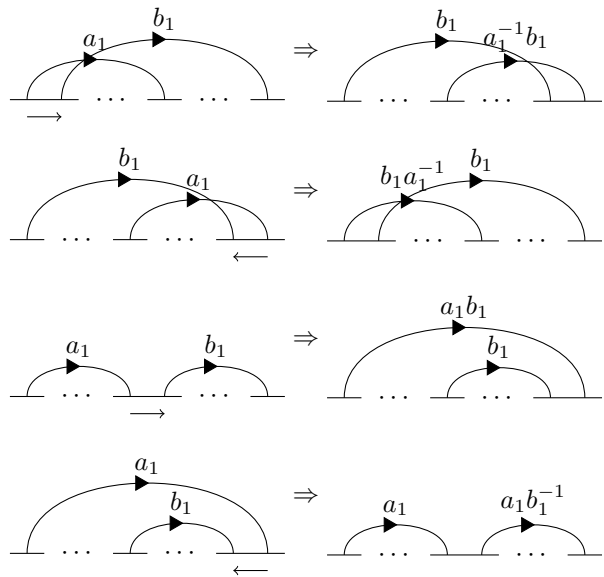

As an example, we verify the decomposition of $\theta_{12}$ into arcslides. Only
the two middle pairs, corresponding to generators $b_2$ and $a_1^{-1}$, are
moved during this sequence of arcslides. We follow what happened to these two
pairs in the following table. We identify pairs of points in a sequence of
arcslides as before. Each line in the table writes the generator corresponding
to pairs identified with the initial $b_2$ and $a_1^{-1}$ pairs in terms of
$\tau_*(\cdot)$ of the initial generators.
\begin{table}[H]
  \centering
  \begin{tabular}{c|c|c}
    Arcslide & $b_2$ & $a_1^{-1}$ \\ \hline
    $4\to 3$ & $b_2$ & $b_2a_1^{-1}$ \\
    $1\to 0$ & $b_2$ & $a_2b_2a_1^{-1}$ \\
    $1\to 2$ & $b_2^{-1}a_2b_2a_1^{-1}$ & $a_2b_2a_1^{-1}$ \\
    $5\to 4$ & $b_2^{-1}a_2b_2a_1^{-1}$ & $a_2b_2a_1^{-1}b_1$ \\
    $6\to 5$ & $b_2^{-1}a_2b_2a_1^{-1}$ & $a_2b_2a_1^{-1}b_1a_1b_2^{-1}a_2^{-1}b_2$ \\
  \end{tabular}
  \label{table:exarcslideaction}
\end{table}
After this sequence of arcslides, the two middle pairs have switched
positions. So the action is $a_1^{-1}\to b_2^{-1}a_2b_2a_1^{-1}$ and $b_2\to
a_2b_2a_1^{-1}b_1a_1b_2^{-1}a_2^{-1}b_2$. The first equation can be rewritten
as $a_1\to a_1b_2^{-1}a_2^{-1}b_2$. This agrees with the fundamental group
action given in Theorem \ref{thm:gensuzuki}.

\begin{proof}[Proof of Lemma \ref{lem:handlebodyinv}]
  By the same argument as in the proof of Lemma \ref{lem:stabinv}, we can show
  that if $\widehatit{CFAD}(\phi)\boxtimes \widehatit{CFD}(\hg) \simeq
  \widehatit{CFD}(\hg)$, then $\widehatit{CFAD}(\mathring{\phi})\boxtimes
  \widehatit{CFD}(\mathbf{H}^{g+1}) \simeq \widehatit{CFD}(\mathbf{H}^{g+1})$.
  where $\mathring{\phi}$ is the element of $\mathrm{MCG}_0(\slz^{g+1})$ that
  fixes the new handle and acts as $\phi$ elsewhere. Here there is again an
  one-to-one correspondence on the generators between
  $\widehatit{CFAD}(\phi)\boxtimes \widehatit{CFD}(\hg)$ and
  $\widehatit{CFAD}(\mathring{\phi})\boxtimes
  \widehatit{CFD}(\mathbf{H}^{g+1})$. There is exactly one domain in the
  adjoined portion that can (and does) contribute an arrow. The evaluation there
  is equivalent to the evaluation of $\widehatit{CFAD}(\iz)\boxtimes
  \widehatit{CFD}(\mathbf{H}^1) \simeq \widehatit{CFD}(\mathbf{H}^1)$ on the
  genus 1 pointed matched circle, giving the arrow in
  $\widehatit{CFD}(\mathbf{H}^{g+1})$ that is inside the adjoined pointed
  matched circle. The remaining domains must be outside the adjoined region,
  showing a one-to-one correspondence between arrows in
  $\widehatit{CFAD}(\phi)\boxtimes \widehatit{CFD}(\hg)$, and the remaining
  arrows in $\widehatit{CFAD}(\mathring{\phi})\boxtimes
  \widehatit{CFD}(\mathbf{H}^{g+1})$. This argument works whether $\slz^{g+1}$
  is formed as $\slz^1\#\slz^g$ or as $\slz^g\#\slz^1$.

  From this, we see that it is sufficient to verify Equation
  (\ref{eq:reparametrization}) for each of the generators of
  $\mathrm{MCG}_0^\beta(F_g)$ in its respective minimum genus (1 or 2) case.

  To do so, we decompose each generator $\phi$ of $\mathrm{MCG}_0^\beta(F_g)$
  (for $g=1$ or $2$ depending on $\phi$) into arcslides
  $\tau_n\circ\cdots\circ\tau_1$, as given above. Then directly compute the left
  side of (\ref{eq:reparametrization}) using the constructions for
  $\hatda(\tau_i)$. This reduces to a finite computation, which we performed on
  a computer using a Python program (which implements the description of
  $\hataa(\iz)$ and the box tensor product). The code for the computation can be
  found at:
  \begin{center}
    \texttt{https://github.com/bzhan/auto2}.
  \end{center}
  The entire computation took less than 20 seconds.

  This concludes the proof of Lemma \ref{lem:handlebodyinv}, and therefore
  Theorem \ref{thm:invclosed}.
\end{proof}

%%% Local Variables:
%%% mode: latex
%%% TeX-master: "dacalc"
%%% End:

\bibliographystyle{plain}
\bibliography{paper}

\begin{thebibliography}{10}

\bibitem{ABP}
Jorgen~Ellegaard Andersen, Alex~James Bene, and R.~C. Penner.
\newblock Groupoid extensions of mapping class representations for bordered
  surfaces.
\newblock {\em Topology and its Applications}, 156(17):2713--2725, 2009.
\newblock arXiv:0710.2651v2.

\bibitem{Bene}
Alex~James Bene.
\newblock A chord diagrammatic presentation of the mapping class group of a
  once bordered surface.
\newblock {\em Geometriae Dedicata}, 144(1):171--190, 2010.
\newblock arXiv:0802.2747.

\bibitem{CK08}
Sabin Cautis and Joel Kamnitzer.
\newblock Knot homology via derived categories of coherent sheaves, {I}: The
  sl(2)-case.
\newblock {\em Duke Math. J.}, 142(3):511--588, 2008.
\newblock arXiv:math/0701194.

\bibitem{FM}
Benson Farb and Dan Margalit.
\newblock {\em A Primer on Mapping Class Groups}.
\newblock PMS-49. Princeton University Press, 2012.

\bibitem{K02}
Mikhail Khovanov.
\newblock A functor-valued invariant of tangles.
\newblock {\em Algebr. Geom. Topol.}, 2:665--741, 2002.
\newblock arXiv:math/0103190.

\bibitem{KS}
Mikhail Khovanov and Paul Seidel.
\newblock Quivers, {F}loer cohomology, and braid group actions.
\newblock {\em J. Amer. Math. Soc.}, 15(1):203--271, 2002.
\newblock arXiv:math/0006056.

\bibitem{KT07}
Mikhail Khovanov and Richard Thomas.
\newblock Braid cobordisms, triangulated categories, and flag varieties.
\newblock {\em Homology Homotopy Appl.}, 9(2):19--94, 2007.
\newblock arXiv:math/0609335v2.

\bibitem{Levine10}
Adam~Simon Levine.
\newblock Knot doubling operators and bordered {H}eegaard {F}loer homology.
\newblock {\em J. Topology}, 5(3):651--712, 2012.
\newblock arXiv:1008.3349.

\bibitem{LOT08}
Robert Lipshitz, Peter~S. Ozsv\'ath, and Dylan~P. Thurston.
\newblock Bordered {H}eegaard {F}loer homology: Invariance and pairing.
\newblock 2008.
\newblock arXiv:0810.0687.

\bibitem{LOT10e}
Robert Lipshitz, Peter~S. Ozsv\'ath, and Dylan~P. Thurston.
\newblock A faithful linear-categorical action of the mapping class group of a
  surface with boundary.
\newblock {\em J. Euro. Math. Soc.}, 15(4):1279--1307, 2013.
\newblock arXiv:1012.1032.

\bibitem{LOT10c}
Robert Lipshitz, Peter~S. Ozsv\'ath, and Dylan~P. Thurston.
\newblock Computing $\widehat{HF}$ by factoring mapping classes.
\newblock {\em Geom. Topol.}, 18(5):2547--2681, 2014.
\newblock arXiv:1010.2550.

\bibitem{LOT10a}
Robert Lipshitz, Peter~S. Ozsv\'ath, and Dylan~P. Thurston.
\newblock Bimodules in bordered {H}eegaard {F}loer homology.
\newblock {\em Geom. Topol.}, 19(2):525--724, 2015.
\newblock arXiv:1003.0598.

\bibitem{OSS}
Peter~S. Ozsv\'ath, Andras~I. Stipsicz, and Zolt\'an Szab\'o.
\newblock Combinatorial {H}eegaard {F}loer homology and nice {H}eegaard
  diagrams.
\newblock {\em Adv. Math.}, 231(1):102--171, 2012.
\newblock arXiv:0912.0830.

\bibitem{OS04b}
Peter~S. Ozsv\'ath and Zolt\'an Szab\'o.
\newblock Holomorphic disks and three-manifold invariants: properties and
  applications.
\newblock {\em Ann. of Math. (2)}, 159(3):1159--1245, 2004.
\newblock arXiv:math.SG/0105202.

\bibitem{OS04a}
Peter~S. Ozsv\'ath and Zolt\'an Szab\'o.
\newblock Holomorphic disks and topological invariants for closed
  three-manifolds.
\newblock {\em Ann. of Math. (2)}, 159(3):1027--1158, 2004.
\newblock arXiv:math.SG/0101206.

\bibitem{SW}
Sucharit Sarkar and Jiajun Wang.
\newblock An algorithm for computing some {H}eegaard {F}loer homologies.
\newblock {\em Ann. of Math. (2)}, 171(2):1213--1236, 2010.
\newblock arXiv:math/0607777.

\bibitem{SS06}
Paul Seidel and Ivan Smith.
\newblock A link invariant from the symplectic geometry of nilpotent slices.
\newblock {\em Duke Math. J.}, 134(3):453--514, 2006.
\newblock arXiv:math/0405089.

\bibitem{ST01}
Paul Seidel and Richard Thomas.
\newblock Braid group actions on derived categories of coherent sheaves.
\newblock {\em Duke Math. J.}, 108(1):37--108, 2001.
\newblock arXiv:math/0001043.

\bibitem{Siegel11}
Kyler Siegel.
\newblock A geometric proof of a faithful linear-categorical surface mapping
  class group action.
\newblock 2011.
\newblock arXiv:1108.3676.

\bibitem{Suzuki}
Shin'ichi Suzuki.
\newblock On homeomorphisms of a 3-dimensional handlebody.
\newblock {\em Canad. J. Math.}, 29(1):111--124, 1977.

\bibitem{BZ1}
Bohua Zhan.
\newblock Explicit {K}oszul-dualizing bimodules in bordered {H}eegaard {F}loer
  homology.
\newblock 2014.
\newblock arXiv:1403.6215.

\end{thebibliography}

\end{document}